\documentclass[12pt,a4paper]{report}
\usepackage{amsmath}
\usepackage{amssymb}
\usepackage{amsfonts}
\usepackage{amsthm}
\usepackage{relsize}
\usepackage{setspace}
\usepackage{geometry}
\usepackage{url}
\usepackage{enumerate}

\input xy
\xyoption{all}

\begin{document}

\newtheorem{thm}{Theorem}[section]
\newtheorem{lem}[thm]{Lemma}
\newtheorem{prop}[thm]{Proposition}
\newtheorem{cor}[thm]{Corollary}
\newtheorem{hyp}[thm]{Hypothesis}
\newtheorem*{propn}{Proposition}
\newtheorem*{corn}{Corollary}

\theoremstyle{definition}
\newtheorem{defn}[thm]{Definition}

\theoremstyle{definition}
\newtheorem{eg}[thm]{Example}

\theoremstyle{definition}
\newtheorem*{egn}{Example}

\theoremstyle{definition}
\newtheorem{que}{Question}

\theoremstyle{definition}
\newtheorem{prob}[que]{Problem}

\theoremstyle{remark}
\newtheorem{rem}{Remark}

\newtheorem*{thma}{Theorem A}
\newtheorem*{thmb}{Theorem B}
\newtheorem*{thmc}{Theorem C}
\newtheorem*{thmd}{Theorem D}
\newtheorem*{thme}{Theorem E}
\newtheorem*{thmf}{Theorem F}
\newtheorem*{thmg}{Theorem G}
\newtheorem*{thmh}{Theorem H}

\newcommand{\Hpi}{\mathrm{Hall}_\pi}
\newcommand{\Hpip}{\mathrm{Hall}_{\pi'}}
\newcommand{\Irr}{\mathrm{Irr}}

\newcommand{\sgp}{{[\leq]}}
\newcommand{\nsgp}{{[\unlhd]}}
\newcommand{\fsgp}{{[\leq_f]}}
\newcommand{\nil}{{[\mathrm{pronil}]}}
\newcommand{\sol}{{[\mathrm{prosol}]}}
\newcommand{\fin}{{[\mathrm{fin}]}}
\newcommand{\fsg}{{[\mathrm{sim}]}}
\newcommand{\elcyc}{{[\mathrm{elcyc}]}}

\newcommand{\ir}{\varrho}

\newcommand{\und}{\underbar}
\newcommand{\Sylp}{\mathrm{Syl}_p}
\newcommand{\piP}{p \in \mathbb{P}}
\newcommand{\la}{\leftarrow}
\newcommand{\IH}{\mathrm{IH}}
\newcommand{\Cf}{\mathrm{Cf}}
\newcommand{\Aut}{\mathrm{Aut}}
\newcommand{\Inn}{\mathrm{Inn}}
\newcommand{\Out}{\mathrm{Out}}
\newcommand{\Fr}{\mathrm{Fr}}
\newcommand{\rO}{\mathrm{O}}
\newcommand{\rH}{\mathrm{H}}
\newcommand{\Comm}{\mathrm{Comm}}
\newcommand{\KComm}{\mathrm{KComm}}
\newcommand{\LComm}{\mathrm{LComm}}
\newcommand{\ICom}{\mathrm{ICom}}
\newcommand{\SNp}{\mathcal{S}}
\newcommand{\e}{\mathrm{e}}
\newcommand{\LFe}{\mathrm{LF-e}}
\newcommand{\db}{\mathrm{db}}
\newcommand{\eb}{\mathrm{eb}}
\newcommand{\Sym}{\mathrm{Sym}}
\newcommand{\Alt}{\mathrm{Alt}}
\newcommand{\ch}{\mathrm{char}}
\newcommand{\PA}{\mathrm{PA}}
\newcommand{\Fin}{\mathrm{Fin}}
\newcommand{\Hom}{\mathrm{Hom}}
\newcommand{\imm}{\mathrm{imm}}
\newcommand{\pro}{\text{pro-}}
\newcommand{\FD}{\mathrm{[FD]}}
\newcommand{\FR}{\mathrm{[FR]}}
\newcommand{\CT}{\mathrm{[CT]}}
\newcommand{\Core}{\mathrm{Core}}
\newcommand{\rQ}{\mathrm{Q}}
\newcommand{\id}{\mathrm{id}}
\newcommand{\lep}{\leq_{[p]}}
\newcommand{\plh}{\mathrm{ht}_{[p]}}
\newcommand{\Cl}{\mathrm{Cl}}
\newcommand{\GL}{\mathrm{GL}}
\newcommand{\SL}{\mathrm{SL}}
\newcommand{\PSL}{\mathrm{PSL}}
\newcommand{\ord}{\mathrm{ord}^\times}
\newcommand{\pd}{\mathrm{pd}}
\newcommand{\Comp}{\mathrm{Comp}}
\newcommand{\St}{\mathrm{St}}
\newcommand{\Emb}{\mathcal{E}_{p'}}
\newcommand{\EmbLF}{\mathcal{E}^{\mathrm{LF}}_{p'}}
\newcommand{\wrd}{\mathrm{wrd}}
\newcommand{\Qd}{\mathrm{Qd}}
\newcommand{\Sp}{\mathrm{Sp}}
\newcommand{\tp}{\mathrm{top}}

\newcommand{\Ob}{\mathrm{Ob}}
\newcommand{\ob}{\mathrm{ob}}
\newcommand{\OI}{\mathrm{OI}}
\newcommand{\Ilhd}{\mathrm{I}^\lhd}

\newcommand{\rN}{(\mathrm{N})}
\newcommand{\rX}{(\mathrm{X})}
\newcommand{\ra}{(\mathrm{a})}
\newcommand{\rh}{(\mathrm{h})}
\newcommand{\rhp}{(\mathrm{h'})}
\newcommand{\ri}{(\infty)}
\newcommand{\rNa}{(\mathrm{Na})}
\newcommand{\rNh}{(\mathrm{Nh})}
\newcommand{\rNi}{(\mathrm{N}\infty)}
\newcommand{\rNip}{(\mathrm{N}\infty')}
\newcommand{\rXa}{(\mathrm{Xa})}
\newcommand{\rXh}{(\mathrm{Xh})}
\newcommand{\rXi}{(\mathrm{X}\infty)}
\newcommand{\rDa}{(\mathrm{Da})}
\newcommand{\rDh}{(\mathrm{Dh})}
\newcommand{\rDi}{(\mathrm{D}\infty)}

\newcommand{\mcA}{\mathcal{A}}
\newcommand{\mcB}{\mathcal{B}}
\newcommand{\mcC}{\mathcal{C}}
\newcommand{\mcD}{\mathcal{D}}
\newcommand{\mcE}{\mathcal{E}}
\newcommand{\mcF}{\mathcal{F}}
\newcommand{\mcG}{\mathcal{G}}
\newcommand{\mcH}{\mathcal{H}}
\newcommand{\mcI}{\mathcal{I}}
\newcommand{\mcJ}{\mathcal{J}}
\newcommand{\mcK}{\mathcal{K}}
\newcommand{\mcL}{\mathcal{L}}
\newcommand{\mcM}{\mathcal{M}}
\newcommand{\mcN}{\mathcal{N}}
\newcommand{\mcO}{\mathcal{O}}
\newcommand{\mcP}{\mathcal{P}}
\newcommand{\mcQ}{\mathcal{Q}}
\newcommand{\mcR}{\mathcal{R}}
\newcommand{\mcS}{\mathcal{S}}
\newcommand{\mcT}{\mathcal{T}}
\newcommand{\mcU}{\mathcal{U}}
\newcommand{\mcV}{\mathcal{V}}
\newcommand{\mcW}{\mathcal{W}}
\newcommand{\mcX}{\mathcal{X}}
\newcommand{\mcY}{\mathcal{Y}}
\newcommand{\mcZ}{\mathcal{Z}}

\newcommand{\bC}{\mathbb{C}}
\newcommand{\bF}{\mathbb{F}}
\newcommand{\bN}{\mathbb{N}}
\newcommand{\bP}{\mathbb{P}}
\newcommand{\bQ}{\mathbb{Q}}
\newcommand{\bR}{\mathbb{R}}
\newcommand{\bZ}{\mathbb{Z}}

\newcommand{\bigast}{\ensuremath{\displaystyle\mathop{\mathlarger{\ast}}}}

\doublespacing

\thispagestyle{empty}

\begin{center}

\

\

\

{\Huge \bf Finiteness Properties\\ of Profinite Groups\\}

\

\

{ \Large Colin David Reid\\

School of Mathematical Sciences\\

Queen Mary, University of London\\}

\

\

{ \Large A thesis submitted to\\

The University of London\\

for the degree of Doctor of Philosophy\\}

\

\

{ \Large \today}

\end{center}

\newpage

\setcounter{page}{2}

\setlength{\parindent}{0pt}

\setlength{\parskip}{3mm}

\onehalfspacing

\

\vspace{0.27\textheight}

\begin{center}

{\bf Declaration}\\

\end{center}

This thesis presents the outcome of research undertaken by myself while enrolled as a student at the School of Mathematical Sciences at Queen Mary, University of London, under the supervision of Robert Wilson.  Except where indicated, everything written here is my own work in my own words.  This work has not been submitted for any degree or award other than that stated to any educational establishment.

\

\

Colin Reid

\today

\newpage

\

\vspace{0.27\textheight}
 
\begin{center}

{\bf Abstract}\\

\end{center}

Broadly speaking, a finiteness property of groups is any generalisation of the property of having finite order.  A large part of infinite group theory is concerned with finiteness properties and the relationships between them.  Profinite groups are an important case of this, being compact topological groups that possess an intimate connection with their finite images. This thesis investigates the relationship between several finiteness properties that a profinite group may have, with consequences for the structure of finite and profinite groups.

\tableofcontents

\chapter*{Preface}

\addcontentsline{toc}{chapter}{\protect\numberline{}Preface}

This thesis concerns several questions in the theory of profinite groups, under the broad heading of `finiteness properties'.  Two questions need to be answered here.

First, what is meant by a profinite group?

\paragraph{Definition}A \emph{profinite group} $G$ is a topological group that is compact and totally disconnected.

The important point here is that we regard the topology of a profinite group $G$ as an inherent part of its definition.  In other words, the ambient category is not the category of groups, but rather the category of topological groups and continuous homomorphisms.  The motivation for this approach is the connection to finite group theory: it is precisely the closed subgroups and continuous homomorphisms that are described by inverse limits of corresponding entities in the category of finite groups.  Furthermore, it is necessary to identify numerical invariants, such as the minimum size of a topological generating set, that have no obvious interpretation for profinite groups as abstract groups.

Second, what does it mean to say that an infinite (topological) group $G$ has a finiteness property?  There are several overlapping interpretations:

\begin{enumerate}
\item There is a connection between the structure of $G$ and the structure of some family of finite groups.  For instance: `$G$ is residually finite'.
\item There is some property of the group $G$, such that every finite group also has this property.  For instance: `$G$ is linear'.
\item There is some numerical invariant $n$, such that $n(G)$ is finite, and the value of $n(G)$ is of interest.  For instance: `$G$ has derived length at most $n$'.
\item Some or all structures of a certain kind derived from $G$ are finite, in a way that is only of interest in an infinite context.  For instance: `every ascending chain of subgroups of $G$ is finite'.\end{enumerate}

All of the interpretations above feature heavily in the theory of profinite groups, especially the first; of all classes of infinite groups, profinite groups have perhaps the deepest connection with finite groups and finite structures.  The overall aim of this thesis is to contribute to the theory of profinite groups in their own right, firstly by drawing direct analogies with established methods finite group theory, and secondly by discussing alternative notions of `smallness' that are particular to the theory of profinite groups.

\section*{Acknowledgements}

First and foremost, I would like to thank Charles Leedham-Green, who has been my \emph{de facto} academic supervisor since early 2008.  It was Charles who introduced me to profinite groups and got me interested in them, and he has been a continuing source of ideas, inspiration and intellectual energy, while providing a great deal of high-level feedback and insight in response to my own efforts.  I am also very grateful to him for his proof-reading of this and other documents, without which the number and severity of errors would have been considerably greater.

My thanks also go to Robert Wilson, my official supervisor, who guided me through the first year and a bit of my doctoral studies and has continued to provide advice and administrative support.  Although the project I worked on with Rob has not found its way into this thesis, I learned a lot in that period about both finite group theory and research skills in general, and Rob's guidance has been very helpful in this regard.  I also wish to thank Robert Johnson, my second supervisor, for his role.

Since starting this project, I have had many productive discussions about profinite groups and related areas with other researchers both inside and outside QMUL who have shared their knowledge and insights.  I will list those who particularly come to mind in alphabetical order: Yiftach Barnea, John Bray, Peter Cameron, Mikhail Ershov, Jonathan Kiehlmann, Benjamin Klopsch, Nikolay Nikolov, Claas R\"{o}ver, Dan Segal and Bert Wehrfritz.

This project would not have been possible without the studentship funded by EPSRC and QMUL and the facilities provided by QMUL.  I would also like to thank the administrative staff at the School of Mathematical Sciences for keeping everything running smoothly.

I would like to thank the staff and students of the School of Mathematical Sciences collectively, for creating an exceptionally welcoming research and social environment, and for organising a fascinating range of study groups and seminars.  The positive atmosphere has had an incalculable effect on my motivation and desire to continue working in mathematics research, and I would strongly recommend the department to anyone wishing to do the same.

Finally, I would like to thank my family and friends for their moral support.  In particular, I would like to thank my parents for always emphasising the importance of education, encouraging me to pursue long-term ambitions, and providing the material support to allow me to focus on them.  Their values have played a decisive role in shaping my own.

Colin Reid

London, \today.

\chapter{Preliminaries}

\section{Definitions and conventions}\label{defprelim}

The purpose of this chapter is not to give new results or even to give an overview of the subject; it is merely to establish some notation and prerequisites for the rest of the thesis.  All results presented in this chapter are drawn directly from or follow easily from existing published literature.

\paragraph{Convention}All groups in this thesis are topological groups, equipped with a natural topology (depending on their construction).  By default, this is the profinite topology in the case of profinite groups, and the discrete topology otherwise.  Subgroups are required to be closed, homomorphisms to be continuous, and generation means topological generation.  When we wish to suppress topological considerations, the word `abstract' will be used, for instance `abstract subgroup'.

Given a topological space or group $X$ and a subset $Y$, write $\overline{Y}$ for the topological closure of $Y$.

Let $\bP$ denote the set of prime numbers, $p$ and $q$ individual primes, $\pi$ a set of primes, and $\pi'$ its complement in $\bP$.  Where there is no ambiguity, $p$ and $p'$ will be used to indicate $\{p\}$ and $\{p\}'$.

Given a prime $p$ and an integer $n$ coprime to $p$, write $\ord(n,p)$ for the multiplicative order of $n$ as an element of $\bF_p$.

Given a group $G$ and an integer $n$, write $G^n$ for the subgroup generated by $n$-th powers of elements of $G$.  Write $G^{(n)}$ for the $n$-th term of the derived series of $G$; in particular, $G' = G^{(1)}$.  Given subgroups $H$ and $K$ of $G$, write $H^K$ for group generated by the $K$-conjugates of $H$, and write $\Core_K(H)$ for the intersection of the $K$-conjugates of $H$.  Write $H \unlhd^2 G$ to mean $H$ is subnormal in $G$ of defect at most $2$, that is, there is some $K$ such that $H \unlhd K \unlhd G$.  Given another group $L$, write $L \lesssim G$ to mean $L$ is isomorphic to a subgroup of $G$ (not necessarily proper).

Given a group $G$, a normal subgroup $K$ and a subgroup $H$ such that $K \leq H \leq G$, say $H$ is the \emph{lift} of $H/K$ to $G$.  If $H$ is normal in $G$, say $G/K$ \emph{covers} $G/H$.

Let $A \cdotp B$ denote any group $G$ such that $A \unlhd G$ and $G/A \cong B$.

With all subset or subgroup relations as applied to topological spaces or groups, a subscript $o$ (for instance $\subset_o$ or $\unlhd_o$) will be used to mean `open' and a subscript $c$ to mean `closed'.

The following classes of group will appear frequently, so receive their own notation:

$[1]$ is the class of trivial groups;

$\fin$ is the class of finite groups;

$\nil$ is the class of pronilpotent groups;

$\sol$ is the class of prosoluble groups;

$\fsg$ is the class of non-abelian finite simple groups.

We define the cardinality $|\mcC|$ of a class $\mcC$ of groups to be the size of a set of representatives for the isomorphism classes in $\mcC$, where such a set exists.  Say $\mcC$ is finite if $|\mcC|$ is finite, and say $\mcC$ is infinite otherwise.

It will also be necessary to use \emph{subgroup classes}: a subgroup class $\mcE$ associates to every group $G$ in a given class a set $\mcE(G)$ of subgroups of $G$.  The following subgroup classes will be especially important:

$\sgp(G)$ is the set of all subgroups of $G$;
 
$\fsgp(G)$ is the set of all subgroups of $G$ of finite index;
 
$\nsgp(G)$ is the set of all normal subgroups of $G$.

Let $\mcX$ be a class of topological groups.  The \emph{$\mcX$-residual} $O^\mcX(G)$ of a topological group $G$ is the intersection of all normal subgroups $N$ such that $G/N$ is an $\mcX$-group.  Say $G$ is \emph{residually-$\mcX$} if $O^\mcX(G)=1$.  In particular, a residually-$\fin$ group is said to be \emph{residually finite}.  The \emph{$\mcX$-radical} $O_\mcX(G)$ is the subgroup generated by all subnormal $\mcX$-subgroups.  Say $G$ is \emph{radically-$\mcX$} if $O_\mcX(G)=G$.  A \emph{radical} of $G$ is a subgroup that is the $\mcX$-radical of $G$ for some class $\mcX$.

Say the profinite group $G$ \emph{involves} the finite group $H$ if there are subgroups $M$ and $N$ of $G$ such that $N \unlhd M$ and $M/N \cong H$.  If $G$ does not involve $H$, say $G$ is \emph{$H$-free}, and if $G$ does not involve $H$ for any $H$ in a class $\mcH$, say $G$ is \emph{$\mcH$-free}.  As a particular case of this, if $\mcH$ is the class of non-abelian finite simple groups of order divisible by $p$, then $\mcH$-free groups are said to be \emph{$p$-separable}.

In what follows, we will often wish to give conditions in terms of invariants of topological groups.  Two basic invariants are as follows:

$d(G)$ is the smallest cardinality of a (topological) generating set for $G$;

$r(G)$ is the \emph{rank} of $G$, which is defined to be the supremum of the number of generators of all closed subgroups of $G$.

We denote more invariants of a profinite group using suffices:

$d_p(G)$ is the number of generators of a $p$-Sylow subgroup of $G$;

given a set of primes $\pi$, we define $d_\pi(G) := \sup_{p \in \pi} d_p(G)$.

\section{The topological structure of profinite groups}\label{topprelim}

We begin with some general observations about compact Hausdorff groups.

\begin{lem}\label{obchain}Let $G$ be a compact Hausdorff group.
\vspace{-12pt}
\begin{enumerate}[(i)]  \itemsep0pt
\item  An abstract subgroup $H$ of $G$ is open if and only if it is closed and of finite index.
\item  Let $O$ be an open neighbourhood of $1$ in $G$.  Let $K_1 > K_2 > \dots$ be a descending chain of closed subgroups of $G$ such that $K_i \not\subseteq O$ for every $i \in I$.  Let $K$ be the intersection of the $K_i$.  Then $K \not\subseteq O$; in particular, $K$ is non-trivial.
\end{enumerate}\end{lem}

\begin{proof}(i) If $H$ is closed and of finite index, then $G \setminus H$ is closed, since it is a union of finitely many right cosets of $H$, so $H$ is open.  Conversely, if $H$ is open, then $G \setminus H$ is open, since it is a union of right cosets of $H$, so $H$ is closed.  Also, the set $\mcH$ of right cosets of $H$ is an open cover of $G$ that cannot be refined, so $|\mcH|=|G:H|$ must be finite.

(ii) Let $F_i = K_i \cap (G \setminus O)$.  Then each $F_i$ is closed and non-empty, and hence the intersection of finitely many $F_i$ is non-empty, since the $F_i$ form a descending chain.  Since $G$ is compact, it follows that the intersection $K \cap (G \setminus O)$ of all the $F_i$ is non-empty.  Hence $K \not\subseteq O$.\end{proof}

\begin{defn}A \emph{homomorphism} of topological groups is a homomorphism of the underlying abstract groups that is also continuous.  An \emph{isomorphism} is a homomorphism possessing a continuous inverse.\end{defn}

In general, a bijective homomorphism of topological groups need not be an isomorphism, as the inverse may not be continuous.  However, this complication does not occur for compact Hausdorff groups:

\begin{prop}Let $G$ and $H$ be compact Hausdorff groups, and let $\theta: G \rightarrow H$ be an abstract homomorphism that is bijective.  Then $\theta$ is an isomorphism of topological groups.\end{prop}

\begin{proof}See \cite{Hof}, Remark 1.8.\end{proof}

For a topological group, the topology is constrained by the algebraic structure, since the topology must be preserved by multiplication and taking inverses.  This is especially true in the case of Hausdorff topological groups.  The following lemma follows easily from the definitions.

\begin{lem}\label{topcentlem}Let $G$ be a Hausdorff topological group, let $n$ be an integer, and let $X$ be any subset.  Then the following subsets of $G$ are closed:
\vspace{-12pt}
\begin{enumerate}[(i)]  \itemsep0pt
 \item $\{ g \in G \mid gx=xg \; \forall x \in X\}$;
 \item $\{ g \in G \mid g^n=1\}$.
\end{enumerate}\end{lem}

\begin{defn}Let $G$ be a topological group.  Define the \emph{profinite completion} $\hat{G}$ of $G$ to be the inverse limit of the inverse system formed by the finite continuous images of $G$.\end{defn}
  
Note that if $G$ is residually finite as a topological group, then $G$ may be regarded as an abstract subgroup of $\hat{G}$.  If $G$ is already profinite then $G = \hat{G}$.

\begin{prop}Let $G$ be a profinite group and let $X$ be an abstract subset of $G$.  Then 
\[ \overline{X} = \bigcap_{N \unlhd_o G} XN.\]
In particular, if $X$ is an abstract subgroup of $G$, then $\overline{X}$ is the intersection of all open subgroups of $G$ that contain $X$.\end{prop}

\begin{defn}Let $G$ be a profinite group, and let $\kappa$ be a cardinal.  Say $G$ is \emph{$\kappa$-based} if $G$ has $\kappa$ open subgroups.  A \emph{countably based} profinite group is one that is either finite or $\aleph_0$-based.\end{defn}

\begin{thm}Let $G$ be a $\kappa$-based profinite group, where $\kappa$ is an infinite cardinal.  Then $|G| = 2^\kappa$.\end{thm}

\begin{proof}See \cite{Tom}, Theorem 4.9.\end{proof}

\begin{cor}Let $G$ be a $\kappa$-based profinite group, where $\kappa$ is an infinite cardinal, and let $H$ be a closed subgroup of $G$.  Then $H$ is a $\lambda$-based profinite group for some $\lambda \leq \kappa$.\end{cor}

For finitely generated profinite groups, the underlying abstract group determines the topology, thanks to the Nikolov-Segal theorem:

\begin{thm}[Nikolov, Segal \cite{NS}]Let $G$ be a finitely generated profinite group.  Then every abstract subgroup of $G$ of finite index is in fact an open subgroup.\end{thm}

We will also make use of the Schreier index formula, as applied to pro-$p$ groups.

\begin{thm}[Schreier index formula for pro-$p$ groups] Let $G$ be a finitely generated pro-$p$ group, and let $H$ be an open subgroup.  Then
\[ d(H) \leq |G:H|(d(G)-1) + 1.\]\end{thm}

\section{Basic Sylow theory of profinite groups}\label{sylprelim}

\begin{defn}A \emph{supernatural number} is a formal product $\prod_{\piP} p^{n_p}$ of prime powers, where each $n_p$ is either a non-negative integer or $\infty$.\end{defn}

Multiplication of supernatural numbers is defined in the obvious manner, giving rise to a semigroup; note that any set of supernatural numbers has a supernatural least common multiple.  Also, by unique factorisation, the semigroup of supernatural numbers contains a copy of the multiplicative semigroup of natural numbers, which may be regarded as the \emph{finite} supernatural numbers.  A \emph{$\pi$-number} is a supernatural number $\prod_{\piP} p^{n_p}$ for which $n_p=0$ for all $p$ in $\pi'$.  Given a supernatural number $x = \prod_{\piP} p^{n_p}$, its \emph{$\pi$-part} $x_\pi$ is $\prod_{p \in \pi}p^{n_p}$.

\begin{defn}\label{syldef}Let $G$ be a profinite group and let $H$ be a subgroup.  Define the \emph{index} $|G:H|$ of $H$ in $G$ to be the least common multiple of $|G/N:HN/N|$ as $N$ ranges over all open normal subgroups of $G$, and the \emph{order} of $G$ to be $|G:1|$; write $|G|_\pi$ for $|G:1|_\pi$.  In particular, the order of a profinite group is a $\pi$-number if and only if the group is pro-$\pi$.  (Note that in contrast to finite group theory, the supernatural order of a profinite group is not determined by the cardinality of the underlying set.)  Say $H$ is a \emph{$\pi$-Hall subgroup} of $G$ if $H$ is a pro-$\pi$ group, and $|G:H|$ is a $\pi'$-number.  We also refer to $\{p\}$-Hall subgroups as \emph{$p$-Sylow subgroups}; write $\Sylp(G)$ for the set of Sylow subgroups of $G$.  Given an element $x$ of $G$, define the order of $x$ to be $|\langle x \rangle : 1|$.\end{defn}

\begin{thm}Let $G$ be a profinite group, and let $H$ be a subgroup.  Then $|G:1|=|G:H||H:1|$.  If $H$ is normal then $|G:H|=|G/H:H/H|$.\end{thm}

\begin{proof}See \cite{Wil}.\end{proof}

The foundational result of Sylow theory in finite groups is Sylow's theorem, and in finite soluble groups this generalises to Hall's theorem.  These theorems generalise to profinite and prosoluble groups respectively, via an inverse limit argument.  Proofs for both can be found in \cite{Wil}.

\begin{thm}[Sylow's theorem for profinite groups]
Let $G$ be a profinite group, and let $p$ be a prime.
\vspace{-12pt}
\begin{enumerate}[(i)]  \itemsep0pt
\item $G$ has a $p$-Sylow subgroup.
\item Any two $p$-Sylow subgroups of $G$ are conjugate.
\item Every pro-$p$ subgroup of $G$ is contained in some $p$-Sylow subgroup.\end{enumerate}
\end{thm}

\begin{thm}[Hall's theorem for profinite groups]
Let $G$ be a prosoluble group, and let $\pi$ be a set of primes.
\vspace{-12pt}
\begin{enumerate}[(i)]  \itemsep0pt
\item $G$ has a $\pi$-Hall subgroup.
\item Any two $\pi$-Hall subgroups of $G$ are conjugate.
\item Every pro-$\pi$ subgroup of $G$ is contained in some $\pi$-Hall subgroup.\end{enumerate}
\end{thm}

Here are some consequences.

\begin{cor}\label{sylsubnor}Let $G$ be a profinite group.  If $G$ is prosoluble, let $\pi$ be an arbitrary set of primes; otherwise, let $\pi$ consist of a single prime.  Let $H$ be a $\pi$-Hall subgroup of $G$, and let $K$ be a subnormal subgroup of $G$.  Then $H \cap K$ is a $\pi$-Hall subgroup of $K$.\end{cor}

\begin{proof}By induction, it suffices to prove the result for $K \unlhd G$.  Now $H \cap K$ is a pro-$\pi$ group, so  by the theorem, there is some $\pi$-Hall subgroup $L$ of $K$ that contains $H \cap K$.  In turn, $L$ is contained in a conjugate $H^{g^{-1}}$ say of $H$.  But then $L^g$ is a $\pi$-Hall subgroup of $K^g$; also, $L^g \leq H$.  Since $K^g = K$, it follows that $H \cap K$ contains a $\pi$-Hall subgroup of $K$, and hence is a $\pi$-Hall subgroup of $K$ by the maximality property of Hall subgroups.\end{proof}

\begin{cor}\label{indexcor}Let $G$ be a prosoluble group, and let $H$ be a subgroup of $G$.  Let $\mcS$ be a set of subsets of the prime numbers such that $\bigcup \mcS = \mu$, and suppose $H$ contains a $\pi$-Hall subgroup of $G$ for every $\pi \in \mcS$.  Then $H$ contains a $\mu$-Hall subgroup of $G$.\end{cor} 

\begin{defn}The \emph{$\pi$-core} $O_\pi(G)$ of $G$ is the group generated by all subnormal pro-$\pi$ subgroups of $G$.  Say $G$ is \emph{$\pi$-normal} if $G$ has a normal $\pi$-Hall subgroup.  The \emph{pro-Fitting subgroup} $F(G)$ is the group generated by all subnormal pro-$p$ subgroups of $G$, over all $\piP$.\end{defn}

\begin{lem}\label{opinor}Let $G$ be a profinite group, and let $\pi$ be a set of primes.  Given $K \unlhd_o G$, let $R_K$ be such that $R_K/K = O_\pi(G/K)$, and let $R = \bigcap_{K \unlhd_o G} R_K$.  Then $O_\pi(G) = R$, and $O_\pi(G)$ is a pro-$\pi$ group.\end{lem}

\begin{proof}We assume the finite case of the lemma, as it is well-known.

Let $O=O_\pi(G)$.  By their construction, $O$ and $R$ are characteristic in $G$, and $R$ is a pro-$\pi$ group by the finite case of the lemma, so $R \leq O$.  For every $K \unlhd_o G$, $OK/K$ is generated by subnormal $\pi$-subgroups of $G/N$, so it is contained in $R_K/K$.  Hence $O \leq R$, and so $O=R$; in particular, $O$ is a pro-$\pi$ group.\end{proof}

The class of pronilpotent groups can be characterised in terms of its Sylow structure in a similar manner to the class of finite nilpotent groups.

\begin{lem}\label{fitlem}A profinite group $G$ is pronilpotent if and only if it is the Cartesian product of its Sylow subgroups, or equivalently, if and only if every Sylow subgroup is normal.  In particular, $F(G)$ is pronilpotent, and contains all pronilpotent normal subgroups of $G$.\end{lem}

\begin{proof}See \cite{Wil}.\end{proof}

The following result will be of consequence later when we consider the action of a profinite group on its pro-Fitting subgroup.  Note that the automorphism group of a finitely generated pro-$p$ group is equipped with a natural profinite topology, by declaring the centraliser of any finite characteristic image to be open.

\begin{thm}\label{cinvthm}Let $P$ be a finitely generated pro-$p$ group, and let $H$ be a closed subgroup of $\Aut(P)$.
\vspace{-12pt}
\begin{enumerate}[(i)]  \itemsep0pt
\item Suppose there is an $H$-invariant normal series
\[ P = P_1 \rhd P_2 \rhd \dots \]
for $P$, such that $\bigcap P_i = 1$, and such that $H$ acts trivially on $P_i/P_{i+1}$ for each $i$.  Then $H$ is a pro-$p$ group.
\item Define the characteristic series $P_i$ by $P_1 = P$, and thereafter $P_{i+1} = [P,P_i]P^p_i$.  Suppose $H$ acts trivially on $P/\Phi(P)$.  Then $H$ acts trivially on $P_i/P_{i+1}$ for all $i$.  In particular, $H$ is a pro-$p$ group.
\item Suppose $P$ is finite and abelian, and $H$ is a $p'$-group.  Then $P = [P,H] \times C_P(H)$.\end{enumerate}\end{thm}

\begin{proof}For part (i) see \cite{DDMS}, for part (ii) see \cite{Lee}, and for part (iii) see \cite{Doe}.\end{proof}

\begin{defn}The \emph{Frattini subgroup} $\Phi(G)$ of a profinite group $G$ is the intersection of all maximal open subgroups of $G$.\end{defn}

\begin{lem}\label{fratlem}\
\vspace{-12pt}
\begin{enumerate}[(i)]  \itemsep0pt
\item Let $G$ be a profinite group, and let $K$ be a normal subgroup of $G$ containing $\Phi(G)$.  Then $K$ is pronilpotent if and only if $K/\Phi(G)$ is pronilpotent.  In particular, $\Phi(G)$ is pronilpotent.
\item Let $G$ be a profinite group.  If $X$ is a set of elements of $G$, then $X$ generates $G$ if and only if the image of $X$ in $G/\Phi(G)$ generates $G/\Phi(G)$.
\item Let $S$ be a pro-$p$ group.  Then $S/\Phi(S)$ is the largest elementary abelian image of $S$, and $d(S) = d(S/\Phi(S))$.\end{enumerate}\end{lem}

\begin{proof}See \cite{Wil}.\end{proof}

\section{Quasisimple groups in profinite group theory}\label{qsprelim}

Since profinite groups are residually finite, any simple profinite group is automatically finite.  The finite simple groups thus play an important role in profinite group theory, and aspects of the Classification of Finite Simple Groups will be invoked at several points in this thesis.

\begin{thm}[Classification of finite simple groups]
Let $G \in \fsg$.  Then $G$ is one of the following:
\vspace{-12pt}
\begin{enumerate}[(i)]  \itemsep0pt
\item An alternating group $Alt(n)$, with $n \geq 5$;
\item A finite simple group of Lie type, or the Tits group;
\item One of 26 sporadic simple groups that do not appear in (i) or (ii).\end{enumerate}\end{thm}

(For a more detailed statement of the Classification, see \cite{ATL}.)

We will need to use some properties of the orders of finite simple groups that can be deduced from the full Classification.

\begin{thm}\label{simordthm}Let $G \in \fsg$.  Then at least one of $6$ and $10$ divides $|G|$.\end{thm}

\begin{rem}The theorem above incorporates the theorem of Feit and Thompson (\cite{Fei}) that any $G \in \fsg$ has even order.\end{rem}

\begin{defn}Given an integer $n$, define $\bP_n$ to be the set of primes at most $n$, and $\bP'_n$ the set of primes greater than $n$.\end{defn}

\begin{thm}\label{lambdafsg}For each $n$, there are only finitely many isomorphism types of finite simple $\bP_n$-groups.\end{thm}

\begin{proof}See for instance \cite{Ale}.\end{proof}

It is often useful to consider a generalisation of the finite simple groups:

\begin{defn}A (pro-)finite group $Q$ is said to be \emph{quasisimple} if $Q$ is perfect and $Q/Z(Q)$ is simple.\end{defn}

\begin{thm}\label{schurmult}Let $G$ be a finite perfect $\pi$-group, for some set of primes $\pi$.  Then there is a finite $\pi$-group $\Gamma$, unique up to isomorphism, such that $\Gamma$ is a perfect central extension of $G$, and any finite perfect central extension of $G$ is an image of $\Gamma$.  In particular, the order of any finite perfect central extension of $G$ is at most $|\Gamma|$.\end{thm}

\begin{proof}See \cite{Suz}.\end{proof}

\begin{cor}\label{qscor}Let $G$ be a quasisimple profinite group.  Then $G$ is finite, and every prime dividing $|G|$ also divides $|G/Z(G)|$.\end{cor}

The outer automorphism groups of finite quasisimple groups are well-known.  We note here some salient features.

\begin{prop}\label{qsout}Let $Q$ be a finite quasisimple group.  Then:
\vspace{-12pt}
\begin{enumerate}[(i)]  \itemsep0pt
\item $\Aut(Q)$ acts faithfully on $Q/Z(Q)$;
\item $\Out(Q)$ is soluble of derived length at most $3$;
\item any abelian subgroup of $\Out(Q)$ has rank at most $4$;
\item $|\Out(Q)|$ is bounded by a function of $r(Q/Z(Q))$.\end{enumerate}\end{prop}

\begin{defn}Let $G$ be a finite quasisimple group.  If $G/Z(G)$ is isomorphic to a finite simple group of Lie type, define $\deg(G)$ to be the Lie rank of $G/Z(G)$.  (If $G/Z(G)$ arises as a group of Lie type in multiple ways, define $\deg(G)$ to be the largest Lie rank that occurs.)  Otherwise, define $\deg(G)$ to be the smallest degree of a faithful permutation action of $G/Z(G)$.\end{defn}

\begin{thm}\label{qsexsec}Let $G$ be a finite quasisimple group, and let $K$ be a finite group.  Suppose that $G$ is $K$-free.  Then $\deg(G)$ is bounded by a function of $K$.\end{thm}

\begin{proof}See for instance \cite{BCP}.\end{proof}

The significance of quasisimple groups for our purposes is that they feature in the generalised pro-Fitting subgroup of a finite group.  This concept can be generalised to profinite groups, as defined below; some consequences of this definition will be obtained in Chapter 4.

\begin{defn}Let $G$ be a profinite group.  A \emph{component} of $G$ is a subnormal subgroup that is quasisimple.  Write $\Comp(G)$ for the set of components of $G$.  Given a set of primes $\pi$, let $\Comp_\pi(G)$ be the set of those components $Q$ of $G$ such that $p$ divides $|Q|$ for every $p \in \pi$.  For any $\pi$, the set $\Comp_\pi(G)$ admits a natural action of $G$ induced by conjugation.  The \emph{layer} of $G$ is $E(G) := \langle \Comp(G) \rangle$; write $E_\pi(G) = \langle \Comp_\pi(G) \rangle$.  Say $G$ is \emph{layer-free} if $E(G)=1$.  Define also $E^*_\pi(G)$ to be the lift of $E_\pi(G/O_\pi(G))$ to $G$.

The \emph{generalised pro-Fitting subgroup} $F^*(G)$ of a profinite group is the group generated by $E(G)$ and $F(G)$.\end{defn}

\begin{thm}\label{finopil}Let $G$ be a finite group.  Then:
\vspace{-12pt}
\begin{enumerate}[(i)]  \itemsep0pt
\item any two distinct components of $G$ commute;
\item $[E(G),F(G)]=1$;
\item $C_G(F^*(G)) \leq F^*(G)$, so in particular $F^*(G) = 1$ if and only if $G=1$.\end{enumerate}\end{thm}

\begin{proof}See \cite{Suz}.\end{proof}

\section{Properties of linear groups}\label{linprelim}

\begin{defn}  A group $G$ is \emph{linear of degree $n$} if $G$ embeds as an abstract subgroup of $\GL(n,F)$ for some field $F$.  Write $\GL(n,p^e)$ for $\GL(\bF^n_{p^e})$.  Let $\mcL(n,\pi)$ denote the class of finite groups that are isomorphic to a subgroup of $\GL(n,p^e)$, for some integer $e$ and $p \in \pi$.

Given a profinite group $G$, define $O^{(n,\pi)}(G) := O^{\mcL(n,\pi)}(G)$.  Define $O^{(n,p)^*}(G)$ to be the intersection of all open subgroups $N$ such that $G/N \in \mcL(n,p)$ and $G/N$ is a $p'$-group.  Define $O^{(n,\pi)^*}(G) := \bigcap_{p \in \pi} O^{(n,p)^*}(G)$.\end{defn}

In this section, we will consider some properties of soluble linear groups of degree $n$; these will have consequences later for profinite groups $G$ such that $O^{(n,\bP)}(G)=1$, or such $O^{(n,p)^*}(G)=1$ for some $p$.

\begin{thm}[Zassenhaus \cite{Zas}; Newman \cite{New}]
Let $G$ be a soluble linear group of degree $n$.  Then $G^{(\db(n))}=1$, where $\db(n)$ is the smallest integer exceeding
\[ 5\log_9(\max(58,n-2)) + 10 - 15(\log 2)(2 \log 3)^{-1}.\]\end{thm}

\begin{defn}Given a finite-dimensional vector space $V$, a subgroup of the general linear group $\GL(V)$ of $V$ is \emph{triangularisable} if it is conjugate to a subgroup of the group $Tr(V)$ of upper-triangular matrices with respect to some basis.\end{defn}

The following lemma is well-known:

\begin{lem}\label{linsylowlem}Let $G = \GL(V)$, where $V$ is an $n$-dimensional vector space over a finite field of characteristic $p$.  Then $Tr(V) = U(V) \rtimes D(V)$, where $D(V)$ is the group of diagonal matrices, and $U(V)$ is the group of upper unitriangular matrices, both with respect to the same basis as for $Tr(V)$.  Furthermore, $U(V)$ is a $p$-Sylow subgroup of $G$, and has nilpotency class $n-1$.\end{lem}

\begin{thm}[Mal'cev \cite{Mal}]Let $G$ be a soluble linear group of degree $n$ over an algebraically closed field.  Then there is a function $\eb(n)$ depending on $n$ alone such that $G$ has a triangularisable normal subgroup $T$ of index dividing $\eb(n)$.\end{thm}

\begin{cor}\label{malcor}There are functions $\eb(n)$ and $\db(n)$ depending on $n$ alone, such that if $\pi$ is a set of primes, and $G$ is a prosoluble group such that $O^{(n,\pi)}(G)=1$, then $G^{(\db(n))}=1$ and $(G^{\eb(n)})'$ is both pronilpotent and pro-$\pi$, and moreover $(G^{\eb(n)})'=1$ whenever $O^{(n,\pi)^*}(G)=1$.\end{cor}

\begin{proof}Let $H$ be an image of $G$ such that $H \leq \GL(n,p^e)$ for some $p \in \pi$ and some $e$.  By Zassenhaus's theorem, $H^{(\db(n))}=1$.  By Mal'cev's theorem, $H$ has a triangularisable normal subgroup of index dividing $\eb(n)$ over the algebraic closure of $\bF_p$, and hence over $\bF_{p^{e'}}$ for some $e'$, since $H$ is finite.  Hence $H^{\eb(n)}$ is triangularisable over $\bF_{p^{e'}}$; hence $(H^{\eb(n)})'$ is a (nilpotent) $p$-group by Lemma \ref{linsylowlem}.  The result now follows from the definitions of $O^{(n,\pi)}(G)$ and $O^{(n,\pi)^*}(G)$.\end{proof}

\section{Control of $p$-transfer in profinite groups}\label{ptransprelim}

An important notion in finite group theory is the \emph{transfer map}, which is a homomorphism that is defined from a finite group to any of its abelian sections.   We will not be using the transfer map directly, but we will be using the closely related notion of control of transfer, and more precisely control of $p$-transfer.  Control of transfer is a concept that generalises easily to profinite groups; see for instance \cite{GRS}.

\begin{defn}Let $G$ be a profinite group, let $H$ be a subgroup, and let $H \leq K \leq G$.  Say $K$ \emph{controls transfer} from $G$ to $H$ if $G' \cap H = K' \cap H$.  If in addition $H$ is a $p$-Sylow subgroup of $G$, then say $K$ \emph{controls $p$-transfer} in $G$.\end{defn}

The following theorem was first proved by Tate (see \cite{Tat}); however, for our purposes we will use a more recent interpretation due to Gagola and Isaacs (\cite{Gag}), in terms of control of $p$-transfer.  Both \cite{Tat} and \cite{Gag} state the result for finite groups, but the generalisation to profinite groups is immediate.

\begin{thm}[Tate]\label{tate}Let $G$ be a profinite group, let $S$ be a $p$-Sylow subgroup of $G$, and let $S \leq K \leq G$.  The following are equivalent:
\vspace{-12pt}
\begin{enumerate}[(i)]  \itemsep0pt
\item $G' \cap S = K' \cap S$;
\item $(G'G^p) \cap S = (K'K^p) \cap S$;
\item $(G'O^p(G)) \cap S = (K'O^p(K)) \cap S$;
\item $O^p(G) \cap S = O^p(K) \cap S$.\end{enumerate}\end{thm}

From now on, the statement `$K$ controls $p$-transfer in $G$' will be taken to mean any of the four equations above interchangeably.

As a simple corollary, there is a connection between control of $p$-transfer and $p'$-normality:

\begin{cor}\label{ptranscomp}Let $G$ be a profinite group, and let $S$ be a $p$-Sylow subgroup of $G$.  Then $S$ itself controls $p$-transfer in $G$ if and only if $G$ is $p'$-normal.\end{cor}

\begin{proof}If $G$ has a normal $p'$-Hall subgroup $H$, then evidently $H = O^p(G)$ and $H \cap S = O^p(S) \cap S = 1$, so $S$ controls $p$-transfer in $G$.  Conversely, if $S$ controls $p$-transfer in $G$, then $O^p(G) \cap S = O^p(S) \cap S = 1$ by Tate's theorem, so $O^p(G)$ is a normal $p'$-Hall subgroup of $G$.\end{proof}

More consequences of Tate's theorem will be discussed later.

\section{Sylow subgroups of certain families of finite groups}\label{sylsimprelim}

Let $p$ be a prime, and let $n$ be a positive integer.  Write $\Sym(n;p)$ for a $p$-Sylow subgroup of $\Sym(n)$, and write $C_n$ for the cyclic group of order $n$.  The groups $\Sym(n;p)$ are well-known:

\begin{prop}\label{symsylow}Let $p$ be a prime, and let $n$ be an integer.
\vspace{-12pt}
\begin{enumerate}[(i)]  \itemsep0pt
\item If $n$ is a power of $p$, then $\Sym(n;p)$ is given by $\Sym(1;p)=1$ and $\Sym(p^k;p) \cong \Sym(p^{k-1};p) \wr C_p$ for $k > 0$.

\item Suppose $n = a_0 + a_1p + \dots + a_kp^k$, where $0 \leq a_i < p$ for $0 \leq i \leq k$.  Then $\Sym(n;p)$ is a direct product of groups $\Sym(p^i;p)$, such that the factor $\Sym(p^i;p)$ appears $a_i$ times.\end{enumerate}\end{prop}

The $p$-Sylow subgroups of the classical groups in characteristic coprime to $p$ were constructed by Weir (\cite{Wei}) for $p$ odd, and by Carter and Fong (\cite{Car}) for $p=2$.  For the purposes of this thesis, we do not need a detailed description of the Sylow subgroups of classical groups; the proposition given below will suffice.

\begin{prop}\label{classicwr}Let $p$ be a prime, and let $q$ be a prime power coprime to $p$.
\vspace{-12pt}
\begin{enumerate}[(i)]  \itemsep0pt
\item Let $n$ be any positive integer.  Suppose $q$ is odd, and let $G$ be one of the following:
\[ \GL(n,q), \Sp(2n,q), O(2n+1,q), O^+(2n,q), O^-(2n,q).\]
Suppose a $p$-Sylow subgroup of $G$ acts irreducibly.  Then $\ord(q,p)$ is even.
\item Let $r$ be a positive integer, and let $n$ be an integer such that $p^{r+1} \leq n < p^{r+2}$.  Let $G$ be one of the following:
\[ \GL(n,q), \Sp(2n,q), U(n,q^2), O(2n+1,q), O^+(2n,q), O^-(2n,q).\]
Let $S$ be a $p$-Sylow subgroup of $G$.  Then $S$ has a quotient isomorphic to $\Sym(p^r;p)$.\end{enumerate}\end{prop}

\begin{proof}(i) See Table 1 of \cite{Sta}.  The Sylow subgroups of `type B' in this table are necessarily reducible.
 
(ii) See \cite{Wei} for the case of $q$ odd, and \cite{Car} for the case of $q$ even.\end{proof}

Given a finite group $G$ of classical Lie type defined over a field of characteristic $p$, the $p$-Sylow subgroups are the maximal unipotent subgroups of $G$.  Some of their properties were described by in \cite{Che} and \cite{Ree}.  From these descriptions, we can deduce the following:

\begin{lem}\label{classicpwr}Let $p$ be a prime, and let $q=p^s$.  Let $n$ be an integer, and let $G$ be one of the following:
\[ \GL(n+1,q), \Sp(2n,q), U(n,q^2), O(2n+1,q), O^+(2n,q), O^-(2n,q).\]
Let $S$ be a $p$-Sylow subgroup of $G$.  Then $d(S)=rs$, where $r$ is the number of simple roots of $G$.\end{lem}

We conclude with the following observation concerning Sylow subgroups of finite simple groups, which is useful for asymptotic results:

\begin{cor}\label{dpdeg}Let $p$ be a prime and let $d$ be an integer.  Let $G \in \fsg$ such that $d_p(G) \leq d$.  Then $\deg(G)$ is bounded by a function of $d$ and $p$.\end{cor}

\begin{proof}If $G$ is of exceptional Lie type or sporadic, then $\deg(G)$ is automatically bounded.  If $G$ is alternating, $\deg(G)$ is bounded by Proposition \ref{symsylow}.  If $G$ is of classical Lie type of characteristic $p$, a bound follows from Lemma \ref{classicpwr}, and if $G$ is of classical Lie type of another characteristic, a bound follows from Proposition \ref{classicwr}.\end{proof}

\chapter{Miscellaneous finiteness properties of profinite groups}

\section{Finiteness conditions in subgroup lattices}\label{sgplat}

One method for studying groups is through the lattice of subgroups; indeed, many statements in group theory can be expressed in terms of containments between subgroups, without reference to individual elements.  In line with our topological convention, it is natural to focus attention on the lattice of closed subgroups.  Consider for instance the following questions, where $G$ is a finitely generated group:

\begin{que}\label{ascq}Suppose there is an ascending chain $\mcH$ of subgroups of $G$, such that the union of $\mcH$ is dense in $G$.  Is $G$ necessarily an element of $\mcH$?\end{que}

\begin{que}\label{infindq}Let $H$ be a subgroup of infinite index.  Is $H$ necessarily contained in a subgroup of $G$ that is maximal subject to having infinite index?\end{que}

\begin{que}\label{infchq}Let $\mcK$ be an infinite set of subgroups of $G$, all of finite index, such that $K \in \mcK$ and $K \leq L \leq G$ implies $L \in \mcK$.  Does $\mcK$ necessarily contain an infinite descending chain?\end{que}

For Questions \ref{ascq} and \ref{infindq}, it is easy to see that the answer is `yes' if $G$ is a discrete group, but `no' if $G$ is a profinite group.  For Question \ref{infchq}, the answer is clearly `no' even if $G = \bZ$.  However, the answer is `yes' to all questions if $G$ is a pro-$p$ group.  The property of being topologically finitely generated does not seem to be the right notion of `smallness' for these questions.  Instead, we need to consider conditions on sublattices of the lattice of closed subgroups, particularly those sublattices generated by a given subgroup class.

In the definitions that follow, $G$ will be a topological group, $\mcX$ a set of subgroups of $G$ such that $G \in \mcX$, and $\mcE$ a subgroup class such that $G \in \mcE(G)$.

\begin{defn}Say $\mcX$ is \emph{maximal-closed} in $G$ if every member of $\mcX \setminus \{G\}$ is contained in a maximal member of $\mcX \setminus \{G\}$.

Define the \emph{Frattini group} $\Phi(\mcX)$ of $\mcX$ to be the intersection of all maximal members of $\mcX \setminus \{G\}$.  Define the \emph{$\mcE$-Frattini subgroup} $\Phi^\mcE(G)$ of $G$ by $\Phi^\mcE(G) := \Phi(\mcE(G))$.  In particular, this definition produces the \emph{Frattini subgroup} $\Phi(G):=\Phi^\sgp(G)$, the \emph{finite index Frattini subgroup} $\Phi^f(G):=\Phi^\fsgp(G)$ and the \emph{normal Frattini subgroup} $\Phi^\lhd(G):=\Phi^\nsgp(G)$ of $G$.\end{defn}

The following is a familiar property of the Frattini subgroup that also applies in this more general context:

\begin{lem}\label{fratmax}Let $G$ be a group, let $\mcX$ be a maximal-closed set of subgroups of $G$, and let $H \in \mcX \setminus \{G\}$.  Then $\langle H, \Phi(\mcX) \rangle \not= G$.\end{lem}

\begin{proof}Since $\mcX$ is maximal-closed, $H$ is contained in a maximal element $M$ of $\mcX \setminus \{G\}$; hence $\langle H, \Phi(\mcX) \rangle \leq M < G$.\end{proof}

\begin{defn}Say $G$ is \emph{$\mcX$-finite} if $\mcX$ is maximal-closed and $|G:\Phi(\mcX)|$ is finite.  Say $G$ is \emph{$\mcE$-finite}, and write $G \in \mcE_\Phi$, if $G$ is $\mcE(G)$-finite.  Say $G$ is \emph{hereditarily $\mcE$-finite}, and write $G \in \mcE^*_\Phi$, if $\fsgp(G) \subseteq \mcE_\Phi$, that is, every subgroup $H$ of $G$ of finite index is $\mcE$-finite in its own right.\end{defn}

There is an easy alternative characterisation of the situation in which $G$ is $\mcX$-finite:

\begin{lem}Let $G$ be a group and let $\mcX$ be a set of subgroups of $G$.  Let $\mcM$ be the set of maximal elements of $\mcX \setminus \{G\}$.  Then $G$ is $\mcX$-finite if and only if the following condition is satisfied:

$(*)$ $\mcM$ is finite, and for every $K \in \mcX \setminus \{G\}$, there is some $M \in \mcX \setminus \{G\}$ such that $K \leq M$ and $|G:M|$ is finite.\end{lem}

\begin{proof}It is clear from the definition of $\Phi(\mcX)$ that $|G:\Phi(\mcX)|$ is finite if and only if $\mcM$ is finite and every element of $\mcM$ has finite index in $H$.  Furthermore, $\mcX$ is maximal-closed if and only if every element of $\mcX \setminus \{G\}$ is contained in some $M \in \mcM$.  Hence condition $(*)$ is necessary for $G$ to be $\mcX$-finite.

Conversely, assume $(*)$ holds.  Given $K \in \mcX \setminus \{G\}$, choose some $M(K) \geq K$ such that $M(K) \in \mcX \setminus \{G\}$ and $|G:M(K)|$ is finite.  Then the set $\mcK$ of elements of $\mcX \setminus \{G\}$ that contain $M(K)$ is finite and nonempty, so there is a maximal element $L$ of $\mcK$; then $L \in \mcM$ and $L <_f G$.  If $K \in \mcM$, then $K=M(K)<_f G$; since $\mcM$ is finite, it follows that $|G:\Phi(\mcX)|$ is finite.\end{proof}

\begin{defn}A \emph{chain} is any totally ordered set.  (Generally, we will be interested in the chains contained in a more general partially ordered set.)  Say the chain $\mcK$ is \emph{ascending} if the order is a well-ordering, and say $\mcK$ is \emph{descending} if $\mcK$ is well-ordered under the reverse ordering.  Note that finite chains are both ascending and descending, but infinite chains cannot be both at once.

Say $\mcX$ is \emph{chain-closed} if, given an ascending chain $\mcC$ in $\mcX$, the closure of the union of $\mcC$ is an element of $\mcX$.  Note that by Zorn's lemma, if $\mcX \setminus \{G\}$ is chain-closed, then $\mcX$ is maximal-closed.  Say $\mcX$ is \emph{max} if it has no infinite ascending chains, and say $G$ is \emph{max-$\mcE$} if $\mcE(G)$ is max.

Say a set $\mcH$ of subgroups of a profinite group $G$ is \emph{upward-closed in $\mcX$} if, given any elements $H_1$ and $H_2$ of $\mcX$ such that $H_1 \in \mcH$ and $H_1 \leq H_2$, then $H_2 \in \mcH$.

Given a subgroup $H$ of $G$, write $\mcX_H$ for $\mcX \cap \sgp(H)$.\end{defn}

The following lemma will have several uses later in the thesis.

\begin{lem}\label{phichain}Let $G$ be a residually finite group, and let $\mcX$ be a set of subgroups of $G$ that is chain-closed.  Let $\mcH$ be a subset of $\mcX$, such that $H$ is $\mcX_H$-finite for every $H \in \mcH$.  Then $\mcX \setminus \mcH$ is chain-closed.

Suppose also that $\mcH$ is upward-closed in $\mcX$ and that $G \in \mcX$.  Then $\mcH$ is max, but if it is infinite, then it contains an infinite descending chain.\end{lem}

\begin{proof}Let $H \in \mcH$.  To show $\mcX \setminus \mcH$ is chain-closed, it suffices to suppose that $H$ is the closure of the union of an ascending chain $\{H_i \mid i \in I\}$ in $\mcX \setminus \mcH$, and derive a contradiction.  Let $U$ be the union of the $H_i$.  Then $\Phi(\mcX_H)U = H$, since $U$ is dense in $H$ and $\Phi(\mcX_H)$ has finite index in $H$.  Let $X$ be a set of right coset representatives for $\Phi(\mcX_H)$ in $H$, and note that $|X| = |H:\Phi(\mcX_H)|$ is finite.  Then for each element $x$ of $X$, there is some $j_x \in I$ such that $\Phi(\mcX_H)H_{j_x}$ contains $x$, and hence $\Phi(\mcX_H)H_j$ contains $X$, where $H_j = \max\{H_{j_x} \mid x \in X \}$.  But then $\Phi(\mcX_H)H_j=H$, so $H=H_j \in \mcX \setminus \mcH$, a contradiction.

From now on, suppose $\mcH$ is upward-closed and infinite.  Suppose there is an infinite ascending chain $H_1 < H_2 < \dots$ in $\mcH$, and let $H$ be the closure of the union of the $H_i$.  Then $H \in \mcH$, and the same argument as before shows $H$ must equal some $H_j$, a contradiction.  Hence $\mcH$ is max.  Now define a directed graph $\Gamma$ with vertex set $\mcH$ as follows: place an arrow from $H_1$ to $H_2$ if $H_2 < H_1$, and there is no $H_3 \in \mcH$ such that $H_2 < H_3 < H_1$.

Let $H \in \mcH$.  If there is an arrow from $H$ to another vertex $K$, then $K$ is a maximal element of $\mcX_H \setminus \{H\}$, by the fact that $\mcH$ is upward-closed in $\mcX$.  So given $H$, there are finitely many possibilities for $K$, since $|H:\Phi(\mcX_H)|$ is finite.  Hence each vertex of $\Gamma$ has finite outdegree.  Clearly $G \in \mcH$; suppose that for some $H \in \mcH$, there is no directed path from $G$ to $H$.  Then we can construct an infinite descending chain in $\mcH$ by setting $G_0 = G$, and thereafter $G_{i+1}$ to be a maximal element of $\mcX \setminus \{G_i\}$ that contains $H$: note that each $G_i$ properly contains $H$, as otherwise $H$ would be reachable from $G$.  Hence we may assume $\Gamma$ is connected.  But in this case, $\Gamma$ has an infinite directed path by K\H{o}nig's lemma; this gives the required descending chain.\end{proof}

The following special cases are immediate:

\begin{cor}\label{phichaincor}Let $G$ be a profinite group.
\vspace{-12pt}
\begin{enumerate}[(i)]  \itemsep0pt
\item Let $\mcH$ be a set of subgroups of $G$ such that $H/\Phi(H)$ is finite for all $H \in \mcH$.  Then $\sgp(G) \setminus \mcH$ is chain-closed in $G$.  Suppose also that $\mcH$ is upward-closed in $\sgp(G)$, and such that $\mcH$ has no infinite descending chains.  Then $\mcH$ is finite.
\item Let $\mcH$ be a set of normal subgroups $H$ of $G$ such that $H/\Phi^\lhd(H)$ is finite for all $H \in \mcH$.  Then $\nsgp(G) \setminus \mcH$ is chain-closed in $G$.  Suppose also that $\mcH$ is upward-closed in $\nsgp(G)$, and such that $\mcH$ has no infinite descending chains.  Then $\mcH$ is finite.\end{enumerate}\end{cor}

\begin{rem}An interesting example for either (i) or (ii) above is if $G$ is a finitely generated pro-$p$ group, and $\mcH$ is any collection of open subgroups.  In both cases, $\mcH$ automatically satisfies the relevant conditions, so that every element of $\sgp(G) \setminus \mcH$ or $\nsgp(G) \setminus \mcH$ is contained in a maximal element.  This is essentially the argument used in \cite{GriNH} to show that any infinite finitely generated pro-$p$ group has a just infinite image.\end{rem}

We now focus on the normal Frattini subgroup of a profinite group.

\begin{lem}\label{ofsglem}Let $G$ be a profinite group such that $O^\fsg(G)=1$.  Then $G$ is a Cartesian product of non-abelian finite simple groups.\end{lem}

\begin{proof}We assume the finite case of the lemma, as it is well-known: see \cite{Cam}, Exercise 4.3.  Given $K \unlhd_o G$, say the set $\mcQ_K = \{K/K, Q_1/K,\dots, Q_n/K\}$ of subgroups of $G/K$ is a witness for the decomposition of $G/K$ if the members of $\mcQ_K \setminus \{K/K\}$ are each non-abelian simple, and together generate $G/K$ as a direct product.  Let $\mcS_K$ be the set of all witnesses for the decomposition of $G/K$.  Then $\mcS_K$ is finite and non-empty for every $K \unlhd_o G$ by the finite case of the lemma, since $O^\fsg(G/K)=1$.  Moreover, given $\mcQ_K \in \mcS_K$, and given $K \leq L \unlhd G$, the set $\mcQ_L = \{L/L, Q_1L/L, \dots , Q_nL/L \}$ is a witness for the decomposition of $G/L$.  This defines a function from $\mcS_K$ to $\mcS_L$, and so the set $\{\mcS_K \mid K \unlhd_o G \}$, together with these functions, forms an inverse system of finite non-empty sets.  It follows that the inverse limit is non-empty, and so there is a set $\mcR$ of subgroups of $G$, such that
\[ \mcR_K := \{ RK/K \mid R \in \mcR \} \]
is a witness for the decomposition of $G/K$, for every $K \unlhd_o G$.  We conclude the following:
\vspace{-12pt}
\begin{enumerate}[(a)]  \itemsep0pt
\item for any $R \in \mcR \setminus \{1\}$ we have $R \unlhd G$, and furthermore $R$ is the inverse limit of $\fsg$-groups, so in fact $R \in \fsg$;
\item $G$ is generated by $\mcR$, and distinct elements of $\mcR$ have trivial intersection.\end{enumerate}

Hence $G$ is the Cartesian product of non-abelian simple groups, with decomposition given by $\mcR \setminus \{1\}$.\end{proof}

\begin{prop}\label{philhdprop}Let $G$ be a profinite group.
\vspace{-12pt}
\begin{enumerate}[(i)]  \itemsep0pt
\item Let $H \unlhd G$.  Then $\Phi^\lhd(H) \leq \Phi^\lhd(G)$.
\item Suppose $L \in \nsgp_\Phi$ for every open normal subgroup $L$ of $G$.  Then $G \in \nsgp^*_\Phi$.
\item Suppose $\Phi^\lhd(G)=1$.  Then $G$ is a Cartesian product of finite simple groups.  In particular, let $X$ be the union of all finite normal subgroups of $G$; then $G = \overline{X}$.\end{enumerate}\end{prop}

\begin{proof}(i) Suppose not.  Then there is a normal subgroup $N$ of $G$ such that $G/N$ is simple, and such that $N$ does not contain $\Phi^\lhd(H)$.  Now $HN/N$ is a non-trivial normal subgroup of $G/N$, so $HN/N = G/N$.  But then $H \cap N$ is a normal subgroup of $H$ such that $H/(H \cap N) \cong HN/N$ is simple, so that $\Phi^\lhd(H) \leq H \cap N \leq N$.

(ii) Let $H$ be an open subgroup of $G$, and let $K$ be the core of $H$ in $G$.  Then $K$ is an open normal subgroup of $G$, so $\Phi^\lhd(K)$ has finite index in $K$ and hence in $G$.  Now $\Phi^\lhd(K) \leq \Phi^\lhd(H)$ by (i), so $\Phi^\lhd(H)$ has finite index in $H$.

(iii) Let $A = O^\fsg(G)$, let $\mcN$ be the set of normal subgroups of $G$ of prime index, and let $B = \bigcap_{N \in \mcN} N$.  Note first that $A \cap B = \Phi^\lhd(G)=1$.  Also, $G/AB$ is an image of both an abelian group $G/B$ and a perfect group $G/A$, so must be trivial; hence $G=AB$, so $G \cong A \times B$.  It follows that $A \cong G/B$ is abelian, and hence is a Cartesian product of its Sylow subgroups.  Every finite image of $G/B$ has squarefree exponent, so for each $p$, its $p$-Sylow subgroup is elementary abelian, and thus a Cartesian product of groups of order $p$.  Similarly $B \cong G/A$, so $B$ is a Cartesian product of non-abelian finite simple groups by Lemma \ref{ofsglem}.\end{proof}

\paragraph{}In the case of a pro-$p$ group $G$, we can also obtain properties of the normal subgroup lattice using centres of images of $G$.

\begin{defn}Given a pro-$p$ group $G$, define the following invariant:
\[r^Z(G)=\mathrm{sup}_{P \unlhd G}(r(Z(G/P)))\]

Note that $r(G) \geq r^Z(G) \geq d(G)$, but $r^Z(G)$ may be infinite even if $d(G)$ is finite: consider for instance the free pro-$p$ group on $d$ generators, for $d \geq 2$.

Given a set of subgroups $\mcX$ of a profinite group, say $Y\in \mcX$ is \emph{redundant} in $\mcX$ if $ \langle \mcX \setminus \{Y\} \rangle = \langle \mcX \rangle$.  Say $\mcX$ is non-redundant if no element of $\mcX$ is redundant in $\mcX$.\end{defn}

\begin{prop}Let $G$ be a pro-$p$ group, and let $n$ be an integer.  Then $r^Z(G) \leq n$ if and only if $|\mcX|\leq n$ for every non-redundant set $\mcX$ of normal subgroups.\end{prop}

\begin{proof}Suppose $r^Z(G) \geq n$.  Let $P$ be a normal subgroup of $G$ with $r(Z(G/P)) =n$.  Then a subgroup of $Z(G/P)$ is generated without redundancy by a finite set $\mcX=\{K_1/P, \dots, K_n/P\}$ of $n$ cyclic groups.  This implies $\mcX'=\{K_1,\dots,K_n\}$ is a non-redundant set of normal subgroups of $G$ of size $n$.

In the other direction, let $\{K_1, \dots, K_m\}$ be a non-redundant set of $m$ normal subgroups, generating a subgroup $M$.  Let $L = \Phi^{\nsgp(G)}(M)$, in other words $L = M^p[G,M]$, and let $R_i = K_iL/L$.  By Lemma \ref{fratmax}, $M/L$ is generated without redundancy by $\{R_1, \dots, R_n\}$.  Furthermore, if $N$ is a proper subgroup of $M$ that is maximal subject to being normal in $G$, then $|M/N|=p$ and $M/N$ is centralised by $G$, since $G$ is a pro-$p$ group.  Hence $M/L \leq Z(G/L)$, and 
 \[ m \leq d(M/L) \leq r(Z(G/L)) \leq r^Z(G). \qedhere\]\end{proof}

\section{Some consequences of Tate's theorem}\label{csqtate}

Tate's theorem (Theorem \ref{tate}) has straightforward but important consequences for fusion in profinite groups, as follows:

\begin{cor}\label{tatecor}Let $G$ be a profinite group, and let $S \in \Sylp(G)$.
\vspace{-12pt}
\begin{enumerate}[(i)]  \itemsep0pt 
\item Let $M$ be a normal subgroup of $G$ such that $S \cap M \leq \Phi(S)$.  Then $SM$ is $p'$-normal, and $O_{p'}(G/M) = O_{p'}(G)M/M$.
\item Let $M$ and $N$ be normal subgroups of $G$ such that $S \cap M \leq \Phi(S)N$.  Then $MN/N$ is $p'$-normal.\end{enumerate}\end{cor}

\begin{proof}(i) We see that $(SM)'(SM)^p \leq \Phi(S)M$, so
\[ ((SM)'(SM)^p) \cap S \leq \Phi(S)M \cap S = \Phi(S) = S'S^p.\]
Hence $S$ controls $p$-transfer in $SM$ by Theorem \ref{tate}, so $SM$ is $p'$-normal by Corollary \ref{ptranscomp}.

For the final assertion, let $O$ be the lift of $O_{p'}(G/M)$ to $G$.  It is clear that $O \geq O_{p'}(G)M$.  On the other hand, we have $SM/M \cap O/M = 1$ since $O/M$ is a pro-$p'$ group, so $S \cap O \leq S \cap M \leq \Phi(S)$. This ensures that $O$ has a normal $p'$-Hall subgroup $K$ say, by the same argument as for $M$; this $K$ is normal in $G$, and $O$ contains $O_{p'}(G)$, so in fact $K = O_{p'}(G)$.  Since $M$ contains a $p$-Sylow subgroup of $O$, it follows that $O=O_{p'}(G)M$.

(ii) $MN/N$ is a normal subgroup of $G/N$, and $\Phi(S/N) = \Phi(S)N/N$ contains $(M \cap S)N/N$.  The result follows by part (i) applied to $G/N$.\end{proof}

\begin{cor}\label{tatefin}Let $G$ be a profinite group with $d_p(G)$ finite.  Then $G/O_{p'}(G)$ is virtually pro-$p$.  If $p=2$, then $G$ is also virtually prosoluble.\end{cor}

\begin{proof}Let $S$ be a $p$-Sylow subgroup of $G$.  Then $d(S)=d_p(G)$ is finite, so $S/\Phi(S)$ is finite, and hence there must be some open normal subgroup $N$ of $G$ such that $S \cap N \leq \Phi(S)$.  By Corollary \ref{tatecor}, $N/O_{p'}(N)$ is a pro-$p$ group, so $G/O_{p'}(N)$ is virtually pro-$p$.  Now $O_{p'}(N) \leq O_{p'}(G)$, so $G/O_{p'}(G)$ is an image of $G/O_{p'}(N)$; hence $G/O_{p'}(G)$ is virtually pro-$p$.

Now suppose $p=2$.  Let $R$ be the subgroup of $G$ such that $R \geq O_{2'}(G)$ and $R/O_{2'}(G) = O_2(G/O_{2'}(G))$.  Since $G/O_{2'}(G)$ is virtually pro-$2$, $R$ has finite index in $G$.  By the Odd Order Theorem, $O_{2'}(G)$ is prosoluble, and so $R$ is prosoluble.\end{proof}

\begin{cor}\label{tatenilp}Let $G$ be a profinite group involving only finitely many primes, such that $d_p(G)$ is finite for every $p$.  Then $G$ is virtually pronilpotent.\end{cor}

\begin{proof}For each prime $p$ dividing $|G:1|$, choose an open normal subgroup $N_p$ of $G$ such that $N_p$ has a normal $p'$-Hall subgroup; such an $N_p$ exists by the previous corollary.  Set $N$ to be the intersection of these $N_p$, and note that $|G:N|$ is finite.  Then $N$ has a normal $p'$-Hall subgroup for every prime $p$, since $N \leq N_p$.  It follows that $N$ has a normal $\pi$-Hall subgroup for any set of primes $\pi$, and so $N$ is the direct product of its Sylow subgroups.  Hence $N$ is pronilpotent.\end{proof}

\begin{rem}This was also proved by Mel'nikov in \cite{Mel}.  We will give a strengthening of Mel'nikov's result in Section 5.3.\end{rem}

\begin{prop}\label{pspace}Let $G$ be a profinite group with a $d$-generated $p$-Sylow subgroup $S$.
\vspace{-12pt}
\begin{enumerate}[(i)]  \itemsep0pt
\item Let $\mcX$ be a set of normal subgroups of $G$, and let $H = \langle \mcX \rangle$.  Then there is a subset $\mcK$ of $\mcX$ such that $|\mcK| \leq d$, and such that $\mcK$ generates a subgroup $K$ of $G$ satisfying
\[ \Phi(S)(H \cap S) = \Phi(S)(K \cap S).\]
In particular, $H/K$ is $p'$-normal.
\item Let $\mcE$ be a subgroup class such that $\mcE \subseteq \nsgp$, with the following closure properties for all $K,L \unlhd G$:
\vspace{-6pt}
\begin{enumerate}[(a)]  \itemsep0pt
\item $K,L \in \mcE(G) \Rightarrow KL \in \mcE(G)$;
\item $LK/K \in \mcE(G/K) \wedge K \in \mcE(G) \Rightarrow KL \in \mcE(G)$.\end{enumerate}

Then there is an $\mcE$-subgroup $K$ of $G$, such that every $\mcE$-subgroup of $G/K$ is $p'$-normal.\end{enumerate}\end{prop}

\begin{proof}(i) Given a normal subgroup $N$ of $G$, write $V_S(N) = (N \cap S)\Phi(S)/\Phi(S)$, regarded as a subspace of $S/\Phi(S) \cong (\bF_p)^{d(S)}$.  Since $H$ is generated by $\mcX$, there are $H_1, \dots, H_k \in \mcX$ such that
\[ V_S(H) = V_S(H_1) + \dots + V_S(H_k),\]
and such that $k \leq \dim(V_S(H)) \leq d$.  Now set $\mcK = \{H_1, \dots, H_k\}$ and let $K = \langle \mcK \rangle$; then clearly 
\[ \Phi(S)(H \cap S) = \Phi(S)(K \cap S)\]
as required.  Hence $H/K$ is $p'$-normal by Corollary \ref{tatecor}.

(ii) Apply part (i) to the class $\mcX = \mcE(G)$, to obtain a finite subset $\mcK$ generating a subgroup $K$ as before.  Also, let $H$ be as before.  Since $\mcK$ is a finite subset of $\mcE(G)$, we have $K \in \mcE(G)$ by property (a).  Now let $M/K$ be an $\mcE$-subgroup of $G/K$.  Then $M \in \mcE(G)$ by property (b), and so $M \leq H$, in other words $M/K \leq H/K$; since $H/K$ is $p'$-normal, so is $M/K$.\end{proof}

Under the circumstances of (ii) above, define the \emph{$\mcE$-based $p$-dimension} of $G$ to be $\dim(V_S(H))$, where $H = \langle \mcE(G) \rangle$.  In particular, define $f_p(G) = \dim(V_S(\overline{X}))$, where $X$ is the union of all finite normal subgroups of $G$.  Note that $f_p(G) \leq f_p(S) \leq d(S)$.

\begin{cor}\label{fplayer}Let $G$ be a profinite group with finitely generated $p$-Sylow subgroup $S$.  Then $G$ has a finite normal subgroup $K$ that is the normal closure of at most $f_p(G)$ elements, such that every finite normal subgroup of $G/K$ is $p'$-normal.\end{cor}

\section{The virtual centre and finite radical of a profinite group}

\begin{defn}The \emph{finite radical} $\Fin(G)$ of a group $G$ is the union of all finite normal subgroups of $G$.  The \emph{virtual centre} $VZ(G)$ of a group $G$ is the set of all elements $x$ of $G$ such that $C_G(x)$ has finite index, or equivalently, the union of all centralisers of (normal) subgroups of finite index.\end{defn}

Both $\Fin(G)$ and $VZ(G)$ are abstract subgroups of $G$, though they need not be closed in general.  Note that $G=VZ(G)$ if and only if all conjugacy classes of $G$ are finite.  Such groups are known as FC-groups, and have been studied extensively in their own right.  For more details, see the research note of Tomkinson (\cite{Tom}) on the subject.

With a topological group $G$ there are two notions of the size of $G$: one is the cardinality $|G|$ of the underlying set, and the other is the smallest cardinality of a dense abstract subgroup, which may be strictly smaller.  Given a topological group in which $\Fin(G)$ is dense, this raises the question of whether $|\Fin(G)|$ is equal to $|G|$, or the smallest cardinality of a dense abstract subgroup, or somewhere in between.  We are particularly interested here in the case where $G$ is a profinite group.  Both extremes occur in the case of abelian countably based profinite groups: for example, an infinite inverse limit of cyclic groups of square-free order has countable finite radical, whereas the Cartesian product of infinitely many isomorphic finite cyclic groups is uncountable, and equal to its finite radical.

An important result that relates $\Fin(G)$ to $VZ(G)$ is Dicman's Lemma:

\begin{lem}[Dicman \cite{Dic}] Let $X$ be a finite subset of the group $G$, such that each element of $X$ has finite order, and such that $X \subseteq VZ(G)$.  Then $\langle X \rangle^G$ is finite.\end{lem}

\begin{cor}\label{diccor}Let $G$ be any group.  Then $\Fin(G)$ is the set of all elements of $VZ(G)$ of finite order.  In particular, $\Fin(H) \subseteq \Fin(G)$ whenever $H$ is a subgroup of $G$ of finite index.\end{cor}

The following lemma will also prove useful:

\begin{lem}[Gorchakov \cite{Gor}; Hartley \cite{Har}] Let $G$ be a periodic FC-group, and suppose $G$ is a subgroup of $\prod_{i \in I} F_i$, where each $F_i$ is finite.  If $I$ is infinite then $|G/Z(G)| \leq |I|$.\end{lem}

\begin{cor}\label{GHcor}Let $G$ be a $\kappa$-based profinite group, where $\kappa$ is an infinite cardinal.  Then $|\Fin(G)/Z(\Fin(G))| \leq \kappa$; in particular, if $|\Fin(G)| > \kappa$ then $|\Fin(G)|=|Z(\Fin(G))|$.\end{cor}

\begin{proof}There is a canonical injection from $G$ to the Cartesian product $\prod_{i \in I}F_i$ of its finite continuous images, so $\Fin(G)$ is isomorphic to a subgroup of this Cartesian product.  Since $G$ is $\kappa$-based, $|I|=\kappa$ here.  The result now follows immediately from the lemma.\end{proof}

The following lemma will be useful for estimating the sizes of the virtual centre and finite radical of a profinite group:

\begin{lem}Let $G$ be a $\kappa$-based profinite group, with $\kappa$ an infinite cardinal.
\vspace{-12pt}
\begin{enumerate}[(i)]  \itemsep0pt
\item Suppose $|VZ(G)| = \mu > \kappa$.  Then there is an open subgroup $H$ of $G$ such that $|Z(H)| = \mu$.
\item Suppose $|\Fin(G)| = \nu > \kappa$.  Then there is a subgroup $M$ of $G$ of cardinality $\nu$ and finite exponent, such that $M$ is a central closed subgroup of an open normal subgroup $H$ of $G$, with $\Fin(H) = \Fin(G)$.\end{enumerate}\end{lem}

\begin{proof}(i) The virtual centre of $G$ is by definition the union of the subgroups $C_G(H)$ of $G$, where $H$ is an abstract subgroup of finite index.  There are at most $\kappa$ distinct centralisers of this form, since $C_G(H) = C_G(\overline{H})$, and $G$ has only $\kappa$ closed subgroups of finite index.  It follows that there is some open subgroup $H$ for which $|C_G(H)| = \mu$; then $|C_G(H):Z(H)|$ is finite, so $|Z(H)|= \mu$.

(ii) By Corollary \ref{GHcor}, $|Z(\Fin(G))| = \nu$.  Now $Z(\Fin(G))$ is the union of subsets of the form $C_G(H) \cap Z(\Fin(G))$, where $H \leq_o G$.  There are at most $\kappa$ distinct subsets of this form, so there is some $H \leq_o G$ for which $|C_G(H) \cap Z(\Fin(G))|=\nu$.  We are free to assume $H$ is normal and contains $\Fin(G)$, by replacing $H$ with the open normal subgroup $\mathrm{Core}_G(H)\Fin(G)$ if necessary; this ensures $\Fin(G) = \Fin(H)$, by Corollary \ref{diccor}.  Let $T_n$ be the subgroup of $C_G(H) \cap Z(\Fin(G))$ generated by the elements of order dividing $n$; then $T_n$ is abelian and has exponent $n$.  Take some $n$ for which $|T_n|=\nu$.  Now take $M = T_n \cap H$.  Then $M$ is a central closed subgroup of $H$ of exponent $n$, and $|M|=|T_n| = \nu$.\end{proof}

In the countably-based case, this specialises to the following:

\begin{cor}Let $G$ be a countably-based profinite group.
\vspace{-12pt}
\begin{enumerate}[(i)]  \itemsep0pt
\item The virtual centre of $G$ is countable if and only if, in every open subgroup $H$ of $G$, the centre $Z(H)$ is finite.
\item The finite radical of $G$ is countable if and only if, in every open subgroup $H$ of $G$, the centre $Z(H)$ has no infinite abstract subgroups of finite exponent.\end{enumerate}\end{cor}

\begin{proof}(i) If $VZ(G)$ is uncountable, then by the lemma there is an open normal subgroup of $G$ with infinite centre.  Conversely, if $VZ(G)$ is countable, it cannot contain any infinite closed subgroup of $G$, and so every open subgroup must have finite centre.

(ii) If $\Fin(G)$ is uncountable, then by the lemma there is an open normal subgroup of $G$ that has an infinite central subgroup of finite exponent.  Conversely, if an open subgroup $H$ of $G$ has an infinite central abstract subgroup $K$ of finite exponent, then the closure of $K$ is also contained in $Z(H)$, and also of the same finite exponent by Lemma \ref{topcentlem}.  Hence the closure of $K$ is contained in $\Fin(G)$, so $\Fin(G)$ is uncountable.\end{proof}

Given a countably-based profinite group $G$ with countable $K$, where $K$ is the virtual centre or finite radical, we turn to the question of whether or not $K$ is finite.  Note that $G$ has infinite virtual centre if and only if every open subgroup of $G$ does, and similarly for the finite radical, so these properties depend only on the commensurability class of $G$.

For upper bounds on $|\Fin(G)|$, we can specialise to the case of pro-$p$ groups, thanks to the following:

\begin{lem}Let $G$ be a profinite group, and let $\kappa$ be an infinite cardinal.  Then $|\Fin(G)| \leq \kappa$ if and only if $|\Fin(G) \cap S| \leq \kappa$ whenever $S$ is a Sylow subgroup of $G$.\end{lem}

\begin{proof}Write $\Fin(G)_p$ for the set of (pro-)$p$ elements of $\Fin(G)$. Let $x \in \Fin(G)$.  Then for some $n$ and distinct primes $p_1, \dots, p_n$, there is a primary decomposition $x = x_1 \dots x_n$ of $x$, such that $x_i \in \Fin(G)_{p_i}$.  Hence
\[|\Fin(G)| \leq \sup_{\piP} |\Fin(G)_p|\aleph_0.\]
By Sylow's theorem, the set $\Fin(G) \cap S$ accounts for all conjugacy classes of pro-$p$ elements of $\Fin(G)$; also, every conjugacy class of $\Fin(G)$ is finite by definition.  Hence
\[|\Fin(G)_p| \leq |\Fin(G) \cap S|\aleph_0.\]
The conclusion is now clear.\end{proof}

Now consider the case of pro-$p$ groups.  A well-known property of finite $p$-groups can be used here to obtain a condition for whether or not the finite radical is finite.

\begin{lem}Let $G$ be a non-trivial finite $p$-group, and let $H$ be a $p$-group of automorphisms of $G$.  Then $C_G(H) > 1$.\end{lem}

\begin{cor}Let $G$ be a pro-$p$ group, such that $\Fin(G) \cap Z(G)$ is finite.  Then $\Fin(G)$ is finite.\end{cor}

\begin{proof}Let $W = \Fin(G) \cap Z(G)$, and suppose $\Fin(G)$ is infinite.  Then $G$ has an open normal subgroup $K$ such that $K \cap W = 1$.  Let $F = \Fin(G) \cap K$.  Then $F$ is a union of finite normal subgroups of $G$, and $|F| = |\Fin(G)|$.  In particular, $F$ contains a finite non-trivial normal subgroup $N$ of $G$.  Now $G/C_G(N)$ is a finite $p$-group, as $G$ is a pro-$p$ group, so $C_N(G) = N \cap Z(G) > 1$.  But $N \cap Z(G) \leq (K \cap \Fin(G) \cap Z(G)) = 1$, a contradiction.\end{proof}

For finitely generated pro-$p$ groups, there is another restriction on the `size' of $\Fin(G)$.

\begin{prop}Let $G$ be a finitely generated pro-$p$ group.  Suppose $G=\overline{\Fin(G)}K$ for some $K \leq G$.  Then $|G:K|$ is finite.\end{prop}

\begin{proof}Every open subgroup $U$ of $G$ has $U/\Phi(U)$ finite.  Hence $\mcX \setminus \mcH$ is chain-closed by Lemma \ref{phichain}, where $\mcX=\sgp(G)$ and $\mcH=\fsgp(G)$.  In particular, if $K$ has infinite index, then $K$ is contained in a maximal element $M$ of $\mcX \setminus \mcH$; we may assume $K=M$.  This means $K$ has infinite index, but every subgroup of $G$ properly containing $K$ has finite index.  In particular, let $N$ be a finite normal subgroup of $G$.  Then $KN$ has infinite index in $G$, as it is the union of finitely many cosets of $K$.  Hence $KN=K$, that is $N \leq K$, by the maximality property of $K$.  Thus $K$ contains every finite normal subgroup of $G$, so $K = \overline{\Fin(G)}K = G$.  But then $|G:K|=1$, a contradiction.\end{proof}

Finally, here is a result concerning the composition of a countably based profinite group with respect to finite normal subgroups, once again illustrating the role played by pronilpotent subgroups, and hence pro-$p$ subgroups, of a profinite group.

\begin{thm}Let $G$ be a profinite group, and let $K$ be a countably based closed subgroup of $G$ that is topologically generated by finite normal subgroups of $G$.  Then $KF(G)/F(G)$ is a central product of finite groups, each of which is normal in $G/F(G)$.\end{thm}

\begin{proof}We may assume that $K$ is infinite.  Let $K = R_0 > R_1 > \dots$ be an irreducible descending series of $G$-invariant open subgroups of $K$ such that $\bigcap R_i = 1$.  Write $T_i$ for the section $R_i/R_{i+1}$, and given integers $i > j$, write $C_{i,j}$ for the centraliser of $T_i$ in $R_j$.  Then $C_{i,j}$ is $G$-invariant, and so either $C_{i,j} \leq R_{j+1}$ or $R_j \leq C_{i,j}R_{j+1}$.  Now construct a graph $\Gamma$: the vertices are the sections $T_i$, and $T_i$ is adjacent to $T_j$ for $i > j$ if $C_{i,j} \leq R_{j+1}$.

The $G$-invariant finite subgroups of $K$ generate a dense subgroup of $K$, and so in particular there is a finite normal subgroup $F_j$ of $K$ such that $R_j = R_{j+1}F_j$.  It follows that $C_K(F_j)$ is an open $G$-invariant subgroup of $K$, and so contains $R_k$ for some $k > j$; hence for any $i \geq k$, $T_j$ does not contribute to the automorphisms induced on $T_i$, and hence $T_i$ and $T_j$ are not adjacent in $\Gamma$.  Thus all vertices of $\Gamma$ have finite degree.

We now claim that $T_i$ is adjacent to $T_k$ whenever $i > j > k$ and $(T_i,T_j,T_k)$ is a path in $\Gamma$.  If this were not the case, $T_k$ would be covered by $C_K(T_i)$, a $G$-invariant subgroup.  Now $[R_j,C_K(T_i)] \leq C_{i,j}$, which is contained in $R_{j+1}$ by the assumption that $T_i$ is adjacent to $T_j$, so $C_K(T_i)$ does not contribute to the automorphisms induced on $T_j$, and so $T_j$ and $T_k$ cannot be adjacent, a contradiction.

It follows that every component of $\Gamma$ is finite; say the components of $\Gamma$ are $\{ \Gamma_l \mid l \in \bN\}$.  Now let $K_l$ be intersection of $C_K(T_j)$ as $T_j$ ranges over all vertices that are not in the component $\Gamma_l$, and let $L$ be the intersection of $C_K(T_j)$ as $j$ ranges over all values.  Then the $K_l$ and $L$ are closed subgroups of $K$ that are normal in $G$.  Let $L_i = (L \cap R_i)/(L \cap R_{i+1})$.  Then $L_i$ is a central section of $L$, so the subgroups $R_i \cap L$ form a central series for $L$.  Thus $L$ is pronilpotent and $L \leq F(G)$.  Also, $|K_l:L|$ is finite since $\Gamma_l$ is finite and each section is finite.

Let $M$ be the closure of the group generated by all $K_l$.  Then $K_l$ covers $T_i$ whenever $T_i \in \Gamma_l$, and so $M$ covers every section of the series; hence $M=K$.  Now consider the interaction between different $K_l$.  We have $K_l \cap K_m = L$ for any distinct $l$ and $m$, so $[K_l,K_m] \leq L$; hence $[K_l,M_l] \leq L$, where $M_l$ is the closure of the group generated by all $K_m$ for $m \in \bN \setminus \{l\}$.  It follows that $KF(G)/F(G)$ is a central product of the subgroups $K_lF(G)/F(G)$, all of which are finite and normal in $G/F(G)$.\end{proof}

\begin{cor}Let $G$ be a profinite group.  Suppose $\overline{K}$ is countably based.  Then $\overline{K}F(G)/F(G)$ is a central product of finite groups, each of which is normal in $G/F(G)$.\end{cor}

\section{Commensurators of profinite groups}\label{commsec}

We begin with some definitions based on those of Barnea, Ershov and Weigel in \cite{BEW}.

\begin{defn}A \emph{virtual automorphism} of the profinite group $G$ is a continuous isomorphism between open subgroups of $G$.  Two virtual automorphisms are regarded as equivalent if they coincide on some open subgroup of $G$.  It is clear that up to equivalence, we can compose any two virtual automorphisms, and that the equivalence classes thus form a group.  This is the \emph{(abstract) commensurator} $\Comm(G)$ of $G$.  At this stage we do not assign a topology to $\Comm(G)$.  Those virtual automorphisms equivalent to the identity are called \emph{virtually trivial}.\end{defn}

Note that $\Comm(G)$ is canonically isomorphic to $\Comm(U)$, where $U$ is any open subgroup of $G$.  The structure of $\Comm(G)$ is of particular interest if $VZ(G)=1$, thanks to the following:

\begin{prop}[\cite{BEW}]\label{vzcomm}Let $G$ be a profinite group with $VZ(G)=1$, and suppose $\phi: U \rightarrow V$ is a virtual automorphism of $G$ that is virtually trivial.  Then $U=V$ and $\phi=\id_U$.\end{prop}

\begin{cor}Let $G$ be a profinite group with $VZ(G)=1$, such that $G$ is an open subgroup of the locally compact group $L$.  Let $l \in N_L(G)$ and suppose $l$ centralises an open subgroup of $G$.  Then $l$ centralises $G$.\end{cor}

Hence in this situation, $\Comm(G)$ contains an abstract copy of any locally compact group $L$ containing $G$ for which $G \leq_o L$ and $C_L(G)=1$.  In particular, if $VZ(G)=1$, there is a natural embedding of $\Aut(H)$ into $\Comm(G)$ for all $H \leq_o G$, and so we may identify $\Aut(H)$ with a subgroup of $\Comm(G)$.

One important aspect of virtual automorphisms of $G$ is their effect on the indices of open subgroups, which corresponds to a homomorphism from $\Comm(G)$ to $\bQ^\times_{>0}$.

\begin{defn}Let $H$ and $K$ be isomorphic open subgroups of $G$.  Given an isomorphism $\theta$ from $H$ to $K$, write $\ir(\theta)$ for $|G:H|/|G:K|$.  This is clearly invariant under equivalence of virtual automorphisms.  Define $\ir(\phi)$ for $\phi \in \Comm(G)$ as $\ir(\theta)$ for any $\theta$ representing $\phi$; this defines a function $\ir$ from $\Comm(G)$ to the multiplicative group $\bQ^\times_{>0}$ of positive rationals, which we call the \emph{index ratio} of $G$.  Say $G$ is \emph{index-stable} if $\ir(\Comm(G))=1$, that is, any pair of isomorphic open subgroups of $G$ have the same index, and say $G$ is \emph{index-unstable} otherwise.\end{defn}

\begin{lem}\label{irhom}Let $G$ be a profinite group.  Then the index ratio $\ir$ of $G$ is a homomorphism of abstract groups from $\Comm(G)$ to $\bQ^\times_{>0}$.  In particular, if $G$ is index-unstable then $|\ir(\phi)|$ is unbounded as $\phi$ ranges over the elements of $\Comm(G)$.\end{lem}

\begin{proof}Let $\phi,\psi \in \Comm(G)$, and let $\phi'$ and $\psi'$ be representatives of $\phi$ and $\psi$ respectively such that the composition $\phi' \psi'$ is defined.  Let $H$ be the domain of $\phi'$.  Then
\[\ir(\phi\psi)= \frac{|G:H|}{|G:H^{\phi'\psi'}|}=\frac{|G:H|}{|G:H^{\phi'}|}\frac{|G:H^{\phi'}|}{|G:H^{\phi'\psi'}|} = \ir(\phi)\ir(\psi).\]
The conclusions are now clear.\end{proof}

\begin{lem}\label{subnorlem}Let $G$ be a profinite group.  Let $H$ and $K$ be open subgroups of $G$, and suppose $\theta$ is an isomorphism from $H$ to $K$.  Then there are subgroups $H_2 \leq H$ and $K_2 \leq K$, with $H_2 \unlhd^2_o G$ and $K_2 \unlhd_o G$, such that the restriction of $\theta$ to $H_2$ induces an isomorphism from $H_2$ to $K_2$.\end{lem}

\begin{proof}Let $H_3$ be the core of $H$ in $G$, and let $K_3$ be its image under $\theta$.  Now let $K_2$ be the core of $K_3$ in $G$, and let $H_2$ be its preimage under $\theta$.  By construction, $K_2$ is normal in $G$, and hence normal in $K_3$.  Since $\theta$ maps $H_3$ isomorphically to $K_3$, this means that $H_2$ must be the corresponding normal subgroup of $H_3$.  But $H_3$ is normal in $G$, so $H_2 \unlhd^2_o G$.\end{proof}

\begin{defn}Let $G$ be a profinite group with $VZ(G)=1$, and let $\mcH$ be a set of open subgroups of $G$.  Define the \emph{local commensurator} $\LComm_\mcH(G)$ with respect to $\mcH$ to be the union of the subgroups $\Aut(H)$ of $\Comm(G)$, as $H$ ranges over $\mcH$.  (Note that $\LComm_\mcH(G)$ itself may not be a subgroup in general.)  The (absolute) \emph{local commensurator} of $G$ is given by $\LComm(G) := \LComm_\mcH(G)$, where $\mcH$ is the set of all open subgroups of $G$.  Denote by $\KComm(G)$ the kernel of the index ratio of $G$.\end{defn}

We have $\LComm(G) \subseteq \KComm(G) \leq \Comm(G)$ for any profinite group $G$ with $VZ(G)=1$.  We consider conditions under which two or more of these subsets coincide.

\begin{defn}Given a subgroup $H$ of a profinite group $G$, say $H$ is \emph{hereditarily characteristic} if, given any open subgroup $K$ of $G$ such that $H \leq K$, then $H$ is characteristic in $K$.  Say $H$ is \emph{one of a kind} if $H \cong K$ implies that $H=K$, for any subgroup $K$ of $G$.  Say $H$ is \emph{one of a kind up to index} if $H \cong K$ and $|G:H| = |G:K|$ together imply that $H=K$.\end{defn}

The significance of hereditarily characteristic and one-of-a-kind subgroups (up to index) for the commensurator is given by the following:

\begin{lem}Let $G$ be a profinite group with $VZ(G)=1$, and let $K$ be an open subgroup.  Then $K$ is hereditarily characteristic if and only if $\Aut(H) \leq \Aut(K)$ as subgroups of $\Comm(G)$ for every subgroup $H$ containing $K$.\end{lem}

\begin{proof}If $\Aut(H) \leq \Aut(K)$, then clearly $K$ is characteristic in $H$.  Conversely, if $K$ is characteristic in $H$, then every automorphism of $H$ restricts to an automorphism of $K$, so that $\Aut(H)$ embeds into $\Aut(K)$.\end{proof}

\begin{prop}\label{commcontrol}Let $G$ be a profinite group with $VZ(G)=1$.  Let $\mcK$ be a set of subgroups of $G$ that form a countable base for the neighbourhoods of $1$.
\vspace{-12pt}
\begin{enumerate}[(i)]  \itemsep0pt
\item Suppose every $K \in \mcK$ is hereditarily characteristic in $G$.  Then $\LComm(G)= \LComm_\mcK(G)$, and there is a descending chain $\mcL \subseteq \mcK$ such that $\LComm(G)$ is the union of the ascending chain of subgroups given by $\{\Aut(K) \mid K \in \mcL\}$.
\item Suppose every $K \in \mcK$ is one of a kind up to index.  Then $\KComm(G) = \LComm(G)$.\end{enumerate}\end{prop}

\begin{proof}It is clear from the properties given that there is a descending chain $\mcL \subseteq \mcK$, such that every open subgroup of $G$ contains some $L \in \mcL$.  Hence we may assume $\mcK$ itself is such a descending chain.

(i) Let $H$ be an open subgroup of $G$, and suppose $H$ contains $K \in \mcK$.  Then by the lemma, $\Aut(H) \leq \Aut(K)$, so $\LComm(G) = \LComm_\mcK(G)$.  By the lemma, $\{ \Aut(K) \mid K \in \mcK\}$ is an ascending chain.

(ii) Let $\theta$ be an isomorphism between open subgroups $H_1$ and $H_2$ of $G$, such that $\ir(\theta)=1$.  Then by Lemma \ref{obchain}, there is some $K \in \mcK$ such that $K \leq H_1 \cap H_2$.  Then $K^\theta \cong K$ and $|G:K|=|G:K^\theta|$, so $K = K^\theta$.  Hence $\theta$ is equivalent to an automorphism of $K$.\end{proof}

Radicals of $G$ give the greatest potential for control of $\Comm(G)$.  Note that if $K \leq G$ such that $K$ is one of a kind in $G$, then $K = O_{[K]}(G)$.

\begin{prop}Let $G$ be a profinite group such that $VZ(G)=1$.  Suppose that $\mcR$ is a set of open radicals of $G$, such that every open subgroup of $G$ contains some $R \in \mcR$.  Then $\Comm(G) = \LComm_\mcR(G)$.\end{prop}

\begin{proof}Let $\theta$ be an isomorphism between open subgroups $H$ and $K$ of $G$.  By Lemma \ref{subnorlem}, we may assume $H$ is subnormal and $K$ is normal.  Then there is some $\mcX$ such that $O_\mcX(G) \in \mcR$, and such that $R$ is contained in the subnormal subgroup $H \cap K$ of $G$.  This ensures
\[ O_\mcX(H \cap K) = O_\mcX(H) = O_\mcX(K) = O_\mcX(G).\]
Since $\theta$ is an isomorphism, $(O_\mcX(H))^\theta = O_\mcX(K) = O_\mcX(H)$.  Hence $\theta$ is equivalent to an element of $\Aut(O_\mcX(H))$, which is the same as $\Aut(O_\mcX(G))$.\end{proof}

Finally, here is an example where the commensurator is known to be a finite extension of the original group; this will be used later as an example in other contexts.

\begin{thm}\label{commnottgp}Let $p$ be a prime, and let $N$ be the Nottingham group over the field of $p$ elements, where $p \geq 5$.  Then the following isomorphisms hold:
\vspace{-12pt}
\begin{enumerate}[(i)]  \itemsep0pt
\item \emph{\textbf{[Klopsch \cite{Klo}]}} $\Out(N) \cong C_{p-1}$;
\item \emph{\textbf{[Ershov \cite{Ers}]}} $\Comm(N) \cong \Aut(N)$.\end{enumerate}\end{thm}

\section{Coprime automorphisms of pro-$p$ groups and the $c$ invariant}\label{cinvsec}

We consider the characteristic subgroup structure of a finitely generated pro-$p$ group $G$, and the restriction this places on coprime automorphisms of $G$.  Throughout this section, we will make use of the definitions and results from Section 1.5.

\begin{defn}Let $G$ be a finitely generated pro-$p$ group.  The action of $\Aut(G)$ induces an action on the finite characteristic image $G/\Phi(G)$, which we regard as a vector space $V$ over $\bF_p$.  More specifically, there is a natural map $\alpha_G$ from $\Aut(G)$ onto $\Out(G)$, and then $\beta_G$ from $\Out(G)$ onto $\Aut(G/\Phi(G))$, which is a subgroup of $\GL(V)$.  (We will write $\alpha_G = \alpha$ and $\beta_G = \beta$ if $G$ is obvious.)  Define $\Delta(G)$ to be the image of $\beta_G$.

In general, $\Delta(G)$ may be significantly smaller than the full general linear group $\GL(V)$; in particular, $\Delta(G)$ may be reducible.  Define the invariant $c(G)$ as follows:

$c(G)$ is the supremum of $\log_p|H:K|$, over all pairs of characteristic subgroups $(H,K)$ of $G$ such that $H \geq K \geq \Phi(G)$ and there are no characteristic subgroups of $G$ lying between $H$ and $K$.

Note that this is equivalently the largest dimension of an irreducible constituent of $V$, regarded as a $\Delta(G)$-module.

If $G$ is a pro-$p$ group that is not finitely generated, we define $c(G)=d(G)$.\end{defn}

The following lemma illustrates the significance of $\Delta(G)$ and $c(G)$ for coprime action.

\begin{lem}\label{deltalem}Let $G$ be a finitely generated pro-$p$ group, with $c(G)=c$.
\vspace{-12pt}
\begin{enumerate}[(i)]  \itemsep0pt
\item The kernels of $\alpha_G$ and $\beta_G$ are pro-$p$ groups.
\item Let $H$ be a profinite group of automorphisms of $G$.  Then $O^{(c,p)}(H)$ is a pro-$p$ group.
\item Let $K$ be a profinite group such that $C_K(G) \leq G \unlhd K$.  Then $O^{(c,p)}(K) \leq O_p(K)$.  If $G = O_p(K)$, then $K/G \lesssim \Delta(G)$.\end{enumerate}\end{lem}

\begin{proof}(i) This is true by definition in the case of $\alpha_G$, and follows immediately from Theorem \ref{cinvthm} in the case of $\beta_G$.

(ii) By considering the action of $H$ on all the finite characteristic images of $G$, we may assume that $G$ is finite.  Consider the action of $H$ on an $H$-invariant normal series for $G/\Phi(G)$; by refining as necessary, we can ensure that no term in this series has rank exceeding $c$.  On each factor, $H$ must act as a subgroup of $\GL(c,p)$, giving a homomorphism from $H$ to a direct product of copies of $\GL(c,p)$.  By the theorem, the kernel of this homomorphism is a pro-$p$ group.

(iii) By (ii), $O^{(c,p)}(K/C_K(G))$ is a normal pro-$p$ subgroup of $K/C_K(G)$; hence $O^{(c,p)}(K)$ is a normal pro-$p$ subgroup of $K$.  If $G = O_p(K)$, then $K/G \lesssim \Out(G)$, so $K/G \lesssim \Delta(G)$ by part (i).\end{proof}

The following is now immediate, given Corollary \ref{malcor}:

\begin{cor}\label{dbebcor}Let $G$ be a finitely generated pro-$p$ group with $c(G) = c$.  Let $H$ be a prosoluble group of automorphisms of $G$.  Let $K = H^{\eb(c)}H^{(\db(c))}$.  Then $K'$ is a pro-$p$ group.\end{cor}

Define $c^\le(G)$ to be the supremum of $c(H)$ as $H$ ranges over all open subgroups of $G$.  The property of having finite $c^\le$-invariant is a generalisation of finite rank, as clearly $c(G) \leq d(G)$, so that $c^\le(G) \leq r(G)$.  On the other hand, it is easy to construct examples where $c^\le(G) < r(G)$: for instance, if $G = \bZ_p \times C_p$ then $c(G) = c^\le(G) = 1$, whereas $r(G) = 2$.  Indeed, for the Nottingham group $N$ over the field of $p$ elements for $p$ at least $5$, it follows from Theorem \ref{commnottgp} that given any open subgroup $U$ of $N$, then $U \cap N_i$ is characteristic in $U$ for any congruence subgroup $N_i$ of $N$, and so $c^\le(N)=1$; at the same time, $N$ is of infinite rank, and indeed $N$ is not even linear.

In general it is very difficult to calculate the commensurator of a pro-$p$ group of infinite rank.  The remainder of this section is therefore devoted to potential methods for establishing finiteness of $c^\le(G)$ based on relatively limited information about the structure of $G$.

Rather than trying to find specific characteristic subgroups, it is more useful to work with descending chains of characteristic subgroups.

\begin{defn}\label{cchaindef}Let $G$ be a pro-$p$ group.  A \emph{$c$-chain of width $w$} for $G$ is a descending chain of characteristic subgroups $W_1 > W_2 > \dots$ with $\bigcap W_i = W \leq \Phi(G)$, satisfying the following conditions:
\vspace{-12pt}
\begin{enumerate}[(i)]  \itemsep0pt
\item $|W_i:W_{i+1}| \leq p^w$ for all $i \geq 1$, but $W_1$ may be of arbitrary index in $G$;
\item $Z(W_1/W) = 1$.\end{enumerate}
Let $[c]_p(w)$ denote the class of pro-$p$ groups with a $c$-chain of width $w$.\end{defn}

Note that in the above definition, $W_1$ is allowed to have arbitrary index.  This gives some flexibility in exhibiting a $c$-chain, but does not significantly weaken the conclusions that can be drawn, as will be seen from the results below.

\begin{lem}Let $G$ be a pro-$p$ group.  Suppose there are characteristic subgroups $K$ and $L$ of $G$, such that $L \leq K \cap \Phi(G)$ and such that $K/L$ is a $[c]_p(w)$-group.  Then $G$ is a $[c]_p(w)$-group.\end{lem}

\begin{proof}Let $W_1/L > W_2/L > \dots$ be a $c$-chain of width $w$ for $K/L$.  It is clear that $W_1 > W_2 > \dots$ is a $c$-chain of width $w$ for $G$.\end{proof}

We can `pull up' $c$-chains using centralisers, possibly at the cost of increasing the width.

\begin{prop}Let $G$ be a finitely generated pro-$p$ group, let $W_1 > W_2 > \dots$ be a $c$-chain of width $w$ for $G$, and let $W = \bigcap W_i$.
\vspace{-12pt}
\begin{enumerate}[(i)]  \itemsep0pt
\item Let $K$ be a characteristic subgroup of $G$ that is not contained in $W$.  Then there is a proper subgroup $L$ of $K$ such that $L$ is characteristic in $G$ and $|K:L| \leq p^{w^2}$.
\item We have $c(G) \leq w^2$.\end{enumerate}\end{prop}

\begin{proof}(i) Consider subgroups $K_{i,j}$ of $G$, defined by $K_{i,j} = C_G(W_i/W_{i+j})$.  Note that $K_{i,j}$ is an open characteristic subgroup of $G$ for every $i$ and $j$.  Condition (ii) of Definition \ref{cchaindef} ensures that $\bigcap_{i,j \in \bN} K_{i,j} = W$.  Suppose that $j$ is minimal such that $K \not\le K_{i,j}$ for some $i$, with the minimum taken over all possible $i$, and let $L$ be the characteristic subgroup $K \cap K_{i,j}$.  It follows from our choice of $j$ that $K$ centralises both $W_{i+1}/W_{i+j}$ and $W_i/W_{i+j-1}$.  Hence the action of $K$ on $W_i/W_{i+j}$ is determined entirely by considering the images under elements of $K$ of a set $X$ of elements of $W_i$, whose images modulo $W_{i+1}$ generate $W_i/W_{i+1}$.  We may assume $|X| \leq w$, and given $x \in X$, all images of $x$ under the action of $K$ must lie inside $xW_{i+j-1}$, which leaves at most $p^w$ possibilities modulo $W_{i+j}$.  It follows that $|K:L|$ is at most $p^{w^2}$.
 
(ii) Using part (i) repeatedly, and setting $G_\lambda = \bigcap_{\alpha < \lambda} G_\alpha$ for limit ordinals $\lambda$, there is a transfinite descending chain $G = G_1 > G_2 > G_3 > \dots$ such that $|G_\beta : G_{\beta+1}| \leq p^{w^2}$ for all ordinals $\beta$, with the chain eventually terminating at $G_\alpha = W$ for some ordinal $\alpha \leq \omega_1$, where $\omega_1$ is the least uncountable ordinal.  The subgroups $G_\beta \Phi(G)$ for $\beta \leq \alpha$ give a series for $G/\Phi(G)$ as an $\Aut(G)$-module in which each factor has rank at most $w^2$; this can be made into a finite series by removing redundant terms, since $G/\Phi(G)$ is finite.  The result now follows by the Jordan-H\"{o}lder theorem.\end{proof}

\begin{cor}\label{heredch}Let $G$ be a finitely generated pro-$p$ group.  Suppose $G$ has a descending chain of characteristic subgroups $W_1 > W_2 > \dots$, satisfying the following conditions:
\vspace{-12pt}
\begin{enumerate}[(i)]  \itemsep0pt
\item $\bigcap W_i = 1$;
\item $|W_i:W_{i+1}| \leq p^w$ for all $i \geq 1$;
\item $Z(W_i)=1$ for all $i$;
\item for every open subgroup $U$ of $G$, all but finitely many $W_i$ are characteristic in $U$.\end{enumerate}
Then $c^\le(G) \leq w^2$.\end{cor}

\begin{proof}Let $U$ be an open subgroup of $G$, and suppose that $W_i$ is characteristic in $U$ for all $i \geq j$.  Then $W_j > W_{j+1} > \dots$ is a $c$-chain for $U$ of width $w$, so $c(U) \leq w^2$ by part (ii) of the proposition.\end{proof}

By Lemma \ref{obchain}, condition (i) of Corollary \ref{heredch} is enough to ensure that every open subgroup contains all but finitely many $W_i$.  Hence to obtain a bound for $c^\le(G)$ it suffices to find an integer $w$ for which there is a descending chain of hereditarily characteristic subgroups $W_i$ of $G$ such that $|W_i:W_{i+1}| \leq p^w$ for all $i$, and such that $\bigcap W_i = 1$.

\chapter{Just infinite groups}

\section{Introduction}\label{jiintro}

In this chapter, we will be concerned with profinite groups for the most part, but some of the results apply equally to other topological groups that may be regarded as just infinite, including discrete groups.  We thus define the just infinite property in a more general context.

\begin{defn}Say $G$ is \emph{just infinite} if it is infinite and residually finite, and every non-trivial normal subgroup of $G$ is of finite index.  Say $G$ is \emph{hereditarily just infinite} if every finite index subgroup of $G$ is just infinite, including $G$ itself.\end{defn}

We recall J.S. Wilson's theory of structure lattices, of which an excellent account is given in \cite{WilNH}; this theory applies to all residually finite just infinite groups that are not virtually abelian.  Given two subnormal subgroups $H$ and $K$ of a just infinite group $G$, say $H$ and $K$ are equivalent if $H \cap K$ has finite index in both $H$ and $K$.  Let $\mcL$ be the set of equivalence classes of subnormal subgroups of $G$.  Following Wilson, we distinguish between the following three \emph{structure types}:

Say $G$ is of structure type $\ra$ if it is virtually abelian.
 
Say $G$ is of structure type $\rh$ if $\mcL$ is finite, but $G$ is not virtually abelian.  It is shown in \cite{WilNH} that this is the case if and only if there exists $N \unlhd_f G$ such that $N$ is the direct product of finitely many conjugates of a hereditarily just infinite profinite group $L$ that is not virtually abelian.  If in fact $L=N=G$, say $G$ is of type $\rhp$.

Say $G$ is of type $\ri$ if $\mcL$ is infinite and $G$ is not virtually abelian.  It was proved by Grigorchuk that all discrete and profinite just infinite groups of type $\ri$ are branch groups, in the sense described below.  See \cite{GriNH} for a more detailed account, and for constructions of such groups (including the group now generally known as the profinite Grigorchuk group).

\begin{defn}A \emph{rooted tree} $T$ is a tree with a distinguished vertex, labelled $\emptyset$.  We require each vertex to have finite degree, though the tree itself will be infinite in general.  The \emph{norm} $|u|$ of a vertex $u$ is the distance from $\emptyset$ to $u$; the \emph{$n$-th layer} is the set of vertices of norm $n$.  Denote by $T_{[n]}$ the subtree of $T$ induced by the vertices of norm at most $n$; by our assumptions, $T_{[n]}$ is finite for every $n$.  Write $\Aut(T)$ for the (abstract) group of graph automorphisms of $T$ that fix $\emptyset$.  Then $\Aut(T)$ also preserves the norm, and so there are natural homomorphisms from $\Aut(T)$ to $\Aut(T_{[n]})$, with kernel denoted $\St_{\Aut(T)}(n)$, the \emph{$n$-th level stabiliser}.  Declare the level stabilisers to be open; this generates a topology on $\Aut(T)$, turning $\Aut(T)$ into a profinite group.\end{defn}

\begin{defn}Let $G$ be a closed or abstract subgroup of $\Aut(T)$.  Then $G$ is said to act \emph{spherically transitively} if it acts transitively on each layer.  Given a vertex $v$, write $T_v$ for the rooted tree with root $v$ induced by the vertices descending from $v$ in $T$.  Define $U^G_v$ to be the group of automorphisms of $T_v$ induced by the stabiliser of $v$ in $G$, and define $L^G_v$ to be the subgroup of $G$ that fixes $v$ and every vertex of $T$ outside $T_v$.  Note that if $G$ acts spherically transitively, the isomorphism types of $U^G_v$ and $L^G_v$ depend only on the norm of $v$; also, there are natural embeddings
\[ L^G_{v_1} \times \dots \times L^G_{v_k} \leq \St_G(n) \leq U^G_{[n]} := U^G_{v_1} \times \dots \times U^G_{v_k},\]
where $v_1, \dots, v_k$ are all the vertices at level $n$.  Now $G$ is a \emph{branch group} if $G$ acts spherically transitively and $|U^G_{[n]} : L^G_{v_1} \times \dots \times L^G_{v_k}|$ is finite for all $n$.  Say $G$ is \emph{self-reproducing at $v$} if there is an isomorphism from $T$ to $T_v$ that induces an isomorphism from $G$ to $U^G_v$.  (The definition of \emph{self-reproducing} given in \cite{GriNH} is that this should hold at every vertex.)\end{defn}

Now let $G$ be a just infinite profinite group.  We distinguish between the following two cases:

Say $G$ is of Sylow type $\rN$ if $G$ is virtually pro-$p$ for some $p$.

Say $G$ is of Sylow type $\rX$ otherwise.

It will turn out that having Sylow type $\rN$ corresponds exactly to the property of having finitely many maximal subgroups.

In the profinite case, we combine the Sylow and structure types to divide the just infinite groups into five mutually disjoint classes:
\[ \rNa, \rNh, \rNi, \rXi, \rXh.\]
The class $\rXa$ is empty and hence omitted.  This is of necessity a crude partition, but it will suffice for the kind of general results under consideration in this chapter.

Most of the published literature to date has been on just infinite pro-$p$ groups, and hence concerns only Sylow type $\rN$.

Of our five classes, the best understood is $\rNa$.  The pro-$p$ groups in this class are known as the irreducible $p$-adic space groups, and an extensive theory of these was developed in the study of pro-$p$ groups of finite coclass, a project initiated by Leedham-Green and Newman in \cite{Lee}.

Next is $\rNh$.  Note that any $\rNh$-group is virtually the direct product of finitely many copies of a hereditarily just infinite pro-$p$ group, so for this class it suffices for most purposes to consider hereditarily just infinite pro-$p$ groups.  Between them, classes $\rNa$ and $\rNh$ include all just infinite virtually pro-$p$ groups of finite rank.  There is a well-developed theory of (virtually) pro-$p$ groups of finite rank, which are also known as compact $p$-adic analytic groups: see \cite{DDMS} for a detailed account, and \cite{Kla} for the beginnings of a classification of the just infinite pro-$p$ groups of this type.  There are also well-studied examples of groups in $\rNh$ of infinite rank, most notably the Nottingham group and some of its generalisations, but as a whole the class of infinite-rank $\rNh$-groups is not all that well understood at present.

The classes $\rNi$ and $\rXi$ can be studied together using general methods for branch groups, such as those pioneered by Grigorchuk.  Nevertheless, even just infinite branch pro-$p$ groups are already considerably more wild in general than the classes $\rNa$ and $\rNh$.

Finally, the class $\rXh$ seems deeply mysterious at present, and until recently it was not known whether or not this class is empty; this question has been resolved by some recent constructions by J.S.~Wilson (unpublished at the time of writing) of hereditarily just infinite profinite groups that are not virtually pronilpotent.  As far as the author is aware, the most important theorems to date concerning this class are the general results of \cite{WilNH}.  For this chapter, results that apply to this class are therefore of particular interest.

\section{Preliminaries}

In this chapter, we will make significant use of the definitions and results of Section \ref{sgplat}.  Here are some further basic results that will be used later in the chapter.

\begin{thm}[Schur \cite{Sch}]\label{cenbyfin} Let $G$ be a group in which $Z(G)$ has finite index.  Then $G'$ is finite.\end{thm}

\begin{cor}\label{abejicor}Let $G$ be a just infinite group, and let $H$ be a normal subgroup.  Suppose $C_G(H) > 1$.  Then $C_G(H)$ is an abelian normal subgroup of $G$ of finite index.  In particular, the virtual centre of a just infinite group $G$ is non-trivial if and only if $G$ is virtually abelian.\end{cor}

\begin{proof}It is clear that $C_G(H)$ is normal; it therefore has finite index in $G$.  Now $H \cap C_G(H)$ has finite index in $C_G(H)$, which means that $C_G(H)$ is centre-by-finite; by Theorem \ref{cenbyfin}, $C_G(H)$ is therefore finite-by-abelian.  This ensures $(C_G(H))'$ is a finite, and hence trivial, normal subgroup of $G$, so  $C_G(H)$ is abelian\end{proof}

\begin{lem}\label{nothji}Let $G$ be a residually finite group with $\Fin(G)=1$, and let $H$ be a finite index subgroup of $G$.  Then $\Fin(H)=1$.  If $H$ is just infinite then every subgroup of $G$ containing $H$ is just infinite, and if $H$ is hereditarily just infinite then $G$ is hereditarily just infinite.\end{lem}

\begin{proof}By Corollary \ref{diccor}, $\Fin(H) \leq \Fin(G) = 1$.  We may now assume that $G$ is not hereditarily just infinite.  Then there is a subgroup $L$ of $G$ of finite index, with a non-trivial normal subgroup $K$ of infinite index; note that $K$ is necessarily infinite as $\Fin(L)=1$.  This gives a non-trivial normal subgroup $K \cap H$ of $L \cap H$ of infinite index; in other words, $L \cap H$ is not just infinite.  As $L \cap H$ has finite index in $H$, this means that $H$ is not hereditarily just infinite.  If $M$ is any subgroup of $G$ containing $H$ that is not just infinite, we can take $L=M$, so that $H=L \cap H$; this means $H$ is not just infinite.\end{proof}

Recall (Section \ref{sgplat}) that we define $\Phi^\lhd(G)$ for any profinite group $G$ to be the intersection of all maximal normal subgroups of $G$, and say $G \in \nsgp^*_\Phi$ if $H/\Phi^\lhd(H)$ is finite for every $H \leq_o G$.

\begin{lem}\label{jiphilhd}Let $G$ be a just infinite profinite group.  Then $G \in \nsgp^*_\Phi$.\end{lem}

\begin{proof}By Proposition \ref{philhdprop}, it suffices to consider an open normal subgroup $H$ of $G$.  This ensures $\Fin(H)=1$, so $\Phi^\lhd(H)>1$ by Proposition \ref{philhdprop}.  This means that $\Phi^\lhd(H)$ is of finite index in $H$, since it is characteristic in $H$ and hence normal in $G$.\end{proof}

\begin{rem}Zalesskii proves in \cite{Zal} that any infinite group in $\nsgp^*_\Phi$ has a just infinite image.  So the class of profinite groups that have a just infinite image is the same as the class of groups that have an infinite image in $\nsgp^*_\Phi$.\end{rem}

Here are some basic properties of groups of structure type $\ra$.

\begin{prop}Let $G$ be a just infinite group that is virtually abelian.  Then $G$ has an abelian normal subgroup $A$ of finite index, such that $A$ is self-centralising, torsion-free and finitely generated.\end{prop}

\begin{proof}By definition, $G$ has an abelian normal subgroup of finite index; choose $A$ to be of least index.  Then $C_G(A)$ abelian by Corollary \ref{abejicor}, so $C_G(A)=A$.

Let $x \in A \setminus 1$, and let $K= \langle x \rangle^G$.  If $x$ has finite order, then $|\langle x \rangle^G|\leq |x|^{|G:A|}$, which contradicts the fact that $\Fin(G)=1$; so $x$ must have infinite order, in other words $A$ is torsion-free.  Furthermore, $K$ is clearly finitely generated and of finite index in $A$, so $A$ is finitely generated.\end{proof}

\begin{cor}Let $G$ be a just infinite group that is virtually abelian.  Then $r(G)$ is finite.\end{cor}

In the profinite and discrete cases, there is a more concrete description.  The following is well-known; see \cite{McC} for a more detailed discussion of the discrete case.

\begin{prop}\label{vastr}Let $G$ be an infinite virtually abelian group that is either discrete or profinite.  Then $G$ is just infinite if and only if the following conditions are satisfied:
\vspace{-12pt}
\begin{enumerate}[(i)]  \itemsep0pt
\item there is a self-centralising normal subgroup $A$ of $G$ of finite index, such that $A \cong O^d$ for some integer $d$, where $O=\bZ$ in the discrete case and $O=\bZ_p$ for some $p$ in the profinite case;
\item $G/A$ acts on $A$ as matrices over $O$;
\item $G/A$ is irreducible as a matrix group over $F$, where $F$ is the field of fractions of $O$.\end{enumerate}\end{prop}

\section{Finite index subgroups and the just infinite property}

We now consider the problem of determining whether a finite index subgroup $H$ of a just infinite group $G$ is itself just infinite.  There are two cases to consider separately here; the case of $G$ virtually abelian, and the case of $G$ not virtually abelian.

We begin with groups that are not virtually abelian, and make use of some ideas from \cite{WilNH}.

\begin{defn}A non-trivial subgroup $B$ of a group $G$ is \emph{basal} if 
\[ B^G = B_1 \times \dots \times B_n\]
for some $n \in \bN$, where $B_1, \dots, B_n$ are the conjugates of $B$ in $G$.  If $B$ is a basal subgroup of $G$, write $\Omega_B$ for the set of conjugates of $B$ in $G$, equipped with the conjugation action of $G$.  (Unless otherwise stated, we will always define $\Omega_B$ in terms of the group $G$.)\end{defn}

\begin{lem}\label{baslem}Let $G$ be a just infinite group that is not virtually abelian.  Let $K$ be a non-trivial subgroup of $G$ such that $K \unlhd^2 G$, and let $\{K_i \mid i \in I\}$ be the set of conjugates of $K$ in $G$; given $J \subseteq I$, define $K_J := \bigcap \{K_j  \mid j \in J\}$.  Then $I$ is finite, and there is some $J \subseteq I$ such that $K_J$ is basal; moreover, $K_J$ is basal for any $J \subseteq I$ such that $K_J > 1$ and the distinct conjugates of $K_J$ have trivial intersection.\end{lem}

\begin{proof}Clearly $K \unlhd K^G$ and $K^G$ has finite index in $G$, so $I$ is finite.  Let $J \subseteq I$ such that $K_J > 1$ and distinct conjugates of $K_J$ have trivial intersection, let $B = K_J$ and let $L = B^G$.  Then all conjugates of $B$ are normal in $B^G$, so in particular they normalise each other; this implies distinct conjugates of $B$ commute, since they have trivial intersection.  Hence $C_L(B)$ contains all conjugates of $B$ apart from $B$ itself, so $L=BC_L(B)$.  Finally, note that $G$ has no non-trivial abelian normal subgroups, so $Z(L)=1$; hence $L=B \times C_L(B)$.  By symmetry, this means $L$ is a direct product of the conjugates of $B$, so $B$ is basal.

Now let $\mcI$ be the set of those $I' \subseteq I$ for which $K_{I'}$ is non-trivial, let $J$ be an element of $\mcI$ of largest size, and let $B=K_J$.  Suppose $B^g$ is a conjugate of $B$ distinct from $B$.  Then $B^g$ is of the form $K_{J'}$ where $|J'|=|J|$; by construction, this means $B \cap B^g = 1$, so by the previous argument, $B$ is basal.\end{proof}

\begin{cor}\label{bascor}Let $G$ be a just infinite group that is not virtually abelian.  Let $K$ be a non-trivial subgroup of $G$ such that $K \unlhd^2 G$, and such that distinct conjugates of $K$ have trivial intersection.  Then $K$ is basal in $G$.\end{cor}

\begin{prop}\label{permji}Let $G$ be a just infinite group that is not virtually abelian, and let $H$ be a subgroup of $G$ of finite index.  Let $\mcB$ be the set of non-normal basal subgroups $B$ of $G$.  Then the following are equivalent:
\vspace{-12pt}
\begin{enumerate}[(i)]  \itemsep0pt
\item $H$ is not just infinite;
\item there is some $B \in \mcB$ such that $H$ acts intransitively on $\Omega_B$.\end{enumerate}
Moreover, if (ii) holds then $B$ may be chosen so that $\Core_G(H)$ acts trivially.\end{prop}

\begin{proof}Assume (i), and let $R$ be a non-trivial normal subgroup of $H$ of infinite index.  By Lemma \ref{nothji}, $R$ is infinite.  This means that $K=R \cap \Core_G(H)$ is an infinite normal subgroup of $H$ of infinite index such that $K \unlhd^2 G$.  By Lemma \ref{baslem}, there is a basal subgroup $B$ of $G$ that is an intersection of conjugates of $K$; by conjugating in $G$ if necessary, we may assume $B \leq K$.  Now $\Core_G(H)$ normalises $K$ and hence every conjugate of $K$, so $\mathrm{Core}_G(H)$ acts trivially on $\Omega_B$.  Furthermore, not all conjugates of $B$ are contained in $K$, as $K$ has infinite index, but the conjugates of $B$ contained in $K$ form a union of $H$-orbits on $\Omega_B$, which is non-empty as $B \leq K$.  So $H$ acts intransitively on $\Omega_B$ as required for (ii).

Assume (ii), and let $R=B^H$.  Then $R$ is an infinite subgroup of $B^G$ of infinite index, since $H$ acts intransitively on $\Omega_B$, so $R \cap H$ is an infinite subgroup of $H$ of infinite index; moreover, $R \cap H$ is normal in $H$.  Hence $H$ is not just infinite, which is (i).\end{proof}

We now arrive at the following theorem.

\begin{thm}\label{maxcor}Let $G$ be a just infinite group that is not virtually abelian.  Let $N$ be a non-trivial normal subgroup of $G$.  The following are equivalent:
\vspace{-12pt}
\begin{enumerate}[(i)]  \itemsep0pt
\item $N$ is just infinite;
\item Every subgroup of $G$ containing $N$ is just infinite;
\item Every maximal subgroup of $G$ containing $N$ is just infinite.\end{enumerate}
$(*)$ In particular, $G$ is hereditarily just infinite if and only if every maximal subgroup of finite index is just infinite.\end{thm}

\begin{proof}Lemma \ref{nothji} ensures that (i) implies (ii), and clearly (ii) implies (iii).  Assume (iii), and let $B$ be a non-normal basal subgroup of $G$, and let $\mcM$ be the set of maximal subgroups of $G$ containing $N$.  Then $M$ acts transitively on $\Omega_B$ for every $M \in \mcM$ by Proposition \ref{permji}.  In any permutation group, a proper transitive subgroup cannot contain a point stabiliser.  It follows that $N_G(B)$ is not contained in any $M \in \mcM$, so $N_G(B)$ does not contain $N$.  Hence $N$ acts non-trivially on $\Omega_B$.  Since $N=\Core_G(N)$, Proposition \ref{permji} now ensures that $N$ is just infinite, giving (i).\end{proof}

\begin{cor}Let $G$ be a residually finite group such that $\Phi^f(G)$ has finite index in $G$, and such that $G$ is not virtually abelian.  Then $G$ is hereditarily just infinite if and only if $\Fin(G)$ is trivial and $\Phi^f(G)$ has a just infinite subgroup of finite index.\end{cor}

\begin{proof}If $\Fin(G)$ is trivial and $\Phi^f(G)$ has a just infinite subgroup of finite index, then every maximal finite index subgroup is just infinite by Lemma \ref{nothji}, and so $G$ is hereditarily just infinite by the theorem.  The converse is immediate.\end{proof}

The following example shows that the word `normal' in the statement of Theorem \ref{maxcor} cannot be replaced with `finite index', even in the case that $G$ is a pro-$p$ group.

\begin{eg}Let $A$ be the group $V \rtimes C_p$, where $V$ is a vector space of dimension $p$ over $\bF_p$, and $C_p$ acts by permuting a basis of $V$.  There is a natural affine action of $A$ on $V$, extending the right regular action of $V$ on itself.  Let $K$ be any just infinite group that is not virtually abelian, and let $G$ be the permutation wreath product $K \wr_V A$ of $K$ by $A$ acting on $V$, where $A$ acts in the way described.  Note that any group of the form $K \wr_\Omega P$, where $K$ is just infinite and $P$ acts faithfully and transitively on the finite set $\Omega$, is necessarily just infinite.  In particular, $G$ is just infinite, with a subgroup $H$ of index $p^2$ of the form $(K \times \dots \times K) \rtimes W$, where $W$ is a subgroup of $V$ of index $p$ that is not normal in $A$.  Clearly $H$ is not just infinite, as $W$ does not act transitively on the copies of $K$.  However, the unique maximal subgroup $M$ of $G$ containing $H$ is of the form $K \wr_V V$; since $V$ acts regularly on itself, this $M$ is just infinite.\end{eg}

Now consider virtually abelian groups.  For simplicity, we will only consider groups that are either discrete or profinite.

A key observation in establishing Theorem \ref{maxcor} was that every intransitive normal subgroup of a finite permutation group is contained in an intransitive maximal subgroup.  Similarly, an imprimitive finite linear group has a maximal subgroup that is reducible, which leads to the following:

\begin{lem}\label{imprim}In the situation of Proposition \ref{vastr}, suppose that $G/A$ is imprimitive as a matrix group over $F$.  Then $G$ has a maximal subgroup of finite index that is not just infinite.\end{lem}

However, there are primitive finite linear groups, all of whose maximal subgroups are irreducible, and so statement $(*)$ does not generalise completely to the virtually abelian case.

\begin{eg}(My thanks go to Charles Leedham-Green for pointing out this example.)  Let $\rQ_{2^n}$ denote the generalised quaternion group of order $2^n$.  In Examples 10.1.18 of \cite{Lee}, the authors give a $4$-dimensional representation of $\rQ_{16}$ over $\bQ_2$, giving rise to just infinite pro-$2$ groups $G$ of the form $A \cdotp \rQ_{16}$ where $A \cong \bZ^4_2$ and $G/A$ acts faithfully on $A$.  Say such a group $G$ is of \emph{primitive quaternionic type}.  Both $\rQ_{16}$ and all its maximal subgroups act irreducibly over $\bQ_2$, and so every maximal open subgroup of $G$ is just infinite.  However, $\Phi^f(G)$ is not just infinite; indeed, an open subgroup $K$ of $G$ is just infinite if and only if $|G/A:KA/A| \leq 2$.\end{eg}

It turns out that the above example describes the only way in which $(*)$ can fail for pro-$p$ groups, and there are no exceptions to $(*)$ among discrete groups for which $O^p(G)=1$.  Case (i) of the following is Theorem 10.1.25 of \cite{Lee}, and case (ii) can easily be deduced from results in section 10.1 of \cite{Lee}:

\begin{thm}[Leedham-Green, McKay \cite{Lee}]\label{leethm}Let $G$ be a finite $p$-group with a faithful primitive representation over a field $F$.
\vspace{-12pt}
\begin{enumerate}[(i)]  \itemsep0pt
\item If $F=\bQ_p$, then either $G$ is $C_p$ acting in dimension $p-1$, or $p=2$ and $G$ is $\rQ_{16}$ acting in dimension $4$.
\item If $F$ is a subfield of $\bR$ that does not contain $\sqrt{2}$, then $G$ is $C_p$ acting in dimension $p-1$.\end{enumerate}\end{thm}

We use this theorem to obtain a result with a conclusion similar to statement $(*)$ of Theorem \ref{maxcor}, but with different hypotheses:

\begin{thm}\label{respthm}Let $G$ be a profinite or discrete group with no non-trivial finite normal subgroups, let $p$ be a prime, and suppose $H$ is a subgroup of finite index that is residually-$p$, such that every maximal subgroup of $H$ of finite index is just infinite and $H$ is not of primitive quaternionic type.  Then $G$ is hereditarily just infinite.\end{thm}

\begin{proof}By Lemma \ref{nothji}, it suffices to prove that $H$ is hereditarily just infinite.  By Theorem \ref{maxcor}, we may assume $H$ is virtually abelian, and hence of the form described in Proposition \ref{vastr}.  Given our hypotheses, the result now follows immediately from Theorem \ref{leethm} together with Lemma \ref{imprim}.\end{proof}

Theorems \ref{maxcor} and \ref{respthm} are particularly useful in the context of pro-$p$ groups, giving the following corollary:

\begin{cor}\label{pcor}Let $G$ be a finitely generated pro-$p$ group.  Then $G$ is hereditarily just infinite if and only $\Phi(G)$ has a just infinite open subgroup and $\Fin(G)=1$.\end{cor}

\begin{proof}If $G$ is hereditarily just infinite, then $\Phi(G)$ is just infinite and $\Fin(G)=1$.  Conversely, suppose $\Fin(G)=1$ and some open subgroup of $\Phi(G)$ is just infinite.  Then $\Phi(G)$ is just infinite by Lemma \ref{nothji}, so $G$ is not of primitive quaternionic type.  Furthermore, every maximal subgroup of $G$ is just infinite, also by Lemma \ref{nothji}.  By Theorem \ref{respthm} it follows that $G$ is hereditarily just infinite.\end{proof}

There is a large variety of primitive finite linear groups that are not $p$-groups, and this gives the potential for just infinite discrete or profinite groups that do not obey $(*)$ in the virtually abelian case.  Here is one example, but it seems likely that there are many others, with no straightforward way of classifying them.
 
\begin{eg}(My thanks go to John Bray for pointing out this example; more details can be found in \cite{Gri}.)  Let $E$ be the extraspecial group of order $2^7$ that is a central product of dihedral groups of order $8$.  Then there is a group $L$ containing $E$ as a normal subgroup with $C_L(E)=Z(E)$, such that $L/E \cong \Out(E) \cong \Alt(8)$, and the smallest supplement to $E$ in $L$ is the double cover $2 \cdotp \Alt(8)$ of $\Alt(8)$.  Furthermore, $L$ has a faithful irreducible $8$-dimensional representation over $\bC$, which can in fact be realised over $\bZ$.  Let $G$ be the corresponding semidirect product $W \rtimes L$, where $W$ is the direct product of $8$ copies of either $\bZ$ or $\bZ_p$ for some $p$; then all maximal subgroups of $G$ of finite index contain a group of the form $W \rtimes E$ or $W \rtimes (2 \cdotp \Alt(8))$.  For both $E$ (see for instance \cite{Doe}) and $2 \cdotp \Alt(8)$ (see \cite{ATL}), the smallest faithful representation in characteristic $0$ has dimension $8$.  It follows that $G$ is a just infinite group, which can be chosen to be either discrete or virtually pro-$p$ with no restrictions on the prime $p$, and every maximal subgroup of $G$ of finite index is just infinite in all cases.  However, $G$ is evidently not hereditarily just infinite.\end{eg}

\section{New just infinite groups from old}

In this section, we will focus on just infinite groups that are not virtually abelian.  In this context, basal subgroups can be used to construct new just infinite groups from old in the following way:

\begin{prop}\label{ncjiprop}Let $G$ be a just infinite group that is not virtually abelian, and let $B$ be a basal subgroup of $G$.  Then $H=N_G(B)/C_G(B)$ is just infinite.\end{prop}

\begin{proof}To show $H$ is just infinite, it suffices to consider a normal subgroup $K$ of $N_G(B)$ of infinite index, such that $C_G(B) \leq K$, and show that in fact $K$ centralises $B$.  Now $K$ and $B$ are both normal in $N_G(B)$, so $[B,K] \leq B \cap K$; hence it suffices to show $B \cap K = 1$.  Set $B=B_1$ and $K=K_1$ say, and let $B_1, \dots, B_t$ be the conjugates of $B$ in $G$.  Since $K$ is normal in $N_G(B)$, we can set $K_j$ to be the conjugate $K^x_1$ of $K$ by $x$ such that $B^x_1 = B_j$, without having to specify $x$, and $K_1, \dots, K_t$ form a complete set of conjugates of $K$ in $G$ (possibly with repetition).  Let $L$ be the intersection of the $K_i$.  Then $L$ is a normal subgroup of $G$ of infinite index, so is trivial, but also $B \cap K \leq L$, since for $i \not= 1$ we have $K_i \geq C_G(B_i) \geq B$.  So $B \cap K = 1$.\end{proof}

In fact, we can use this method to construct a just infinite image of any non-trivial normal subgroup of $G$:

\begin{prop}\label{newjiprop}Let $G$ be a just infinite group that is not virtually abelian, with non-trivial normal subgroup $K$.  Then either $K$ is just infinite, or there is a basal subgroup $B$ of $G$, such that $B$ is a normal subgroup of $K$ and $K/C_K(B)$ is just infinite.\end{prop}

\begin{proof}Among all basal subgroups of $G$ that are normalised by $K$, let $B$ be a basal subgroup with $|G:N_G(B)|$ as large as possible; say $|G:N_G(B)|=t$.  Such a $B$ exists as $|G:K|$ is finite.  As $B \cap K$ is also basal by Corollary \ref{bascor} and $N_G(B) = N_G(B \cap K)$, we may assume $B \leq K$.  Let $Q = KC_G(B)/C_G(B)$ and let $H = N_G(B)/C_G(B)$.  Then $K/C_K(B) \cong Q \unlhd H$, and $H$ is just infinite by the previous proposition.
 
Suppose $Q$ is not just infinite.  Then by Proposition \ref{permji}, $H$ has a non-normal basal subgroup $R$ that is normalised by $Q$; we may assume $R \leq Q$ for the same reason that we could assume $B \leq K$.  Form the subgroup $T$ of $N_G(B)$ by lifting $R$ to $N_G(B)$ and then taking the intersection of this with $B$.  Now $T$ is evidently a normal subgroup of $K$; we have $N_G(T) \leq N_G(B)$ since distinct conjugates of $B$ have trivial intersection and $T \leq B$, and in fact $N_G(T) < N_G(B)$ since $R$ is not normal in $H$.  Let $\{x_i\}$ be a set of right coset representatives for $N_G(T)$ in $N_G(B)$, with $x_1=1$, let $\{y_j\}$ be a set of right coset representatives for $N_G(B)$ in $G$, with $y_1=1$, and let $T_{ij} = T^{x_i y_j}$.

Assume $T_{ij} \cap T_{kl} \not= 1$ for some $(i,j,k,l)$.  Note first that $T_{ij} \leq B^{y_j}$ for all $j$, so $j=l$ by the fact that $B$ is basal in $G$.  Without loss of generality we may assume $j=l=1$.  Then $T_{i1} \cap T_{k1} \leq B$, so $T_{i1} \cap T_{k1} \cap C_G(B) \leq Z(B)$; moreover, $Z(B^G) = 1$ since $G$ is not virtually abelian, so $Z(B)=1$.  Hence $T_{i1} \cap T_{k1}$ maps injectively into $R^{x_i} \cap R^{x_k}$.  Since $R$ is basal in $H$, this forces $i=k$.

We conclude that the conjugates $T_{ij}$ of $T$ in $G$ have pairwise trivial intersection, while $K \leq N_G(T) < N_G(B)$.  This shows that $T$ is basal by Corollary \ref{bascor} and contradicts the choice of $B$.  This contradiction shows that $Q$ is just infinite, so $K/C_K(B)$ is just infinite.\end{proof}

\begin{cor}\label{newjicor}Let $\mcC$ be a class of groups that is closed under quotients.  Let $G$ be a just infinite group that has a non-trivial subnormal $\mcC$-subgroup, but no non-trivial normal abelian subgroup.  Then either $O_{\mcC}(G)$ is already just infinite, or there is a basal subgroup $B$ of $G$ such that $O_{\mcC}(G) \leq N_G(B)$, and such that both $H = N_G(B)/C_G(B)$ and $O = O_\mcC(H)$ are just infinite.\end{cor}

\begin{proof}Let $K = O_{\mcC}(G)$, and assume $K$ is not just infinite.  Then by Proposition \ref{newjiprop}, there is a basal subgroup $B$ of $G$ such that $B$ is a normal subgroup of $K$ and $KC_G(B)/C_G(B)$ is just infinite.  Let $L$ be a subnormal subgroup of $G$ that is a $\mcC$-group.  Then $L \leq K \leq N_G(B)$, and since $\mcC$ is closed under quotients, it follows that $LC_G(B)/C_G(B)$ is a subnormal $\mcC$-group of $H$.  Hence $KC_G(B)/C_G(B) \leq O$.  Finally, note that $H$ is just infinite by Proposition \ref{ncjiprop}, so $\Fin(O)=1$, and $O$ contains a just infinite subgroup $KC_G(B)/C_G(B)$ of finite index, so $O$ is also just infinite by Lemma \ref{nothji}.\end{proof}

\section{Generalised obliquity}

\begin{defn}Given a profinite group $G$ and subgroup $H$, define the \emph{oblique core} $\Ob_G(H)$ and \emph{strong oblique core} $\Ob^*_G(H)$ of $H$ in $G$ as follows:
\[ \Ob_G(H) := H \cap \bigcap \{K \unlhd_o G \mid K \not\le H\}\]
\[ \Ob^*_G(H) := H \cap \bigcap \{K \leq_o G \mid H \le N_G(K), \; K \not\le H\}\]
Note that $\Ob_G(H)$ and $\Ob^*_G(H)$ have finite index in $H$ if and only if the relevant intersections are finite.\end{defn}

The main theorem of this section is the following:

\begin{thm}[Generalised obliquity theorem]\label{genob}Let $G$ be an infinite profinite group.  Then the following are equivalent:
\vspace{-12pt}
\begin{enumerate}[(i)]  \itemsep0pt
\item $G$ is just infinite;
\item the set $\mcK_H = \{K \unlhd_o G \mid K \not\le H\}$ is finite for every open subgroup $H$ of $G$;
\item there exists a family $\mcF$ of open subgroups of $G$ with trivial intersection, such that $\{K \unlhd_o G \mid K \not\le H\}$ is a finite set for every $H \in \mcF$.\end{enumerate}\end{thm}

\begin{proof}Assume (i); then $G \in \nsgp^*_\Phi$ by Lemma \ref{jiphilhd}. Suppose $\mcK_H$ is infinite for some $H \leq_o G$.  Then $\mcK_H$ is upward-closed, so by Corollary \ref{phichaincor} there is an infinite descending chain $K_1 > K_2 > \dots$ of open normal subgroups occurring in $\mcK_H$ for which $K_i \not\le H$.  By Lemma \ref{obchain}, the intersection $K$ of these $K_i$ is a non-trivial normal subgroup of infinite index.  Hence $G$ is not just infinite, a contradiction.  Hence (i) implies (ii).

Clearly (ii) implies (iii).  Assume (iii), and let $K$ be a non-trivial closed normal subgroup of $G$.  Then there is an element $H$ of $\mcF$ that does not contain $K$.  It follows that $K$, being the intersection of the open normal subgroups of $G$ containing $K$, is the intersection of some open normal subgroups not contained in $H$.  All such subgroups contain $\Ob_G(H)$, which is of finite index, since it is the intersection of a finite set of open normal subgroups of $G$.  Hence $K \geq \Ob_G(H)$, and hence $K$ is open in $G$, proving (i).\end{proof}

\begin{rem}This generalises Theorem 36 of \cite{BGJMS}, which corresponds to the above theorem under the assumption that $G$ is a pro-$p$ group.\end{rem}

As motivation for the term `generalised obliquity', recall the following definition:

\begin{defn}[Klaas, Leedham-Green, Plesken \cite{Kla}] Let $G$ be a pro-$p$ group for which each lower central subgroup is an open subgroup. Then the \emph{$i$-th obliquity} of $G$ is given as follows (with the obvious convention that $\log_p(\infty) = \infty$):
\[ o_i(G) := \log_p (|\gamma_{i+1}(G):\Ob_G(\gamma_{i+1}(G))|).\]
The \emph{obliquity} of $G$ is given by $o(G) := \sup_{i \in \bN}o_i(G)$.\end{defn}

It is an immediate consequence of Theorem \ref{genob} that if $G$ is a pro-$p$ group such that each lower central subgroup is an open subgroup, then $G$ is just infinite if and only if $o_i(G)$ is finite for every $i$.

\begin{cor}Let $G$ be a just infinite profinite group.
\vspace{-12pt}
\begin{enumerate}[(i)]  \itemsep0pt
\item Let $n$ be a positive integer.  Then $G$ has finitely many open subgroups of index $n$.
\item Let $H$ be an infinite profinite group such that every finite image of $H$ is isomorphic to some image of $G$.  Then $G \cong H$.
\item Let $H$ and $K$ be proper normal subgroups of $G$ such that $HK=G$.  Then $\Ob_G(\Phi^\lhd(G)) \leq H \cap K$.  In particular, any surjective homomorphism from $G$ to a directly decomposable group must factor through the finite quotient $G/\Ob_G(\Phi^\lhd(G))$.\end{enumerate}\end{cor}

\begin{proof}(i) Since a subgroup of index $n$ has a core of index at most $n!$, and a normal subgroup of finite index can only be contained in finitely many subgroups, it suffices to consider normal subgroups.  Suppose $G$ has an open normal subgroup $K$ of index $n$.  Then $K$ does not contain any open normal subgroup of $G$ of index $n$, other than itself.  Hence by the theorem, the set of such subgroups is finite.

(ii) By part (i), $G$ has finitely many open normal subgroups of any given index.  It is shown in \cite{Wil} that in this situation, given any profinite group $H$ such that every finite image of $H$ is isomorphic to an image of $G$, then $H$ is isomorphic to an image of $G$.  Hence there is some $N \lhd G$ such that $G/N \cong H$; since $H$ is infinite and $G$ is just infinite, $N=1$.

(iii) This is immediate from the definitions and the theorem.\end{proof}

Part (i) of the above is trivial in the case of just infinite virtually pro-$p$ groups, since they are always finitely generated; but a just infinite profinite group need not be finitely generated in general.

Theorem \ref{genob} also gives a characterisation of the hereditarily just infinite property:

\begin{cor}\label{hjicor}Let $G$ be an infinite profinite group.  Then $G$ is hereditarily just infinite if and only if $\Ob^*_G(H)$ has finite index for every open subgroup $H$ of $G$.\end{cor}

\begin{proof}Suppose $G$ has an open subgroup $K$ that is not just infinite.  Then by the theorem, there is an open subgroup $H$ of $K$ (and hence of $G$) that fails to contain infinitely many normal subgroups of $K$, so $\Ob^*_G(H)$ has infinite index.

Conversely, suppose $G$ is hereditarily just infinite.  Let $H$ be an open subgroup of $G$, and let $\mcH$ be the set of subgroups of $G$ containing $H$; then $\mcH$ is finite.  Let $K$ be a subgroup of $G$ such that $H \leq N_G(K)$, but such that $K \not\le H$.  Let $L= HK$.  Then $L$ is just infinite and $K$ is a normal subgroup of $L$ not containing $H$.  Hence $\Ob^*_G(H)$ contains $\bigcap_{L \in \mcH} \Ob_L(H)$; by the theorem each $\Ob_L(H)$ has finite index, and hence $\Ob^*_G(H)$ has finite index.\end{proof}

\section{A quantitative description of the just infinite property}

\begin{defn}Given a profinite group $G$, let $\Ilhd_n(G)$ denote the intersection of all open normal subgroups of $G$ of index at most $n$.  Now define $\OI_n(G)$ to be $\Ob_G(\Ilhd_n(G))$, and $\OI^*_n(G)$ to be $\Ob^*_G(\Ilhd_n(G))$.  We thus obtain functions $\ob_G$ and $\ob^*_G$ from $\bN$ to $\bN \cup \{ \infty \}$ defined by
\[ \ob_G(n) := |G:\OI_n(G)| \quad ; \quad \ob^*_G(n) := |G:\OI^*_n(G)|. \]
These are respectively the \emph{generalised obliquity function} or \emph{$\ob$-function} and the \emph{strong generalised obliquity function} or \emph{$\ob^*$-function} of $G$.  Given a function $\eta$ from $\bN$ to $\bN$, let $\mcO_\eta$ denote the class of profinite groups for which $\ob_G(n) \leq \eta(n)$ for every $n \in \bN$, and let $\mcO^*_\eta$ denote the class of profinite groups for which $\ob^*_G(n) \leq \eta(n)$ for every $n \in \bN$.\end{defn}

These functions give characterisations of the just infinite property and the hereditarily just infinite property in terms of finite images, as will be seen in the next two theorems.

\begin{lem}Let $G$ be a profinite group, and let $n$ be a positive integer.
\vspace{-12pt}
\begin{enumerate}[(i)]  \itemsep0pt
\item Let $N$ be a normal subgroup of $G$.  Then:
\[\Ilhd_n(G)N/N \leq \Ilhd_n(G/N);\]
\[\OI_n(G)N/N \leq \OI_n(G/N);\]
\[\OI^*_n(G)N/N \leq \OI^*_n(G/N).\]
\item Let $\mcN = \{N_i \mid i \in I \}$ be a set of open normal subgroups of $G$ forming a base of neighbourhoods of $1$.  Let $\pi_i$ be the quotient map from $G$ to $G/N_i$.  Then:
\[ \Ilhd_n(G) = \bigcap_{i \in I} \pi^{-1}_i(\Ilhd_n(G/N_i));\]
\[ \OI_n(G) = \bigcap_{i \in I} \pi^{-1}_i(\OI_n(G/N_i));\]
\[ \OI^*_n(G) = \bigcap_{i \in I} \pi^{-1}_i(\OI^*_n(G/N_i)).\]\end{enumerate}\end{lem}

\begin{proof}Let $L = \Ilhd_n(G)$, let $M/N = \Ilhd_n(G/N)$, and let $M_i/N_i = \Ilhd_n(G/N_i)$.

(i) If $H/N$ is a normal subgroup of index at most $n$ in $G/N$, then $H$ also has index at most $n$ in $G$.  This proves the first inequality, in other words $L \leq M$.

If $H/N$ is a normal subgroup of $G/N$ not contained in $M/N$, then $H$ is also not contained in $M$ and hence not in $L$.  This proves the second inequality.

If $H/N$ is a subgroup of $G/N$ that is normalised by $M/N$ but not contained in it, then $H$ is also normalised by but not contained in $M$, and hence also normalised by but not contained in $L$.  This proves the third inequality.

(ii) Given part (i), it suffices to show for each equation that the left-hand side contains the right-hand side.

If $H$ is a normal subgroup of $G$ index at most $n$, then there is some $N_i$ contained in $H$, which means that $M_i$ is contained in $H$, since $H/N_i$ has index at most $n$ in $G/N_i$.  This proves the first equation, in other words $L = \bigcap_{i \in I} M_i$.

If $H$ is a normal subgroup of $G$ not contained in $L$, then there is some $M_i$ that does not contain $H$, by the first equation.  This proves the second equation.

Let $H$ be an open subgroup of $G$ that is normalised by $L$ but not contained in it.  Then $HL$ is an open subgroup of $G$ that contains $\bigcap_{i \in I} M_i$.  By Lemma \ref{obchain}, this means that there is some $M_i$ contained in $HL$, which implies that this $M_i$ normalises $H$.  By the first equation, there is some $M_j$ not containing $H$.  Now take $M_k \leq M_i \cap M_j$, and note that $\Ob^*_{G/N_k}(M_k/N_k)$ is contained in $H$.  This proves the third equation.\end{proof}

\begin{thm}\label{obfnthmi}Let $G$ be a profinite group.  The following are equivalent:
\vspace{-12pt}
\begin{enumerate}[(i)]  \itemsep0pt
\item $G$ is finite or just infinite;
\item There is some $\eta$ for which $G$ is an $\mcO_\eta$-group;
\item There is some $\eta$ for which every image of $G$ is an $\mcO_\eta$-group;
\item There is some $\eta$, and some set of open normal subgroups $\mcN = \{N_i \mid i \in I \}$ of $G$ forming a base of neighbourhoods of $1$, such that each $G/N_i$ is an $\mcO_\eta$-group.\end{enumerate}
Moreover, (ii), (iii) and (iv) are equivalent for any specified $\eta$.\end{thm}

\begin{proof}Clearly (iii) implies both (ii) and (iv).  It is clear from the lemma that (ii) implies (iii), and that (ii) and (iv) are equivalent.  These implications hold for any specified $\eta$.  So it remains to show that (i) and (ii) are equivalent.

Suppose (i) holds.  Then $G$ has finitely many normal subgroups of index $n$ for any integer $n$, so $\Ilhd_n(G)$ has finite index.  It follows by Theorem \ref{genob} that $\Ob_G(\Ilhd_n(G))$ also has finite index, so $\ob_G(n)$ is finite.  This implies (ii) by taking $\eta = \ob_G$.

Suppose (i) is false.  Then by Theorem \ref{genob}, there is an open subgroup $H$ of $G$ such that $\Ob_G(H)$ has infinite index in $G$.  Now $H$ has index $h$ say, so that $\Ilhd_h(G) \leq H$.  It follows that $\OI_h(G)$ must be contained in $\Ob_G(H)$, and so $\ob_G(h) = |G:\OI_h(G)| = \infty$.  This implies that (ii) is also false.\end{proof}

For hereditarily just infinite groups, we have the following:

\begin{thm}\label{obfnthmii}Let $G$ be a profinite group.  The following are equivalent: 
\vspace{-12pt}
\begin{enumerate}[(i)]  \itemsep0pt
\item $G$ is finite or hereditarily just infinite;
\item There is some $\eta$ for which $G$ is an $\mcO^*_\eta$-group;
\item There is some $\eta$ for which every image of $G$ is an $\mcO^*_\eta$-group;
\item There is some $\eta$, and some set of open normal subgroups $\mcN = \{N_i \mid i \in I \}$ of $G$ forming a base of neighbourhoods of $1$, such that each $G/N_i$ is an $\mcO^*_\eta$-group.\end{enumerate}

Moreover, (ii), (iii) and (iv) are equivalent for any specified $\eta$.\end{thm}

\begin{proof}The proof of this theorem is entirely analogous to that of Theorem \ref{obfnthmi}, with $\mcO^*_\eta$ in place of $\mcO_\eta$ and $\ob^*$ in place of $\ob$, and using Corollary \ref{hjicor} in place of Theorem \ref{genob}.\end{proof}

The definition of $\ob$-functions and $\ob^*$-functions leads to the following general question:

\begin{que}Which functions from $\bN$ to $\bN$ can occur as $\ob$-functions or $\ob^*$-functions for (hereditarily) just infinite profinite groups?  What growth rates are possible?\end{que}

\begin{rem}In more specific contexts, the subgroups $\Ilhd_n(G)$ in the definition of (strong) generalised obliquity functions can be replaced with various other characteristic open series, and Theorems \ref{obfnthmi} and \ref{obfnthmii} would remain valid, with essentially the same proof.  For instance, in case of pro-$p$ groups, one could use lower central exponent-$p$ series, and in the case of prosoluble groups with no infinite soluble images, one could use the derived series.\end{rem}

As an illustration, consider a pro-$p$ group $G$ of finite obliquity $o$.  As mentioned in \cite{BGJMS}, this also implies that there is some constant $w$ such that $|\gamma_i(G):\gamma_{i+1}(G)| \leq w$ for all $i$.  It is proved in \cite{BGJMS} that the condition of finite obliquity is equivalent to the following:

There exists a constant $c$ such that for every normal subgroup $N$ of $G$, and for every normal subgroup $M$ not contained in $N$, we have $|N:N \cap M| \leq p^c$.

Lower and upper bounds for $\ob_G$ can easily be derived in terms of these invariants, from which follows a characterisation of the pro-$p$ groups $G$ for which $\ob_G$ is bounded by a linear function.

\begin{prop}\
\vspace{-12pt}
\begin{enumerate}[(i)]  \itemsep0pt 
\item The $\ob$-function of $\bZ_p$ is given by $\ob_{\bZ_p}(n) = p^s$, where $s$ is the largest integer such that $p^s \leq n$.  In particular $\ob_{\bZ_p}(n) \leq n$ for all $n$.
\item Let $G$ be a pro-$p$ group of finite obliquity, with invariants as described above.  Then 
\[ \ob_G(p^e) \leq p^{e+c+w+o-2} \]
for all non-negative integers $e$.  In particular, there is a constant $k$ such that $\ob_G(n) \leq kn$ for all $n$.
\item Let $G$ be a pro-$p$ group for which there is a constant $k$ such that $\ob_G(n) \leq kn$ for all $n$.  Then either $G \cong \bZ_p$, or $G$ has obliquity at most $\log_p(k)$.\end{enumerate}\end{prop}

\begin{proof}(i) This is immediate from the definitions.

(ii) Let $N$ be a normal subgroup of $G$ of index at most $p^n$.  Then $N$ is not properly contained in any lower central subgroup that has index at least that of $N$; the first such, say $\gamma_r(G)$, has index at most $p^{e+w-1}$.  Hence $\Ilhd_{p^e}(G)$ contains $\Ob_G(\gamma_r(G))$, which has index at most $p^{e+w+o-1}$.

Now let $M$ be a normal subgroup of $G$ not contained in $\Ilhd_{p^e}(G)$.  Then there is a normal subgroup $K$ of $G$ of index at most $p^e$ such that $K$ does not contain $M$.  Hence $M$ properly contains a normal subgroup $M \cap K$ of $G$ of index at most $p^{e+c}$.  In particular, $M$ is of index at most $p^{e+c-1}$, so contains $\Ilhd_{p^{e+c-1}}(G)$.  Thus $\OI_{p^e}(G)$ contains $\Ilhd_{p^{e+c-1}}(G)$, a subgroup of index at most $p^{e+c+w+o-2}$.

(iii) By Theorem \ref{obfnthmi}, $G$ is finite or just infinite.  We may assume $G$ is not $\bZ_p$, which ensures that all lower central subgroups are open (see \cite{BGJMS}).  Let $H$ be a lower central subgroup, of index $h$ say.  Then $H$ contains $\Ilhd_{h}(G)$, so $\Ob_G(H)$ contains $\OI_{h}(G)$, which in turn is a subgroup of $G$ of index at most $kh$.  Hence $|H:\Ob_G(H)|$ is at most $k$.\end{proof}

Now consider self-reproducing profinite branch groups, in the sense defined in Section \ref{jiintro}.

\begin{prop}Let $G$ be a just infinite profinite branch group acting on the rooted tree $T$, such that $G$ is self-reproducing at some vertex $v$.  Then there is a constant $c$ such that $\ob_G(n) \leq c^n$ for all $n$.\end{prop}

\begin{proof}Since $G$ is a just infinite profinite group, and the subgroups $\St_G(n)$ are all open in $G$, we can define a function $f$ from $\bN$ to $\bN$ by the property that $\Ob_G(\St_G(n))$ contains $\St_G(n+f(n))$ but not $\St_G(n+f(n)-1)$, for all $n$.  Suppose $|v|=k$, and consider a normal subgroup $K$ of $G$ not contained in $\St_G(k+n)$ for some integer $n$.  If $K$ is not contained in $\St_G(k)$, then it contains $\St_G(k+f(k))$.  Otherwise, there is some vertex $u$ of norm $k$ such that $K$ acts non-trivially on $(T_u)_{n}$; since $G$ is spherically transitive, we may take $u=v$.  This means that $K/\St_K(T_v)$ contains $\Ob_V(\St_V(n))$, where $V = U^G_v$; since $V \cong G$ as groups of tree automorphisms, we have in turn $\Ob_V(\St_V(n)) \geq \St_V(n+f(n))$.  Since $K$ is normal in $G$, it follows that $K$ induces all automorphisms of $T$ occurring in $G$ that fix the layers up to $k+n+f(n)$, and hence $K$ contains $\St_G(k+n+f(n))$.  Thus $\Ob_G(\St_G(k+n))$ contains $\St_G(k+f(k)) \cap \St_G(k+n+f(n))$, which means that $f(k+n) \leq \max \{f(n),f(k)\}$.  By induction on $n$, this implies $f(n) \leq r$ for all $n$, where $r = \max_{1 \leq i \leq k} f(i)$.

Let $N$ be a normal subgroup of index at most $n$, where $n \geq 2$.  Let $l(n)$ be the greatest integer such that $\St_G(l(n))$ has index less than $n$.  Then $N$ is not properly contained in $\St_G(l(n)+1)$, so it contains $\St_G(l(n)+1+f(l(n)+1))$, and hence $\Ob_G(N)$ contains $\St_G(l(n)+1+2f(l(n)+1))$, which in particular contains $\St_G(l(n)+2r+1)$.  Hence $\OI_n(G)$ contains $\St_G(l(n)+2r+1)$.  This means that
\begin{align*}
\ob_G(n) &\leq |G:\St_G(l(n))||\St_G(l(n)):\St_G(l(n)+2r+1)| \\
&\leq n|\St_G(l(n)):\St_G(l(n)+2r+1)|.
\end{align*}
By applying the self-reproducing property of $G$ repeatedly, we obtain an embedding 
\[ \frac{\St_G(l(n))}{\St_G(l(n)+2r+1)} \hookrightarrow \frac{\St_G(t)}{\St_G(t+2r+1)} \times \dots \times \frac{\St_G(t)}{\St_G(t+2r+1)},\]
where $t$ is the integer in the interval $(0,k]$ such that $l(n) \equiv t$ modulo $k$, and the direct factors on the right correspond to the vertices of $T$ of norm $l(n)$ descending from a given vertex of norm $t$.  Since $G$ is spherically transitive, there are less than $n$ vertices of $T$ of norm $l(n)$, so that
\[ \ob_G(n) \leq n(\max_{0 < t \leq k} |\St_G(t):\St_G(t+2r+1)|)^n \]
from which the result follows.\end{proof}

We also consider how the $\ob$-function and $\ob^*$-function of a just infinite profinite group $G$ relate to those of its open normal subgroups.

\begin{lem}\label{obineqlem}Let $G$ be a profinite group, and let $H$ be a subgroup of $G$ of index $h$.  Then:
\vspace{-12pt}
\begin{enumerate}[(i)]  \itemsep0pt
\item $\Ilhd_n(G) \geq \Ilhd_n(H) \geq \Ilhd_{tn^h}(G)$, where $t=|G:\Core_G(H)|$;
\item If $G$ is just infinite, then $\Ilhd_{hn}(G) \geq \Ilhd_n(H)$ for sufficiently large $n$.\end{enumerate}\end{lem}

\begin{proof}
(i) Let $n$ be a positive integer, and let $K$ be a normal subgroup of $G$ of index at most $n$.  Then $H \cap K$ is a normal subgroup of $H$ of index at most $n$.  Hence $\Ilhd_n(G) \geq \Ilhd_n(H)$.  On the other hand, let $L$ be a normal subgroup of $H$ of index at most $n$.  Then $M = L \cap \Core_G(H)$ has index at most $n$ in $\Core_G(H)$, and $M$ has at most $h$ conjugates in $G$, all of which are contained in $\Core_G(H)$, so that $\Core_G(M)$ has index at most $n^h$ in $\Core_G(H)$, and hence index at most $tn^h$ in $G$.  Thus every normal subgroup of $H$ of index at most $n$ contains a normal subgroup of $G$ of index at most $tn^h$, so $\Ilhd_n(H) \geq \Ilhd_{tn^h}(G)$.

(ii) If $G$ is just infinite, there is some integer $m$ such that $H$ contains every normal subgroup of $G$ of index greater than $hm$, by Theorem \ref{genob}.  In addition, $\Ilhd_{hm}(G) \cap H$ is open in $G$, so contains an open normal subgroup $L$ of $G$.  Now let $n$ be any integer at least $|H:L|$.  Then for every normal subgroup $K$ of $G$ of index at most $hn$, then either $K$ is already a normal subgroup of $H$ of index at most $n$, or $K$ contains $\Ilhd_{hm}(G) \cap H$ and hence $K$ contains $L$; in either case, $K$ contains $\Ilhd_n(H)$.  Hence $\Ilhd_{hn}(G)$ contains $\Ilhd_n(H)$ as required.\end{proof}

\begin{thm}\label{obineq}Let $G$ be a just infinite profinite group, and let $H$ be a subgroup of $G$ of index $h$.
\vspace{-12pt}
\begin{enumerate}[(i)] \itemsep0pt
\item The following inequality holds for sufficiently large $n$:
\[ \ob_H(n) \geq h^{-1} \ob_G(hn). \]
\item Let $t=|G:\Core_G(H)|$. The following inequality holds for all $n$:
\[\ob^*_H(n) \leq h^{-1}\ob^*_G(tn^h).\]
\item For a given $n$, let $\mcI_n$ be the set of subgroups of $G$ containing $\Ilhd_n(G)$.  The following inequality holds:
\[\ob^*_G(n) \leq \prod_{L \in \mcI_n} |G:L|\ob_L(n).\]\end{enumerate}\end{thm}

\begin{proof}For part (i), we may assume that $n$ is large enough that $\Ob_G(H) \geq \Ilhd_{hn}(G) \geq \Ilhd_n(H)$, by Lemma \ref{obineqlem}.  The claimed inequalities are demonstrated by the relationships between subgroups given below:

\begin{align*}\OI_n(H) &= \Ilhd_n(H) \cap \bigcap \{ N \unlhd_o H \mid N \not\le \Ilhd_n(H) \} \\
&\leq \Ilhd_{hn}(G) \cap \Ob_G(H) \cap \bigcap \{ N \unlhd_o H \mid N \not\le \Ilhd_{hn}(G) \} \\
&\leq \Ilhd_{hn}(G) \cap \bigcap \{ N \unlhd_o G \mid N \not\le \Ilhd_{hn}(G) \} \\
&= \OI_{hn}(G).\\
\OI^*_n(H) &= \Ilhd_n(H) \cap \bigcap \{ L \leq_o H \mid \Ilhd_n(H) \leq N_H(L), \; L \not\le \Ilhd_n(H) \} \\
&\geq \Ilhd_{tn^h}(G) \cap \bigcap \{ L \leq_o G \mid \Ilhd_{tn^h}(G) \leq N_G(L), \; L \not\le \Ilhd_{tn^h}(G) \} \\
&= \OI^*_{tn^h}(G).\\
\OI^*_n(G) &= \Ilhd_n(G) \cap \bigcap \{ H \leq_o G \mid \Ilhd_n(G) \leq N_G(H), \; H \not\le \Ilhd_n(G) \} \\
&\geq \bigcap_{L \in \mcI_n} \Ilhd_n(L) \cap \bigcap_{L \in \mcI_n} \bigcap \{ H \unlhd_o L \mid H \not\le \Ilhd_n(L) \}\\
&= \bigcap_{L \in \mcI_n} \OI^*_n(L). \qedhere \end{align*}\end{proof}

\section{Isomorphism types of normal sections and open subgroups}

Another consequence of generalised obliquity concerns non-abelian normal sections of a just infinite profinite group $G$.  In sharp contrast to the case of abelian normal sections, a given isomorphism type of non-abelian finite group can occur only finitely many times as a normal section of $G$.  In fact, more can be said here, as will be seen in Theorem \ref{normseclass}.

\begin{prop}Let $G$ be a just infinite profinite group.  Let $M$ and $N$ be open normal subgroups of $G$ such that $N \leq M$, and let $H$ be an open subgroup of $G$, with $\Core_G(H)$ of index $h$.  Then at least one of the following holds:
\vspace{-12pt}
\begin{enumerate}[(i)] \itemsep0pt
\item $M/N$ is abelian;
\item $H$ contains both $M$ and $C_G(M/N)$;
\item $M$ contains the open subgroup $\Ob_G(\Ob_G(H))$, and so $|G:M| \leq \ob_G(\ob_G(h))$.\end{enumerate}\end{prop}
 
\begin{proof}Assume (i) and (ii) are false.  Since $M$ is a normal subgroup of $G$, to demonstrate (iii) it suffices to prove that $M$ is not properly contained in $\Ob_G(H)$.  Let $K=C_G(M/N)$; note that since (i) is false, $K$ does not contain $M$.  If $H$ does not contain $M$, then $M$ contains $\Ob_G(H)$, so we may assume $H$ contains $M$.  It now follows that $H$ does not contain $K$, by the assumption that (ii) is false.  Since $K$ is normal in $G$, it must contain $\Ob_G(H)$, and hence $\Ob_G(H)$ cannot contain $M$.\end{proof}

\begin{defn}Given a profinite group $G$, define $\Phi^{\lhd n}(G)$ by $\Phi^{\lhd 0}(G) = G$ and thereafter $\Phi^{\lhd (n+1)}(G) = \Phi^\lhd(\Phi^{\lhd n}(G))$.  Define the \emph{$\Phi^\lhd$-height} of a finite group to be the least $n$ such that $\Phi^{\lhd n}(G)=1$.\end{defn}

Given $G \in \nsgp^*_\Phi$, note that $\Phi^{\lhd n}(G)=O^\mcX(G)$, where $\mcX$ is the class of finite groups of $\Phi^\lhd$-height at most $n$, and that $O^\mcX(G)$ is therefore an open subgroup of $G$.

\begin{thm}\label{normseclass}Let $G$ be a just infinite profinite group.
\vspace{-12pt}
\begin{enumerate}[(i)] \itemsep0pt
\item Let $\mcF$ be a class of non-abelian finite groups, and let $\mcA$ be the class of groups $A$ satisfying $\Inn(F) \leq A \leq \Aut(F)$ for some $F \in \mcF$.  Suppose there are infinitely many pairs $(M,N)$ of normal subgroups of $G$ such that $N \leq M$ and $M/N \in \mcF$.  Then either $G$ is residually-$\mcA$, or it has an open normal subgroup that is residually-$\mcF$.  In particular, at least one of $\mcA$ and $\mcF$ contains groups of arbitrarily large $\Phi^\lhd$-height, and hence $\mcF$ must contain infinitely many isomorphism classes.
\item Suppose $G$ is not virtually abelian, and let $H$ be a proper open normal subgroup of $G$.  Then $G$ has only finitely many normal subgroups $K$ such that $K/\Phi^\lhd(K') \cong H/\Phi^\lhd(H')$.  In particular, $G$ has only finitely many normal subgroups that are isomorphic to $H$.\end{enumerate}\end{thm}

\begin{proof}(i) We may assume that $G$ is not a residually-$\mcA$ group, so $O^\mcA(G)$ has finite index.  Let $M$ and $N$ be normal subgroups such that $M/N$ is a $\mcF$-group, and let $H=C_G(M/N)$.  Then $G/H$ is a $\mcA$-group, and hence an image of $G/O^\mcA(G)$.  On the other hand, $H$ does not contain $M$.  By the above proposition, this means that $|G:M|$ is bounded by a function of $G$ and $|G/O^\mcA(G)|$, and hence there are only finitely many possibilities for $M$.  This means that for some open normal subgroup $M$, there must be infinitely many images of $M$ that are $\mcF$-groups.  Hence $O^\mcF(M)$ is a normal subgroup of $G$ of infinite index, and hence trivial, so that $M$ is residually-$\mcF$.

(ii) Since $G$ is not virtually abelian, $H'$ is an open normal subgroup of $G$.  Since $G \in \nsgp^*_\Phi$, it follows that $\Phi^\lhd(H')$ is a proper normal subgroup of $H'$ of finite index, so $H/\Phi^\lhd(H')$ is finite and non-abelian.  The result follows by part (i) applied to $\mcF = [H/\Phi^\lhd(H')]$.\end{proof}

\begin{rem}Part (i) of the above theorem does not extend to abelian sections.  Indeed, given any positive integer $n$ and a prime $p$, then any just infinite pro-$p$ group has infinitely many abelian normal sections of order $p^n$; this is clear for $\bZ_p$, and for any non-nilpotent pro-$p$ group $G$ there will be suitable sections inside $\gamma_k(G)/\gamma_{2k}(G)$ for any $k \geq n$.\end{rem}

\begin{cor}\label{vaiso}Let $G$ be a just infinite profinite group.  Then $G$ has infinitely many isomorphism types of open normal subgroup if and only if $G$ is not virtually abelian.\end{cor}

\begin{proof}Suppose $G$ is virtually abelian.  Then $G$ has an open normal subgroup $V$ that is a free abelian pro-$p$ group.  Every open subgroup of $G$ contained in $V$ is isomorphic to $V$, and by Theorem \ref{genob}, all but finitely many open normal subgroups of $G$ are contained in $V$.  Hence $G$ has only finitely many isomorphism types of open normal subgroup.  The converse follows from part (ii) of Theorem \ref{normseclass}.\end{proof}

\begin{defn}Say a profinite group $G$ is \emph{index-unstable} if it has a pair of isomorphic open subgroups of different indices, and \emph{index-stable} otherwise.\end{defn}

Recall the definition given in Section \ref{commsec} of the commensurator $\Comm(G)$ and the index ratio $\varrho$ of $G$, and Lemma \ref{irhom}.  Since the commensurator accounts for all virtual automorphisms of a profinite group $G$ up to equivalence, $G$ is index-stable if and only if its index ratio is trivial.  An understanding of index-stability is thus important for determining the structure of the commensurator of a profinite group, and also that of locally compact, totally disconnected groups in general.

Any just infinite virtually abelian profinite group $G$ is virtually a free abelian pro-$p$ group for some $p$, and hence $G$ is index-unstable.  On the other hand, given part (ii) of Theorem \ref{normseclass} it seems to be difficult to construct just infinite profinite groups that are index-unstable but not virtually abelian.  The remainder of this section is concerned with the following question:

\begin{que}\label{indexq}Let $G$ be a (hereditarily) just infinite profinite group that has isomorphic open subgroups of different indices, or equivalently, such that the index ratio of $G$ is non-trivial.  Is $G$ necessarily virtually abelian?\end{que}

\begin{defn}Let $G$ be a just infinite profinite group that is not virtually abelian, and let $N$ be an open normal subgroup of $G$.  Define the following invariant of $G$:
\[ j_N(G) = \frac{\inf \{|G:M| \mid M \cong N, \; M \unlhd_o G\}}{\inf \{|G:M| \mid M \cong N, \; M \unlhd^2_o G\}}.\]\end{defn}

Clearly, if $G$ is index-stable then $j_N(G)=1$ for all $N \unlhd_o G$.  In fact, there is a strong converse to this statement.

\begin{prop}\label{jbdprop}Let $G$ be a just infinite profinite group.  Suppose that there are infinitely many isomorphism types of open normal subgroup $N$ of $G$ for which $j_N(G) \leq k$, for some constant $k$.  Then $G$ is index-stable.\end{prop}

\begin{proof}By Corollary \ref{vaiso}, $G$ is not virtually abelian.  Suppose $G$ is index-unstable.  Then by Lemmas \ref{irhom} and \ref{subnorlem}, there are isomorphic subgroups $H$ and $K$ of $G$ such that $|G:H|/|G:K| > k$, and such that $H \unlhd^2_o G$ and $K \unlhd_o G$.  Now $H$ contains all but finitely many normal subgroups of $G$, so all but finitely many isomorphism types of open normal subgroups of $G$ occur only as subgroups of $H$.  This means that there is a normal subgroup $N$ of $G$ such that $j_N(G) \leq k$, and such that all normal subgroups of $G$ isomorphic to $N$ are subgroups of $H$; take $N$ to be of least possible index.  Then $N^\theta \unlhd^2_o G$, where $\theta$ is any isomorphism from $H$ to $K$, and $N^\theta$ is isomorphic to $N$.  Since $N$ was chosen to be of least possible index, it follows that $j_N(G) \geq \ir(\theta) > k$, a contradiction.\end{proof}

We conclude this section with the following theorem.

\begin{thm}\label{istabthm}Let $G$ be a just infinite profinite group.
\vspace{-12pt}
\begin{enumerate}[(i)] \itemsep0pt
\item Suppose there is an isomorphism $\theta: H \rightarrow G$, where $H$ is a proper open subgroup of $G$.  Then $G$ is virtually abelian.
\item Suppose $G$ has infinitely many distinct radicals.  Then $G$ is index-stable.\end{enumerate}\end{thm}

\begin{proof}(i) Suppose $G$ is not virtually abelian.  Then $H$ contains all but finitely many normal subgroups of $G$, so all but finitely many isomorphism types of open normal subgroups of $G$ occur only as subgroups of $H$.  This means that there is an open normal subgroup $N$ of $G$ such that all normal subgroups of $G$ isomorphic to $N$ are subgroups of $H$; take $N$ to be of least possible index.  Then $N^\theta$ is not normal in $G$ by the minimality of $|G:N|$.  But $N$ is normal in $H$, and so $N^\theta \lhd H^\theta = G$, a contradiction.

(ii) Let $N=O_\mcX(G)$ for some class of groups $\mcX$, and suppose $N$ is non-trivial.  Let $M$ be a subnormal subgroup of $G$ isomorphic to $N$.  Then by definition, $M$ is generated by its subnormal $\mcX$-subgroups.  But these are then subnormal in $G$, and so contained in $N$.  Hence $M \leq N$, demonstrating that $j_N(G)=1$.  Hence the non-trivial radicals of $G$ form an infinite set of pairwise non-isomorphic open normal subgroups $N$ satisfying $j_N(G)=1$.  By Proposition \ref{jbdprop}, this ensures that $G$ is index-stable.\end{proof}

\begin{rem}Part (ii) of Theorem \ref{istabthm} is not vacuous, as there are certainly just infinite profinite groups that have infinitely many distinct radicals.  For instance, Theorem \ref{commnottgp} gives an example in which every characteristic subgroup is a radical.\end{rem}

\section{Sylow structure of just infinite profinite groups}

Let $G$ be a just infinite profinite group.  What does this tell us about the Sylow structure of $G$?

We obtain some characteristic properties of the Sylow types $\rN$ and $\rX$ for profinite groups.

\begin{lem}\label{phiprim}Let $G$ be a finite group, acting faithfully and primitively on a finite set $\Omega$, and let $H$ be a normal subgroup of $G$.  Then $\Phi(H)=1$.\end{lem}

\begin{proof}Note first that for any normal subgroup $N$ of $G$, the $N$-orbits define a $G$-invariant equivalence relation on $\Omega$, so either $N=1$ or $N$ acts transitively; in particular, since $\Phi(H)$ is normal in $G$ it suffices to show that $\Phi(H)$ acts intransitively.  We may assume $H$ acts transitively, and so $\Omega$ can be regarded as the set of right cosets of some subgroup $K$ of $H$, acted on by right multiplication.  Let $M$ be a maximal subgroup of $H$ containing $K$.  Then $\{Km  \mid m \in M\}$ is an $M$-orbit, so $M$ acts intransitively.  Hence $\Phi(H) \leq M$ also acts intransitively.\end{proof}

\begin{prop}Let $G$ be a just infinite profinite group.
\vspace{-12pt}
\begin{enumerate}[(i)] \itemsep0pt
\item Suppose $G$ is of Sylow type $\rN$.  Then $\Phi(G)$ is an open normal subgroup of $G$, and $G$ is finitely generated.
\item Suppose $G$ is of Sylow type $\rX$.  Then $\Phi(G)=1$.  Every Sylow subgroup of $G$ is either finite or infinitely generated.\end{enumerate}\end{prop}

\begin{proof}(i) It is clear that $G$ has an open normal pro-$p$ subgroup for exactly one prime $p$; let $P$ be such a subgroup.  Now $\Phi(P)=\Phi^\lhd(P)$ has finite index in $G$ by Lemma \ref{jiphilhd}, which means that $P$ is finitely generated, and hence $G$ is also finitely generated.

Let $N$ be the core of a maximal subgroup of $G$ of finite index.  Then $G/N$ admits a faithful primitive permutation action on some set $\Omega$.  This means that $\Phi(O_p(G/N))=1$, by Lemma \ref{phiprim}.  Hence $O_p(G/N)$ is elementary abelian, and so the same is true for its subgroup $O_p(G)N/N$; hence $\Phi(O_p(G)) \leq N$.  As this holds for all cores of maximal finite index subgroups, we have $\Phi(O_p(G)) \leq \Phi(G)$.  Hence $\Phi(G)$ is of finite index in $G$.
 
(ii) By definition, $G$ is not virtually pro-$p$ for any prime $p$, and so cannot have any non-trivial normal pro-$p$ subgroup.  Hence $O_p(G)$ is of infinite index, and hence trivial, for all $\piP$.  By Lemma \ref{fitlem}, this means $G$ has no non-trivial pronilpotent normal subgroups, so $\Phi(G)=1$ by Lemma \ref{fratlem}.
  
Let $S$ be a $p$-Sylow subgroup of $G$ for some $p$, and suppose $S$ is finitely generated.  Then $G/O_{p'}(G)$ is virtually pro-$p$ by Corollary \ref{tatefin}; since $G$ is not virtually pro-$p$, this ensures $G/O_{p'}(G)$ is finite.  Now $S \cap O_{p'}(G) = 1$, so in fact $S$ is finite.\end{proof}

Given a just infinite virtually pro-$p$ group $G$, the $p$-core $O_p(G)$ need not be just infinite.  However, note the following special case of Corollary \ref{newjicor}: 

\begin{cor}Let $G$ be a just infinite profinite group that is virtually pro-$p$ but not virtually abelian.  Then either $O_p(G)$ is already just infinite, or there is a basal subgroup $B$ of $G$ such that $O_p(G) \leq N_G(B)$, and such that $H = N_G(B)/C_G(B)$ and $O = O_p(H)$ are both just infinite.\end{cor}

Outside the virtually abelian case, there is thus a close connection between the class of just infinite pro-$p$ groups and the class of just infinite virtually pro-$p$ groups.  It will be shown later that if $\mcH$ is the set of isomorphism types of just infinite virtually pro-$p$ groups, all with the same isomorphism type of just infinite $p$-Sylow subgroup, then $\mcH$ is finite.

In the next chapter, we will discuss two classes of profinite group, one of which contains all virtually pronilpotent groups, and the other of which contains all just infinite groups of type $\rX$.

\chapter{The generalised pro-Fitting subgroup of a profinite group}

\section{Introduction}

This chapter builds on the basic definitions and results given in Sections \ref{sylprelim} and \ref{qsprelim}.

In particular, recall the definition of the generalised pro-Fitting subgroup $F^*(G)$ of a profinite group $G$ used in this thesis: it is the group generated by all subnormal subgroups of $G$ that are pro-$p$ for some $p$, together with all subnormal subgroups that are quasisimple.  This definition is a natural generalisation of the Fitting subgroup in finite group theory, for the same reason that maximal pro-$p$ subgroups are a natural generalisation of the concept of Sylow $p$-subgroups of a finite group.  This chapter considers the role of this generalised pro-Fitting subgroup in profinite group theory.

In contrast to the situation in finite groups, if $G$ is a profinite group then $F^*(G)$ may be trivial even if $G$ is non-trivial.  As a result, the key property of the generalised pro-Fitting subgroup of a finite group, namely that it contains its own centraliser, does not carry over to the class of profinite groups as a whole.  However, there is an interesting class of profinite groups for which the generalised pro-Fitting subgroup does contain its own centraliser, as we will see later, and this class is in some sense dual to the class of profinite groups $G$ for which $F^*(G) =1$.

\begin{defn}Let $G$ be a profinite group.  Say $G$ is \emph{Fitting-degenerate} if $F^*(G) = 1$.  Say $G$ is \emph{Fitting-regular} if no non-trivial image of $G$ is Fitting-degenerate.  We write $\FD$ for the class of Fitting-degenerate profinite groups and $\FR$ for the class of Fitting-regular profinite groups.  Note that $\FD \cap \FR = [1]$.\end{defn}

The author is not aware of any definitions similar to these in the existing literature.  The goal of this chapter is therefore to serve as an introduction to these properties and their role in the structure of profinite groups.

\section{The internal structure of the generalised pro-Fitting subgroup}

\begin{thm}\label{opilayer}Let $G$ be a profinite group.
\vspace{-12pt}
\begin{enumerate}[(i)] \itemsep0pt
\item The pro-Fitting subgroup of $G$ contains all pronilpotent subnormal subgroups, and is the intersection over all open normal subgroups $N$ of the subgroups $F_N$ of $G$ such that $F_N/N = F(G/N)$.
\item Any component $Q$ of $G$ commutes with both $F(G)$ and all components of $G$ distinct from $Q$.  In other words, $E(G)$ is an unrestricted central product of the components, and $F^*(G)$ is a central product of $F(G)$ and $E(G)$.\end{enumerate}\end{thm}

\begin{proof}(i) This follows immediately from Lemmas \ref{opinor} and \ref{fitlem}.
 
(ii) Let $Q$ be a component, and let $L$ be either a normal pro-$p$ subgroup of $G$, or some component distinct from $Q$.  Suppose that $[x,y] \not= 1$ for some $x \in Q, y \in L$.  Then $G$ has an open normal subgroup $N$ that intersects trivially with $Q$, which does not contain $[x,y]$, and such that $QN \not= LN$, and if $L$ is a component we can also choose $N$ to intersect trivially with $L$.  Then $G/N$ is a finite group, with a subgroup $LN/N$ that is either a normal $p$-subgroup of $G/N$ or a component of $G/N$ distinct from $QN/N$, and yet $LN/N$ does not commute with the component $QN/N$.  This contradicts Theorem \ref{finopil}.\end{proof}

Say a profinite group $G$ is an $F^*$-group if $G = F^*(G)$.

\begin{cor}\label{opicor}Let $G$ be a profinite group.
\vspace{-12pt}
\begin{enumerate}[(i)] \itemsep0pt
\item Let $\mcQ \subseteq \Comp(G)$.  Then $\mcQ$ is a non-redundant set of subgroups of $G$.
\item Let $N$ be a normal subgroup of $G$.  Then $F(N) = F(G) \cap N$ and $E(N) = E(G) \cap N$.  In particular, $N/E(N)$ is isomorphic to a normal subgroup of $G/E(G)$, and $N/F(N)$ is isomorphic to a normal subgroup of $G/F(G)$.
\item Let $N$ be a normal subgroup of $G$.  Then $F(G)N/N \leq F(G/N)$ and $E(G)N/N \leq E(G/N)$.  In particular, if $G$ is an $F^*$-group then so is $G/N$.
\item The group $G$ is an $F^*$-group if and only if $G/F(G)$ is perfect and $G/E(G)$ is pronilpotent.  Moreover, if $G$ is an $F^*$-group then $G/F(G)$ is a Cartesian product of non-abelian finite simple groups.  In particular, every image of an $F^*$-group is an $F^*$-group.
\item We have $G = E(G)$ if and only if $G$ is perfect and $O^\fsg(G)$ is central.\end{enumerate}\end{cor}

\begin{proof}(i) Let $Q \in \mcQ$, and let $H = \langle \mcQ \setminus \{Q\} \rangle$.  Then $H$ centralises $Q$; but $Q$ is non-abelian, so $Q \not\le H$.

(ii) Since $N$ is normal in $G$, it is clear that $F(N)$ is a subnormal pronilpotent subgroup of $G$, so $F(N) \leq F(G)$, and $E(N)$ is generated by quasisimple subnormal subgroups of $G$, so $E(N) \leq E(G)$.

(iii) Clearly $F(G)N/N$ and $F(G) \cap N$ are pronilpotent, and $E(G)N/N$ and $E(G) \cap N$ are unrestricted central products of quasisimple groups by the theorem.  

Let $N \unlhd G$, and suppose $G = F^*(G)$.  Then $F(G)N/N$ is a pronilpotent normal subgroup of $G/N$, so is contained in $F(G/N)$.  Given $Q \in \Comp(G)$, then either $Q \leq N$, or $QN/N$ is a component of $G/N$; hence $E(G)N/N \leq E(G/N)$.  So if $G = F(G)E(G)$, then 
\[ G/N = F(G)E(G)/N = (F(G)N/N)(E(G)N/N) \leq F(G/N)E(G/N).\]

(iv) If $G=F(G)E(G)$, then $G/F(G)$ is isomorphic to an image of $E(G)$, and hence perfect, while $G/E(G)$ is isomorphic to an image of $F(G)$, and hence pronilpotent; indeed, $G/F(G) \cong E(G)/Z(E(G))$, which means that $G/F(G)$ is a Cartesian product of non-abelian finite simple groups by the theorem.  Conversely, suppose $G/F(G)$ is perfect and $G/E(G)$ is pronilpotent.  Then $G/F^*(G)$ is both perfect and pronilpotent, and hence trivial.

(iv) Let $Z = O^\fsg(G)$. Suppose $G$ is perfect and $Z$ is central.  By Lemma \ref{ofsglem}, $G/Z$ is a Cartesian product of its components.  Moreover, given $Q/Z \in \Comp(G/Z)$, then $O^\sol(Q)$ is a perfect central extension of $Q/Z$, and $O^\sol(Q)Z=Q$.  Thus $G=E(G)Z$.  But $G$ is perfect, so $G = O^\sol(E(G)Z)) = E(G)$.  The converse follows immediately from the theorem.\end{proof}

\begin{prop}Let $G$ be a profinite group, and let $M$ be a normal subgroup.  The following are equivalent:
\vspace{-12pt}
\begin{enumerate}[(i)] \itemsep0pt
\item $M$ is an $F^*$-group;
\item $M \leq F^*(G)$;
\item $MN/N \leq F^*(G/N)$, for every open normal subgroup $N$ of $G$;
\item $MN/N$ is an $F^*$-group, for every open normal subgroup $N$ of $G$.\end{enumerate}\end{prop}

\begin{proof}Assume (i).  By Corollary \ref{opicor}, $F(M)$ is contained in $F(G)$ and $E(M)$ is contained in $E(G)$, so $M=E(M)F(M) \leq F^*(G)$.

Assume (ii).  Then $MN/N$ is an $F^*$-group that is a normal subgroup of $G/N$, so must be contained in $F^*(G/N)$ by the fact that (i) implies (ii).

Assume (iii).  Clearly $F^*(G/N)$ itself is a finite $F^*$-group.  It follows that $MN/N$ is also an $F^*$-group as $MN/N \unlhd F^*(G/N)$.

Assume (iv).  Let $E_N$ be the lift of $E(M/M \cap N)$ to $M$, and let $E$ be the intersection of the $E_N$ as $N$ ranges over the open normal subgroups.   Clearly $E \unlhd M$.  From the structure of finite $F^*$-groups, it is clear that $E_N = E_KN$ whenever $K \unlhd_o G$ such that $K \leq N$.  Hence $E_N = EN$; this ensures $M/E$ is an inverse limit of nilpotent groups, so is pronilpotent.  Furthermore, $E$ is the inverse limit of perfect groups, so is itself perfect, and $O^\fsg(E) \leq Z(E)$ by considering the action of $O^\fsg(E)$ on the finite images of $E$.  Thus $E = E(M)$.  Let $F_N$ be the lift of $F(M/M \cap N)$ to $M$, and let $F$ be the intersection of the $F_N$ as $N$ ranges over the open normal subgroups.  Then $M/F_N$ is perfect for every $N$, and $F$ is an inverse limit of nilpotent groups, so in fact $F=F(M)$, which means $M/F(M)$ is perfect.  Hence $M$ is an $F^*$-group, by Corollary \ref{opicor}.\end{proof}

It is certainly not the case that all profinite groups are Fitting-regular.  For instance, any non-abelian free profinite group is Fitting-degenerate, and the group described in \cite{Luc} is a just infinite profinite group that is Fitting-degenerate.  There are also examples with restrictions on the Sylow subgroups:

\begin{prop}\
\vspace{-12pt}
\begin{enumerate}[(i)] \itemsep0pt
\item There exist non-trivial Fitting-degenerate prosoluble groups, all of whose Sylow subgroups are finite.
\item There exist non-trivial Fitting-degenerate prosoluble groups that are countably based pro-$\{p,q\}$ groups, for any distinct primes $p$ and $q$.\end{enumerate}\end{prop}

\begin{proof}There is a sequence of finite groups $G_i$ and primes $p_i$, with the following properties:
\vspace{-12pt}
\begin{enumerate}[(i)] \itemsep0pt
\item $G_1$ is cyclic of order $p_1$;
\item $G_{i+1} = V_{i+1} \rtimes G_i$, where $V_{i+1}$ is an elementary abelian $p_{i+1}$-group on which $G_i$ acts faithfully.\end{enumerate}

No matter what finite group we have for $G_i$, it is possible to choose a suitable $G_{i+1}$, and moreover we can choose the primes $p_i$ freely.  These $G_i$ also form an inverse system in an obvious way, giving rise to a profinite group $G$, which is countably based and prosoluble by construction.  For (i), make all the $p_i$ distinct, so that every Sylow subgroup is finite, and for (ii), make $p_i = p$ for $i$ odd and $p_i=q$ for $i$ even, so that $G$ is a pro-$\{p,q\}$ group.

In each case, consider the pro-Fitting subgroup of $G$.  For each $G_i$, the image of $F(G)$ in $G_i$ is contained in $F(G_i)$.  But it is clear from the construction of $G_i$ that $F(G_i)$ is always a $p_i$-group.  If $F(G)$ is non-trivial, then it contains a non-trivial pro-$p$ subgroup for some $p$.  This would force $p_i=p$ for all sufficiently large $i$.  But this does not occur for either (i) or (ii), and so the constructions yield Fitting-degenerate profinite groups in both cases.\end{proof}

\section{A structure theorem for Fitting-regularity}

First, note some closure properties of $\FD$:

\begin{lem}The class $\FD$ is closed under subnormal subgroups, extensions, and sub-Cartesian products.\end{lem}

\begin{proof}For the first claim, it suffices to consider normal subgroups.  Let $N$ be a normal subgroup of $G \in \FD$.  Then $F^*(N) \leq F^*(G) = 1$.
 
Now let $G$ be a profinite group with $N \lhd G$ such that $N, G/N \in \FD$.  Then $F^*(G) \cap N = F^*(N) =1$, and $F^*(G)N/N \leq F^*(G/N) =1$, so $F^*(G)=1$.

Now let $G$ be a profinite group that is a sub-Cartesian product of groups $H_i \in \FD$.  In other words, there are surjective maps $\rho_i$ from $G$ to $H_i$ for each $i$, such that the kernels of the $\rho_i$ have trivial intersection.  But $(F^*(G))^{\rho_i}$ is an $F^*$-group and hence trivial for each $i$, since it is normal in the Fitting-degenerate group $H_i$, so $F^*(G)$ must be trivial.\end{proof}

\begin{cor}Let $G$ be a profinite group.  Then $G/O^\FD(G) \in \FD$, and $O^\FD(G) \in \FR$.\end{cor}

\begin{proof}The first claim follows immediately from the fact that $\FD$ is closed under sub-Cartesian products.  For the second,  let $G$ be a profinite group, let $R = O^\FD(G)$, and let $N = O^\FD(R)$.  Then $N$ is characteristic in $R$, and hence in $G$, and $G/N \in \FD$, since $\FD$ is closed under extensions.  Hence by definition $R \leq N$, so $N = R$, in other words $R$ has no non-trivial Fitting-degenerate images.\end{proof}

Here are some closure properties of $\FR$.

\begin{lem}The class $\FR$ is closed under subnormal sections, subnormal joins and extensions.\end{lem}

\begin{proof}For the first claim, it suffices to show $\FR$ is closed under normal subgroups, by induction on the degree of subnormality and by the definition of Fitting-regularity.  Let $G \in \FR$, and suppose $N \lhd G$ such that $O^\FD(N) < N$.  Then $O^\FD(N)$ is a normal subgroup of $G$, since it is characteristic in $N$, so by replacing $G$ by $G/O^\FD(N)$ and $N$ by $N/O^\FD(N)$, we may assume that $N$ is Fitting-degenerate.  Since $N$ is Fitting-degenerate it has trivial centre, and so the normal subgroup $H=C_G(N)$ of $G$ has trivial intersection with $N$.  Hence $N$ is isomorphic to $NH/H$, and the centraliser in $G/H$ of $NH/H$ is trivial.  In particular, $NH/H$ intersects non-trivially with any non-trivial normal subgroup of $G/H$.  But then $F^*(G/H) \cap NH/H$ is a non-trivial normal subgroup of $NH/H$ that is an $F^*$-group, contradicting the Fitting-degeneracy of $NH/H$.

Now let $G$ be a profinite group generated by subnormal subgroups $N_i \in \FR$.  Let $M$ be a proper normal subgroup of $G$.  Then there is some $N_i$ not contained in $M$, and so $G/M$ has a subnormal subgroup $N_i M/M$ that is isomorphic to the non-trivial image $N_i/(N_i \cap M)$ of $N_i$, so $F^*(G/M) \geq F^*(N_i M/M) > 1$.  As $M$ was an arbitrary proper normal subgroup, this means $G \in \FR$.

Now let $G$ be a profinite group with $N \unlhd G$ such that $N \in \FR$ and $G/N \in \FR$.  Let $M$ be a proper normal subgroup of $G$.  If $N \leq M$, then $G/M \not\in \FD$, as it is an image of $G/N$.  Otherwise, $G/M \not\in \FD$ by the argument of the previous paragraph.\end{proof}

\begin{rem}All countably based profinite groups are subgroups of the Fitting-regular group $\prod_{n \geq 5} \Alt(n)$ (see \cite{Wil}), but there are non-trivial Fitting-degenerate countably based profinite groups.  Hence the class $\FR$ is not closed under subgroups, though it is closed under taking open subgroups, as will be seen shortly.  It is not clear whether or not a subgroup of a Fitting-regular prosoluble group can fail to be Fitting-regular.\end{rem}

We are now ready to prove the following structure theorem:

\begin{thm}Let $G$ be a profinite group.  Then $G$ has a characteristic subgroup $R = O^\FD(G) = O_\FR(G)$, such that:
\vspace{-12pt}
\begin{enumerate}[(i)] \itemsep0pt
\item A subnormal subgroup of $G$ is Fitting-regular if and only if it is contained in $R$;
\item $G/R$ is Fitting-degenerate, and covers every Fitting-degenerate quotient of $G$.\end{enumerate}

Let $H$ be an open subgroup of $G$.  Then $O_\FR(H) = O_\FR(G) \cap H$.\end{thm}

\begin{proof}Let $R=O^\FD(G)$.  We have already seen that $R$ is Fitting-regular, so $R \leq O_\FR(G)$, and that this implies all subnormal subgroups of $R$ are Fitting-regular.

Conversely, let $N$ be a Fitting-regular subnormal subgroup of $G$.  Then $NR/R$ is Fitting-regular, as it is isomorphic to an image of $N$, but it is also Fitting-degenerate, as it is a subnormal subgroup of $G/R \in \FD$.  Hence $NR/R=1$, and so $N \leq R$.  This demonstrates that (i) holds, and also that $R = O_\FR(G)$.

We have also already seen that $G/R$ is Fitting-degenerate, and by definition it covers every Fitting-degenerate quotient of $G$.  This is (ii).

Finally, let $H$ be an open subgroup of $G$.  Let $K=O_\FR(G)$.  Then $G/K$ is Fitting-degenerate, so the core of $HK/K$ in $G/K$ is Fitting-degenerate; this means that $HK/K$ has a Fitting-degenerate normal subgroup of finite index, and hence $O_\FR(HK/K)$ must be finite.  But then the elements of $O_\FR(HK/K)$ have centralisers of finite index in $HK/K$, and hence also in $G/K$, and they are of finite order; thus $O_\FR(HK/K)$ is contained in a finite normal subgroup of $G/K$, by Dicman's lemma.  Since $G/K$ is Fitting-degenerate, this implies $O_\FR(HK/K)=1$, and so $HK/K$ is Fitting-degenerate.  Hence $H/(H \cap K)$ is a Fitting-degenerate image of $H$, which ensures that $O_\FR(H) \leq H \cap K$.

Let $M$ be the core of $H$ in $G$.  Then $M \cap K$ is an open normal subgroup of $H \cap K$, and $M \cap K$ is also a normal subgroup of $K$ and hence Fitting-regular.  It follows that $H \cap K$ itself is (Fitting-regular)-by-finite, and thus Fitting-regular, as $\FR$ is closed under extensions.  So $H \cap K \leq O_\FR(H)$.\end{proof}

Here is one application.

\begin{cor}\label{vzfr}Let $G$ be a profinite group such that $VZ(G)$ is dense.  Then $G \in \FR$.\end{cor}

\begin{proof}Since $VZ(G/K)$ is dense in $G/K$ for any $K \unlhd G$, it suffices to suppose $G \in \FD$ and show $VZ(G)=1$.  Let $H \unlhd_o G$.  Then $C_G(H)$ is virtually abelian, so $C_G(H) \in \FR$, and also $C_G(H) \unlhd G$, so in fact $C_G(H) \leq O_\FR(G) = 1$ by the theorem.  Hence $VZ(G)=1$.\end{proof}

We conclude this section with some results that show that the property of Fitting-degeneracy can be identified from `small' images.

\begin{defn}A finite group $G$ is \emph{primitive} if there exists a maximal subgroup $M$ such that $\Core_G(M)=1$, in other words if $G$ has a faithful primitive permutation action on some finite set.  (Note that this is distinct from the notion of a primitive linear group.)  Note that a finite image $G/N$ of a profinite group $G$ is primitive if and only if $N$ is the core of a maximal open subgroup of $G$.\end{defn}

\begin{prop}Let $G$ be a non-trivial profinite group.  Then $G \in \FD$ if and only if the following holds, for every $x \in G \setminus 1$:
 
$(*)$ There is an open normal subgroup $K$ of $G$, depending on $x$, such that $G/K$ is primitive and $xK$ is not contained in $F^*(G/K)$.\end{prop}

\begin{proof}Suppose $G$ is not Fitting-degenerate, and take $x \in F^*(G) \setminus 1$.  Then $xK \in F^*(G/K)$ given any $K$, so $(*)$ is false.

Now assume $G$ is Fitting-degenerate.  This implies that $G$ has no non-trivial pronilpotent normal subgroups, so in particular $\Phi(G)$ is trivial.  Hence for any element $x$ of $G \setminus 1$, there is a maximal subgroup $H$ such that $H$ does not contain $x$, and so the core $K$ of $H$ also does not contain $x$.  Since $G/K$ is primitive, this shows that $G$ is the inverse limit of finite primitive groups.  If $xK \in F^*(G/K)$ for a given $x$ in every such image, it would imply $x \in F^*(G)$, contradicting the assumption that $G$ is Fitting-degenerate.  This proves $(*)$.\end{proof}

\begin{prop}Let $G$ be a profinite group that is not Fitting-regular.  Then there is a proper normal subgroup $N$ of $G$ such that $G/N$ is countably based and Fitting-degenerate.\end{prop}

\begin{proof}Without loss of generality, we may replace $G$ by $G/O^\FD(G)$, and so assume $G$ is Fitting-degenerate.  Set $N_1$ to be any proper open normal subgroup of $G$.  We obtain open normal subgroups $N_{i+1}$ for $i \in \mathbb{N}$ inductively as follows:

Let $M$ be the lift to $G$ of a non-trivial normal subgroup of $G/N_i$.  Suppose for every open normal subgroup $K$ of $G$ contained in $N_i$ that $F^*(G/K)$ covers $M/N_i$.  Then by a standard inverse limit argument, we would obtain a subgroup of $F^*(G)$ covering $M/N_i$, which is impossible as $G$ is Fitting-degenerate.  So there must be an open normal subgroup $K_M \leq N_i$ of $G$ such that $F^*(G/K_M)$ does not cover $M/N_i$.  Set $N_{i+1}$ to be the intersection of all the chosen normal subgroups $K_M$.

Now set $N$ to be the intersection of all the $N_i$; by construction, $G/N$ is countably based.  Also by construction, $F^*(G/N_{i+1})$ cannot cover any non-trivial normal subgroup of $G/N_i$, and hence $F^*(G/N_{i+1}) \leq N_i/N_{i+1}$.  As a result $F^*(G/N) \leq N_i/N$ for all $i$, so $G/N$ is Fitting-degenerate.\end{proof}

\section{The generalised pro-Fitting subgroup in Fitting-regular groups}

\begin{thm}\label{frcent}Let $G \in \FR$.  Then $C_G (F^*(G)) = Z(F(G))$.\end{thm}
 
\begin{proof}Consider the subgroup $H = C_G(F(G))$ and its pro-Fitting subgroup.  Since $H$ is normal in $G$, we have $F(H) \leq F(G)$ and hence $F(H) = Z(H) = Z(F(G))$.
 
Set $K=H/Z(H)$.  Let $L$ be the subgroup of $H$ such that $L/Z(H) = F(K)$.  Then $L$ is a central extension of $F(K)$, and hence pronilpotent; it is also normal in $H$.  Hence $L \leq F(H)=Z(H)$, and so $F(K)=1$.  Now consider $D = C_K(E(K))$.  Then $D$ cannot contain a component, as this would be contained in $E(K)$, and yet a component cannot centralise itself; so $E(D)=1$.  Also, $F(D) \leq F(K) = 1$, so $F^*(D)=1$.  However, $D$ is a normal section of $G$, so $D$ is Fitting-regular.  Hence $D=1$, in other words $K$ acts faithfully on $E(K)$.

Let $T$ be a component of $K$.  Then $T$ is simple as $F(K)=1$.  Let $U$ be the subgroup of $H$ such that $U/Z(H) = T$.  Then $U$ is a central extension of $T$, and $O^\sol(U)$ is a perfect central extension of $T$.  Here $O^\sol(U)Z(H)/Z(H) = T$, and $O^\sol(U)$ is subnormal in $G$ by construction, so $O^\sol(U) \leq E(G)$.  This proves that $E(G)Z(H)/Z(H) \geq E(K)$.

In conclusion, $H/Z(H)$ acts faithfully on $E(G)Z(H)/Z(H)$, so that $C_H(E(G)) \leq Z(H)=Z(F(G))$; in fact $C_H(E(G)) = Z(F(G))$, as $[E(G),F(G)]=1$.  But $C_H(E(G))$ is precisely $C_G(E(G)) \cap C_G(F(G))$, which is the centraliser of $E(G)F(G) = F^*(G)$.\end{proof}

\begin{rem}The converse of this theorem is false: consider for instance a profinite group of the form $G = V:L$ where $V$ is elementary abelian of countably infinite rank, and $L$ acts faithfully on $V$.  Since $V$ admits a faithful action of any countably based profinite group, we can choose $L$ to be non-trivial Fitting-degenerate, and yet $F(G) \geq V \geq Z(F(G))$.\end{rem}

In general, the pro-Fitting subgroup of a Fitting-regular group $G$ may have infinitely generated Sylow subgroups, even if the Sylow subgroups of $G$ are finitely generated, and so the automorphism group of $F^*(G)$ in isolation may not place much of a restriction on $G/F^*(G)$.  However, in the case of Fitting-regular prosoluble groups there is a useful interaction between the pro-Fitting subgroup and the pronilpotent residual, which means that some information about the action on $F(G)$ can be obtained from automorphisms of Sylow subgroups of $G$.

\begin{prop}\label{nilpresFR}Let $G$ be a Fitting-regular prosoluble group, let $p$ be a prime and let $N=O^\nil(G)$.  Then there is a pro-$p$ subgroup $R$ of $G$ such that:
\vspace{-12pt}
\begin{enumerate}[(i)] \itemsep0pt
\item $NR$ contains all $p$-Sylow subgroups of $G$;
\item $R$ normalises a $q$-Sylow subgroup $S_q$ of $G$, for all $q \in \bP$;
\item $R/(R \cap F(G)) \lesssim \prod_{q \in p'} \Delta(S_q)$.\end{enumerate}\end{prop}

\begin{proof}Let $K$ be the subgroup of $G$ such that $K \geq N$ and $K/N = O_{p'}(G/N)$.  Then $K$ is normal in $G$, and contains every $q$-Sylow subgroup of $G$ for every $q \in p'$, since $G/N$ is pronilpotent.  By Sylow's theorem and compactness, by choosing suitable $q$-Sylow subgroups $S_q$ for $q \in p'$, we can ensure $KY = G$, where $Y = \bigcap_{q \in p'} N_G(S_q)$.  This means that $KR=G$ for any $p$-Sylow subgroup $R$ of $Y$; this $R$ clearly satisfies (i) and (ii).

Let $D = \bigcap_{q \in p'}C_R(S_q)$.  Then $R/D \lesssim \prod_{q \in p'} \Delta(S_q)$, since $R/C_R(S_q)$ acts faithfully on $S_q$ and hence $R/C_R(S_q) \lesssim \Delta(S_q)$.  Also, it is clear that $R \cap F(G) \leq D$.  Let $\theta$ be the quotient map from $G$ to $G/O_p(G)$.  Then $O_p(G^\theta)=1$, and $D^\theta$ centralises $O_q(G^\theta)$ for all $q \in p'$ since $O_q(G^\theta)$ is contained in $S^\theta_q$.  Hence $D^\theta \leq C_{G^\theta}(F(G^\theta))$.  But $G^\theta$ is Fitting-regular and prosoluble, so $C_{G^\theta}(F(G^\theta)) \leq F(G^\theta)$, and hence $D^\theta \leq O_p(G^\theta) = 1$.  Hence $D \leq O_p(G) \leq F(G)$, so $D = R \cap F(G)$.\end{proof}

As an example, consider the case of a profinite group $G$ involving exactly two primes $p$ and $q$.  It is a well-known result of Burnside that all finite $\{p,q\}$-groups are soluble; this ensures that all pro-$\{p,q\}$ groups are prosoluble.  Recall the invariant $c$ from Section \ref{cinvsec}.  Given a profinite group $G$ with $p$-Sylow subgroup $S$, define $c_p(G) := c(S)$; given a set of primes $\pi$, define $c_\pi(G) := \sup_{p \in \pi} c_p(G)$.

\begin{cor}Let $p$ and $q$ be distinct primes, let $G$ be a pro-$\{p,q\}$ group, and let $N=O^\nil(G)$.  Let $S$ be a $p$-Sylow subgroup of $G$, let $T$ be a $q$-Sylow subgroup of $G$, and suppose that both $d(S)$ and $d(T)$ are finite. 
\vspace{-12pt}
\begin{enumerate}[(i)] \itemsep0pt
\item There is some $M \leq NF(G)$ such that $M \unlhd S$ and $S/M \lesssim \Delta(T)$.
\item If $\ord(p,q) > c_q(G)$, then $S \leq NF(G)$.
\item If $\ord(p,q) > c_q(G)$ and $\ord(q,p) > c_p(G)$, then $G$ is pronilpotent.\end{enumerate}\end{cor}

\begin{proof}By Corollary \ref{tatenilp}, $G$ is virtually pronilpotent; hence $G$ is Fitting-regular.  It follows from Proposition \ref{nilpresFR} that there is a pro-$p$ subgroup $R$ of $G$ such that $S \leq NR$ and $R/(R \cap F(G)) \lesssim \Delta(T)$.  Now $R \lesssim S$ by Sylow's theorem, so assertion (i) follows.
  
Let $c = c_q(G)$.  By Lemma \ref{deltalem}, $O^{(c,q)^*}(S/M)=1$.  If $\ord(p,q) > c$, then $|\GL(c,p)|_q = 1$, so $S=M$, in other words $S \leq NF(G)$, proving (ii).
  
If $\ord(p,q) > c_q(G)$ and $\ord(q,p) > c_p(G)$, then by (ii), all Sylow subgroups of $G$ are contained in $NF(G)$, so $G=NF(G) \leq G'F(G)$.  But then $G/F(G)$ is both perfect and prosoluble, and hence trivial, so $G = F(G)$, and hence $G$ is pronilpotent by Lemma \ref{fitlem}.\end{proof}

\begin{rem}Given any $n$ and fixed $p$, $\ord(p,q) > n$ for all $q > p^n$, so in particular, if $S$ is any given finitely generated pro-$p$ group, then $\ord(p,q) > d(S)$ for all but finitely many primes $q$.\end{rem}

\chapter{Profinite groups in terms of their Sylow subgroups}

\section{Finite simple groups involved in a profinite group}

In this section, $G$ will be a profinite group, such that for each $p$ in some non-empty set of primes $\pi$, the isomorphism type of a $p$-Sylow subgroup of $G$ is specified, and $d_p(G)$ is finite.  Recall (Corollary \ref{tatefin}) that if $d_2(G)$ is finite, then $G$ is virtually prosoluble.  In this section, we will derive further results concerning components and non-abelian composition factors, using the notation and results established in Section \ref{qsprelim}.

\begin{lem}\label{compphi}Let $G$ be a profinite group, let $p$ be a prime, and let $S$ be a $p$-Sylow subgroup of $G$.  Let $Q \in \Comp_p(G) \cup \{1\}$.  Let $L$ be a normal subgroup of $G$, and suppose $(Q^S \cap S)\Phi(S) \leq (L \cap S)\Phi(S)$.  Then $Q \leq L$.\end{lem}

\begin{proof}Dividing out by $L$ does not affect the hypotheses, so it suffices to assume $L=1$ and prove $Q=1$.  We may also assume $G=SQ^S$.  Under these assumptions, $Q^S \cap S \leq \Phi(S)$ and $Q^S \unlhd G$.  Hence $Q^S$ is $p'$-normal by Corollary \ref{tatecor}, and so also $Q$ is $p'$-normal; in particular $|Q/Z(Q)|_p=1$.  Hence $Q=1$.\end{proof}

Recall the definition of $f_p$ given after Proposition \ref{pspace}; note that $f_p(S) \leq d_p(S)$ for any finitely generated pro-$p$ group $S$.

\begin{cor}\label{comporb}Let $G$ be a profinite group and let $S$ be a $p$-Sylow subgroup of $G$.
\vspace{-12pt}
\begin{enumerate}[(i)] \itemsep0pt 
\item Let $\mcQ \subset \Comp_p(G)$ such that $(\langle \mcQ \rangle \cap S)\Phi(S) = (E_p(G) \cap S)\Phi(S)$ and such that $\langle \mcQ \rangle$ is normalised by $S$.  Then $\Comp_p(G) = \mcQ$.
\item Suppose in addition $d(S)$ is finite, and let $n$ be the number of orbits of $S$, acting on $\Comp_p(G)$ by conjugation.  Then $n \leq f_p(S) \leq d(S)$.\end{enumerate}\end{cor}

\begin{proof}(i) By Lemma \ref{compphi} we have $Q \leq \langle \mcQ \rangle$ for every $Q \in \Comp_p(G)$.  By Corollary \ref{opicor} (i), it follows that $Q \in \mcQ$ for every $Q \in \Comp_p(G)$.

(ii) We see that $E_p(G) \cap S$ is generated by normal subgroups of $S$, and hence the rank of $(E_p(G) \cap S)\Phi(S)/\Phi(S)$ is at most $f_p(S)$, we can find some $\mcQ$ satsifying the equation in (i) by taking the union of $m$ orbits in $\Comp_p(G)$ under the action of $S$, where $m \leq f_p(S)$.  But then $\mcQ = \Comp_p(G)$ by part (i), so $n = m$.\end{proof}

\begin{rem}A similar result, concerning non-abelian chief factors, is given in \cite{Wil}.\end{rem}

\paragraph{}Now consider the following question:
 
\begin{que}\label{cfque}How many non-abelian composition factors $Q$ can $G$ have such that $p$ divides $|Q|$ for all $p \in \pi$?\end{que}
 
In constrast to the case of chief factors, if $\pi=\{p\}$ there is no bound in terms of $d_p(G)$: consider for instance groups of the form $G=\Alt(k) \wr C_{p^n}$ where $k=\max\{p,5\}$.  However, there is a bound if $\pi$ contains two or more primes and the corresponding Sylow subgroups are finitely generated.

\begin{defn}Given natural numbers $n$ and $b$, let $s_b(n)$ be the sum of the digits of the base $b$ expansion of $n$.\end{defn}  

We will use the following numerical theorem:

\begin{thm}[Senge, Straus \cite{Sen}]Let $a$ and $b$ be integers such that $\log a/\log b \not\in \bQ$, let $s$ be a natural number, and let $X$ be the set of natural numbers $n$ such that $\max \{s_a(n),s_b(n)\} \leq s$.  Then $X$ is finite.\end{thm}

\begin{thm}\label{compbound}\
\vspace{-12pt}
\begin{enumerate}[(i)] \itemsep0pt 
\item Let $G$ be a profinite group.  Let $p$ and $q$ be distinct primes, and suppose $\max\{d_p(G),d_q(G)\} = d$ is finite.  Then the number of composition factors of $G$ of order divisible by $pq$ is bounded by a function of $(d,p,q)$.
\item Let $G$ be a profinite group.  Suppose $\max\{d_2(G),d_3(G),d_5(G)\} = d$ is finite.  Then the number of non-abelian composition factors of $G$ is bounded by a function of $d$.\end{enumerate}\end{thm}
 
\begin{proof}It suffices to consider the case of $G$ finite.  Given a prime $p$, let $S_p$ be a $p$-Sylow subgroup of $G$.

(i) Let $\pi = \{p,q\}$.  By first dividing out by a suitable normal subgroup, we may assume that every normal subgroup of $G$ has a composition factor of order divisible by $pq$.  This ensures that $G$ acts faithfully on $E_\pi(G)$.  Let $t_\pi = |\Comp_\pi(G)|$.  By Corollary \ref{comporb} (ii), there are at most $d_p(G)$ orbits in the action of $S_p$ on $\Comp_\pi(G)$, but also at most $d_q(G)$ orbits in the action of $S_q$ on $\Comp_\pi(G)$.  Since all orbits of a permutation $r$-group have size a power of $r$ for any prime $r$, it follows that $s_p(t_\pi) \leq d$ and $s_q(t_\pi) \leq d$.  Hence by the theorem of Senge and Straus, $t_\pi$ is bounded by a function of $(d,p,q)$.

Now consider the quotient $G/E_\pi(G)$; note that this is isomorphic to a subgroup of $\Out(E_\pi(G))$.  Let $N= \bigcap \{ N_G(Q) \mid Q \in \Comp_\pi(G)\}$.  By Proposition \ref{qsout}, $N/E_\pi(G)$ is soluble; also, $G/N \lesssim \Sym(t_\pi)$.  It follows that in some composition series for $G$, the number of factors of order divisible by $pq$ is bounded by a function of $t_\pi$, and hence by a function of $d$; this same bound then applies to an arbitrary composition series by the Jordan-H\"{o}lder theorem.
 
(ii) Let $t = |\Comp(G)|$, let $t_3 = |\Comp_{\{2,3\}}(G)|$, and let $t_5 = |\Comp_{\{2,5\}}(G)|$.  By Theorem \ref{simordthm}, $\Comp_{\{2,3\}}(G) \cup \Comp_{\{2,5\}}(G) = \Comp(G)$, so $t \leq t_3 + t_5$.  By part (i), $t_3$ is bounded by a function of $(d,2,3)$, and $t_5$ by a function of $(d,2,5)$; hence both $t_3$ and $t_5$ are bounded by a function of $d$, and so $t$ is bounded by a function of $d$.  This gives a bound on the number of non-abelian composition factors of $G$ by the same argument as before.\end{proof}

Depending on the structure of $S$, we may obtain a bound on the number of components by other means.

\begin{defn}Given any subgroup $H$ of a profinite group $G$, define $m_G(H)$ to be the minimum number of conjugates of $H$ needed to generate $H^G$.  Note that if $H \in \Comp(G)$, then $m_G(H) = |G:N_G(H)|$.
 
Let $S$ be a finitely generated pro-$p$ group and let $F$ be a finite subgroup.  Say $F$ is a \emph{pseudo-component} of $S$ if $F^S$ is a central product of a finite set of conjugates of $F$.\end{defn}

\begin{lem}\label{pscomplem}Let $G$ be a profinite group with $p$-Sylow subgroup $S$, and let $Q \in \Comp_p(G)$.  Then $N_G(Q \cap S) \leq N_G(Q)$, and $m_S(Q \cap S) = m_{SE(G)}(Q)$.\end{lem}

\begin{proof}Note first that $Q \cap S$ is a $p$-Sylow subgroup of $Q$, by Corollary \ref{sylsubnor}.  Let $g \in G$, and suppose $Q \not= Q^g$.  Then $Q \cap Q^g \leq Z(Q)$, as $Q$ and $Q^g$ are both components of $G$.   Now $Z(Q)$ cannot contain a $p$-Sylow subgroup of $Q$, and hence neither can $Q \cap Q^g$; thus $Q \cap S$ is not normalised by $g$.  Thus $N_G(Q \cap S) \leq N_G(Q)$.
  
Let $Z = Z(E(G))$.  Then $QZ/Z$ and $Q^gZ/Z$ are elements of $\Comp_p(G/Z)$, and by the same argument they do not have any $p$-Sylow subgroups in common; as a result, we may assume $Z=1$, so that $Q$ is simple.

We may now assume $G=SE(G)$.  Let $m_{SE(G)}(Q) = m$ and let $K=Q^G$.  Then $K \cap S$ is a $p$-Sylow subgroup of $K$, and is also the normal closure of $Q \cap S$ in $S$; indeed, since every component is normal in $E(G)$, the conjugates of $Q \cap S$ in $S$ are precisely the subgroups of the form $Q^g \cap S$ for some $g \in G$.  Also, $K$ is a direct product of the conjugates of $Q$ by Theorem \ref{opilayer}, so $|K|_p = |Q|^m_p$.  Thus all $m$ conjugates of $Q \cap S$ in $S$ are necessary to generate $K \cap S$, so $m_S(Q \cap S) = m_{SE(G)}(Q)$.\end{proof}

\begin{cor}Let $G$ be a profinite group, with $p$-Sylow subgroup $S$.  Suppose that the multiplicity of every pseudo-component of $S$ is at most $m$.  Then each orbit of $S$ on $\Comp_p(G)$ has size at most $m$.\end{cor}

\begin{proof}Let $Q \in \Comp_p(G)$.  Then $Q \cap S$ is a pseudo-component of $S$, so $m_{SE(G)}(Q) = m_S(Q \cap S) \leq m$.\end{proof}

\section{Profinite groups with finite $c_\bP$}

Recall the $c$ invariant of pro-$p$ groups, as discussed in Section \ref{cinvsec}.  We now consider profinite groups $G$ for which $c_\bP(G)=c$ is finite, that is, the Frattini factor of each Sylow subgroup $S$ has a series preserved by $\Aut(S)$, the factors of which have rank at most $c$.  At first we will consider this case in full generality; in later sections, we will obtain stronger results in special cases.

\begin{defn}Let $G$ be a profinite group.  Say $G$ is \emph{Sylow-finite} if every Sylow subgroup of $G$ is contained in a finite normal subgroup.\end{defn}

By Dicman's lemma, this condition is equivalent to requiring each Sylow subgroup to be finite and to have a centraliser in $G$ of finite index.

Note some properties of the class of Sylow-finite groups.

\begin{lem}The class of Sylow-finite groups is closed under subgroups, quotients and finite direct products.  Given a Sylow-finite group $G$, then $\Fin(G)$ is dense in $G$ and $G \in \FR$.\end{lem}

\begin{proof}The properties of having finite Sylow subgroups, and having Sylow subgroups with centralisers of finite index, are both preserved by the operations in question.  Given a Sylow-finite group $G$, then $\Fin(G)$ is dense as it contains every Sylow subgroup.  Hence $VZ(G)$ is dense, so $G \in \FR$ by Corollary \ref{vzfr}.\end{proof}

Our main aim in this section will be to establish the following decomposition theorem.

\begin{thm}\label{cstarthm}Let $G$ be a profinite group for which $c_\bP(G)=c$ is finite.  Then $|G:O_\sol(G)|$ is finite, and $c_\bP(O_\sol(G))=c'$ is finite.  Let $K$ be the smallest normal subgroup of $O_\sol(G)$ such that $O_\sol(G)/K$ has exponent dividing $\eb(c')$, and derived length at most $\db(c')$; this ensures $O_\sol(G)/K$ is of order bounded by a function of $c'$ and the maximum of $d_p(O_\sol(G))$ as $p$ ranges over $\bP_{\eb(c')}$.
 
Then $N = K'F(G)/F(G)$ is Sylow-finite.

In particular, $G$ is pronilpotent-by-(Sylow-finite)-by-abelian-by-finite.\end{thm}

The following is a simple but useful observation.

\begin{lem}\label{fistarlem}Let $G$ be a profinite group, and let $N$ be an open normal subgroup of $G$.  Then $d_\bP(N) \leq |G:N|d_\bP(G)$ and $d_\bP(G) \leq d_\bP(N) + d_\bP(G/N)$.  In addition, $c_\bP(G)$ is finite if and only if $c_\bP(N)$ is finite.\end{lem}

\begin{proof}The first inequality follows from the Schreier index formula applied to each Sylow subgroup; the second follows immediately from the fact that $d(S) \leq d(S \cap N) + d(S/(S \cap N))$ for every Sylow subgroup $S$ of $G$.

Since $G/N$ is finite, it must be a $\pi$-group for some finite set of primes $\pi$.  Suppose $c_\bP(G)$ is finite.  Then $c_p(N) = c_p(G)$ for $p \in \pi'$, and $d_p(G)$ is finite for $p$ in the remaining finite set of primes $\pi$, which forces $d_p(N)$ and hence $c_p(N)$ to be finite.  If on the other hand $c_\bP(G)$ is infinite, then either $c_p(G)$ is infinite for some $p$, in which case $d_p(N)$ and hence $c_p(N)$ is also infinite, or the invariants $c_p(G)$ take arbitrarily large values as $p$ ranges over the primes, in which case the same is true as $p$ ranges over $\pi'$.  In either case, $c_\bP(N)$ must be infinite.\end{proof}

The following lemma is the main part of the proof of the decomposition theorem above.

\begin{lem}\label{cstarlem}Let $G$ be a prosoluble group with $c_\bP(G) =c$, for some integer $c$.  Let $K$ be the smallest normal subgroup of $G$ such that $G/K$ has exponent dividing $\eb(c)$ and derived length at most $\db(c)$.  Then:
\vspace{-12pt}
\begin{enumerate}[(i)] \itemsep0pt
\item $K' \leq S[G,S]C_G(S)$ for any Sylow subgroup $S$ of $G$;
\item any $p$-Sylow subgroup of $K'$ centralises a $\bP'_n$-Hall subgroup of $G$, for some $n$ depending on $p$ and $G$.\end{enumerate}\end{lem}

\begin{proof}(i) Let $S$ be a $p$-Sylow subgroup of $G$, for some prime $p$.  By the Frattini argument as applied to Sylow's theorem,
\[ G = [G,S]N_G(S) \]
so that $G/(S[G,S]C_G(S))$ is isomorphic to a quotient of $L = N_G(S)/SC_G(S)$.  By coprime action $L$ acts faithfully on $S/\Phi(S)$, which means that $O^{(c,p)^*}(L)=1$.  By Corollary \ref{malcor}, $L$ therefore has an abelian subgroup $M$ such that $L/M$ has exponent at most $\eb(c)$ and derived length at most $\db(c)$, and so the same must be true for $G/(S[G,S]C_G(S))$.

(ii) Since $G/O_{p'}(G)$ is virtually pro-$p$, it is pro-$\bP_n$ for some $n$, and for all primes $q > n$, some $q$-Sylow subgroup $T$ of $G$ is contained in $O_{p'}(G)$.  By part (i)
\[ K' \leq T[G,T]C_G(T) \leq O_{p'}(G)C_G(T) \]
from which it follows that a $p$-Sylow subgroup of $K'$ is contained in $C_G(T)$.  Hence by Sylow's theorem, any $p$-Sylow subgroup of $K'$ centralises a $q$-Sylow subgroup of $G$ for all $q > n$, and hence by Corollary \ref{indexcor}, any $p$-Sylow subgroup of $K'$ centralises a $\bP'_n$-Hall subgroup of $G$.\end{proof}

\begin{proof}[Proof of Theorem]Since $c_\bP(G)$ is finite, it follows that $d_2(G)$ is finite, so $|G:O_\sol(G)|$ is finite.  Lemma \ref{fistarlem} ensures that $c_\bP(O_\sol(G))$ is finite.  Now applying Lemma \ref{cstarlem} to $O_\sol(G)$, we see that any $p$-Sylow subgroup of $K'$ centralises a $\bP'_n$-Hall subgroup of $O_\sol(G)$, for some $n$ depending on $p$ and $O_\sol(G)$.

It remains to show that $N$ is Sylow-finite.  Let $S$ be a $p$-Sylow subgroup of $K'$, so that $S$ centralises a $\bP'_n$-Hall subgroup $L$ of $O_\sol(G)$ for some $n$; also $S$ is contained in a $\bP_n \cup \{p\}$-Hall subgroup $M$ of $O_\sol(G)$.  Now $M$ is virtually pronilpotent by Corollary \ref{tatenilp}, so $O_p(M)$ has finite index in $S$; also $O_p(M) \unlhd \langle L,M \rangle = O_\sol(G)$, so $O_p(M) \leq F(O_\sol(G))=F(G)$.  This establishes that all Sylow subgroups of $N$ are finite.  But this means that a $\bP'_n$-Hall subgroup of $N$ has finite index, for any integer $n$, and so every Sylow subgroup of $N$ also has a centraliser of finite index.\end{proof}

There are several easy consequences of this decomposition:

\begin{cor}Let $M$ be a profinite group with $c_\bP(M)$ finite, and let $G \leq M$.
\vspace{-12pt}
\begin{enumerate}[(i)] \itemsep0pt
\item There is a series
\[ 1 \leq F(G) \leq K'F(G) \leq K \leq G \]
of normal subgroups of $G$, such that $K'F(G)/F(G)$ is Sylow-finite and $G/K$ has order bounded by properties of $M$.  In particular $G \in \FR$.
\item Every Sylow subgroup of $G/F(G)$ is abelian-by-finite, and in particular has finite rank and derived length.  For all but finitely many primes $p$, all $p$-Sylow subgroups of $G/F(G)$ are finite-by-abelian, and in particular have finite nilpotency class.
\item Suppose in addition that $VZ(G/F(G))$ is finite.  Then $G/F(G)$ is finite.\end{enumerate}\end{cor}

\begin{proof}(i) Theorem \ref{cstarthm} gives the required decomposition for $M$ itself; this is clearly inherited by $G$.  The classes of pronilpotent groups, Sylow-finite groups, abelian profinite groups and finite groups are all contained in $\FR$.  Hence $G \in \FR$, since $\FR$ is closed under extensions.

(ii) By the decomposition, $G/F(G)$ is (Sylow-finite)-by-abelian-by-(finite $\pi$-group) for some finite set of primes $\pi$.  Hence every Sylow subgroup is finite-by-abelian-by-finite, and so abelian-by-finite, and every $p$-Sylow subgroup for $p \in \pi'$ is finite-by-abelian.  Every Sylow subgroup of $G/F(G)$ is finitely generated, since $d_p(G)$ is finite for all $p$; the remaining assertions are now clear.

(iii) Let $K$ be as in the decomposition, and let $L = K'F(G)/F(G)$.  Then $\Fin(L) \subseteq VZ(G/F(G))$, and hence $\Fin(L)$ is finite.  But $\Fin(L)$ is dense in $L$, so $L$ is finite.  Hence $G/F(G)$ is finite-by-abelian-by-finite, and hence virtually abelian; in other words, $G/F(G)$ has an abelian open normal subgroup $N$.  But then $N \leq VZ(G/F(G))$, so $N$ is finite.  Hence $G/F(G)$ is finite.\end{proof}

Using Proposition \ref{nilpresFR} and Corollary \ref{malcor}, we also obtain the following:

\begin{prop}Let $G$ be a prosoluble group with $c_\bP(G)=c$ finite and let $N = O^\nil(G)$.  Then $G$ has a subgroup $H$ such that $NH = G$, and such that for every Sylow subgroup $R$ of $H$, we have $A' \leq F(G)$, where $A=R^{\eb(c)}R^{(\db(c))}$.

In particular, $G/NF(G)$ has a characteristic abelian subgroup $K/NF(G)$ such that $G/K$ has exponent at most $\eb(c)$ and derived length at most $\db(c)$.\end{prop}

\begin{proof}Let $R$ and $p$ be as in Proposition \ref{nilpresFR}.  Then $R/(R \cap F(G) \lesssim \prod_{q \in p'} \Delta(S_q)$, and so by coprime action, $O^{(c,p')}(R)=1$.  Hence $A' \leq F(G)$ by Corollary \ref{malcor}.
  
Now by Hall's theorem, we can make suitable choices of subgroups $R_q$ of the form of $R$ above for each prime $q$, conjugating if necessary, so that the set $\mcR = \{ R_q \mid q \in \bP \}$ is pairwise permutable, and thus forms a Sylow basis for the subgroup $H$ generated by $\mcR$.  Then $NH=G$ by Proposition \ref{nilpresFR}.  The assertion about $G/NF(G)$ follows by the fact that $G/NF(G)$ is pronilpotent, and hence isomorphic to the Cartesian product of its Sylow subgroups.\end{proof}

\section{Profinite groups involving finitely many primes}

In this section we consider profinite groups involving finitely many primes; these can equivalently be referred to as pro-$\bP_n$ groups, where $n$ is the largest prime involved.  If a pro-$\bP_n$ group additionally has all Sylow subgroups finitely generated, then it is virtually pronilpotent, as shown in Section \ref{csqtate}.  In fact, something stronger is true:

\begin{thm}\label{lambdanthm}Let $G$ be a pro-$\bP_n$ group, such that $d_\bP(G)=d$.  Then $|G:F(G)|$ is bounded by a function of $d$ and $n$.\end{thm}

\begin{proof}By Corollary \ref{tatefin}, $G/O_\sol(G)$ is finite; our first aim is to bound its order by a function of $d$ and $n$.  We have $F(G/O_\sol(G))=1$, and so $G/O_\sol(G)$ acts faithfully on $E(G/O_\sol(G))$, which is in turn a direct product of non-abelian finite simple $\bP_n$-groups.  By Theorem \ref{compbound}, the number $e$ of components of $G/O_\sol(G)$ is bounded by a function of $d$.  Since there are only finitely many simple $\bP_n$-groups, it follows that there are only finitely many possibilities for $E(G/O_\sol(G))$, and hence also for $G/O_\sol(G)$.  From now on, we may assume that $G$ is prosoluble, as $d_\bP(O_\sol(G)) \leq |G/O_\sol(G)|d_\bP(G)$.

Define the sequence $G_i$ of subgroups inductively as follows: $G_0=G$, and thereafter $G_{i+1}=O^\nil(G_i)$.  Note that $H=O^\nil(H)$ implies $H=1$ for any prosoluble group $H$.  We now make a series of claims that will lead to the conclusion of the theorem.

\emph{(i) For any prime $p$, the $p$-Sylow subgroup of $G/G_1F(G)$ is isomorphic to a subgroup of $\prod_{q \in \bP_n \setminus \{p\}} \Delta(S_q)$.  Hence $|G:G_1F(G)|$ is bounded by a function of $d$ and $n$.}

This is just a special case of Proposition \ref{nilpresFR}.  The second assertion follows from the fact that $\Delta(S_q) \lesssim \GL(d,q)$.
  
\emph{(ii)Let $K_i = G_iF(G)$; then $|G:K_i|$ is bounded by a function of $(d,n,i)$.}

Suppose $|G:K_i|$ is bounded by a function of $(d,n,i)$ for some integer $i$.  Then $d_\bP(K_i)$ is bounded by a function of $|G:K_i|$ and $d$, and hence by a function of $(d,n,i)$.  By claim (i),  $|K_i:O^\nil(K_i)F(K_i)|$ is bounded by a function of $d_\bP(K_i)$ and $n$, and hence by a function of $(d,n,i)$.  Moreover, $F(K_i) = F(G)$, and $K_{i+1}$ contains $O^\nil(K_i)$ by the fact that $K_i/K_{i+1}$ is pronilpotent.  Hence $|G:K_i|$ is bounded by a function of $(d,n,i)$.  The claim follows by induction.

\emph{(iii) Let $i \geq 0$, and suppose $(G_i \cap S_p)\Phi(S_p) = (G_{i+2} \cap S_p)\Phi(S_p)$ for all $\piP$.  Then $G_{i+1}=1$.}

It follows from Corollary \ref{tatecor} that $G_i/G_{i+2}$ is $p'$-normal for all $p$, so $G_i/G_{i+2}$ is pronilpotent; hence $G_{i+1} = O^\nil(G_i) \leq G_{i+2} = O^\nil(G_{i+1})$, so $G_{i+1}=1$.

\emph{(iv) We have $G_t = 1$, where $t = 2d|\bP_n|+1$.  In particular, $|G:F(G)|$ is bounded by a function of $d$ and $n$ by claim (ii).}

Let $r(i) = \sum_{p \in \bP_n} \log_p|(G_i \cap S_p)\Phi(S_p)/\Phi(S_p)|$.  Then $r(0)$ is at most $d|\bP_n|$, and $r(i) \geq r(i+1)$ for all $i$.  Suppose $r(i) = r(i+2)$ for some $i$.  Then $(G_i \cap S_p)\Phi(S_p) = (G_{i+2} \cap S_p)\Phi(S_p)$ for all $\piP$, and so $G_{i+1} = 1$ by claim (iii).  It follows that the sequence $r(2i)$ is strictly decreasing until $r(2u)=0$ for some $u \leq r(0)$, at which point $G_{2u+1} =1$.\end{proof}

\begin{rem}In the case of prosoluble groups at least, the proof of the above theorem could be followed carefully to give an explicit bound on $|G:F(G)|$ as a function of $d$ and $n$.  However, it seems likely that this bound grows too rapidly to be of much interest, and that considerably better bounds are available.\end{rem}

\section{Profinite groups of finite rank}

Let $G$ be a profinite group, and write $r_\bP(G)$ for the supremum of $r(G)$, as $r$ ranges over all Sylow subgroups of $G$.  Thanks to the following theorem, the rank of $G$ itself is almost completely determined by $r_\bP(G)$:

\begin{thm}[Guralnick \cite{Gur}]Let $G$ be a finite group.  Then $d(G) \leq d_\bP(G) + 1$.\end{thm}

\begin{cor}Let $G$ be a profinite group.  Then $r_\bP(G) \leq r(G) \leq r_\bP(G) + 1$.\end{cor}

Section 8.4 of \cite{Wil} contains the following theorem:

\begin{thm}Let $G$ be a profinite group of finite rank.  Then $G$ has normal subgroups $C \leq N \leq G$ such that $C$ is pronilpotent, $N/C$ is soluble and $G/N$ is finite.\end{thm}

We now give a more detailed decomposition theorem, using the generalised pro-Fitting subgroup, and the decomposition obtained for groups with finite $c_\bP$.

\begin{thm}Let $G$ be a profinite group, with $r_\bP(G) = r$ finite.  Then $G$ has a series
 \[ F(G) \leq H \leq O_\sol(G) \leq E \leq G \]
of characteristic subgroups, such that:
\vspace{-12pt}
\begin{enumerate}[(i)] \itemsep0pt
\item $G/E$ and $O_\sol(G)/H$ both have order bounded by a function of $r$;
\item $H/F(G)$ is abelian;
\item $E/O_\sol(G)$ is a direct product of non-abelian finite simple groups $Q_1, \dots Q_n$, where $n \leq r/2$ and $\sum^n_{i=1}r_\bP(Q_i) \leq r$.\end{enumerate}\end{thm}

\begin{proof}Let $L = O_\sol(G)$.  Consider first the quotient $K = G/L$.  This is finite by Corollary \ref{tatefin}, as $d_2(G)$ is finite.  Moreover, $F(K)=1$, and so $K$ acts faithfully on its layer $E(K)$, which leads to (iii) by the fact that $r_\bP(K) \leq r$.  Let $N= \bigcap \{ N_K(Q) \mid Q \in \Comp(K)\}$.  Now $K/N \leq \Sym(n)$ and $N/E(K) \leq \prod^n_{i=1} \Out(Q_i)$; hence $K/N$ and $N/E(K)$ both have order bounded by a function of $r$ by (iii) and by Proposition \ref{qsout} respectively.  Hence $G/E$ has order bounded by a function of $r$.

By Theorem \ref{cstarthm}, $G \in \FR$, and $F(G/\Phi(F(G))) = F(G)/\Phi(F(G))$ by Corollary \ref{tatecor}.  Hence $L/F(G)$ acts faithfully by conjugation on $F(G)/\Phi(F(G))$, so we may assume $\Phi(F(G))=1$.  For each prime $p$, this means that $O_p(G)$ is elementary abelian of rank at most $r$, so $\Aut(O_p(G)) \leq \GL(r,p)$.  Let $H = F(G)L^{\eb(r)}L^{(\db(r))}$.  Then $H$ is characteristic in $G$, and it follows by Corollary \ref{malcor} that $H'$ acts on $O_p(G)$ as a $p$-group for every $p$, thus $H' \leq F(G)$ as required for (ii).  Finally, note that the factors of the derived series of $L/H$ all have exponent dividing $\eb(r)$ and rank at most $r$, and hence order at most $\eb(r)^r$, and that there are at most $\db(r)$ such factors; thus $|L/H| \leq \eb(r)^{r\db(r)}$, completing the proof of (i).\end{proof}

\begin{cor}Let $G$ be a profinite group.  Then $G$ has finite rank if and only if it has normal subgroups $N \leq A \leq G$, such that $N$ is pronilpotent and of finite rank, $A/N$ is finitely generated abelian, and $G/A$ is finite.\end{cor}

\begin{proof}By the theorem, any profinite group of finite rank admits such a decomposition.  Conversely, if $G$ has such a series of normal subgroups, then the rank of $G$ is at most $r(N) + r(A/N) + r(G/A)$, and all three terms of this sum are clearly finite.\end{proof}

\begin{rem}Profinite groups of finite rank need not be virtually pronilpotent, as demonstrated by the following construction.  Let $p$ be a prime, and let $q_1, q_2, \dots$ be a sequence of distinct primes such that $p^i$ divides $(q_i - 1)$ for all $i$.  Now let $C$ be the Cartesian product of cyclic groups of order $q_i$, one for each $i$.  Then $C$ admits a faithful action of $\bZ_p$, and so there is a profinite group of the form $C \rtimes \bZ_p$ of rank $2$ that is not virtually pronilpotent (since $F(C \rtimes \bZ_p) = C$).  The index of the prosoluble radical cannot be bounded by a function of the rank alone, as for example the rank of the finite simple group $\PSL(2,p)$ (where $p \geq 5$) is independent of $p$, whereas $|\PSL(2,p)|$ tends to infinity as $p$ tends to infinity.\end{rem}

\chapter{Virtually pro-$p$ groups with a specified $p$-Sylow subgroup}

\section{Fusion and $p$-local maps in finite groups}

In finite groups, the theory of fusion is a bridge between the theory of finite $p$-groups, and that of finite groups in general.  Typically, it is assumed that we understand something about the structure of a Sylow $p$-subgroup $S$ of a finite group $G$, and wish to apply this knowledge to give `global' information about the structure of $G$ itself.  This is done by studying subgroups of $S$ and the actions induced on them by their normalisers in $G$.  (This is sometimes called `local analysis' by finite group theorists, but in the present context the term `local' could be confused with the infinite group theorist's largely unrelated notion of `local subgroups'.  Hence the use of the alternative term `fusion' in this document.)  The theory of fusion in finite groups is old (arguably dating back to Sylow's theorems in 1872) and well-developed, and in particular played a large role in the classification of finite simple groups.

In principle, exactly the same approach can be applied to profinite groups as well, since a version of Sylow's theorem still applies.  However, fusion theory is much less developed for profinite groups than for finite groups, and the published literature on the subject is quite limited.  As far as the author is aware, the first significant foray into this area was a paper by Gilotti, Ribes and Serena (\cite{GRS}); since then, fusion and fusion systems in a profinite context have also featured in the work of Peter Symonds (see for instance \cite{Sym}).

Given a finite group $G$, subgroups $H$ and $K$, and a homomorphism $\phi$ from $H$ to $K$, we say $\phi$ is induced by $G$ if there is an element $g \in G$ for which $h^g = h^\phi$ for all $h \in H$; write $\Hom_G(H,K)$ for the set of all such homomorphisms between $H$ and $K$.  The following can be considered the `$p$-local' problem of fusion in $G$:

\paragraph{Problem} Describe $\Hom_G(P,Q)$ for all pairs of $p$-subgroups $(P,Q)$ of $G$.

More precisely, we wish to know:
\vspace{-12pt}
\begin{enumerate}[(i)] \itemsep0pt
\item a set $\mcP$ of representatives for the conjugacy classes of $p$-subgroups of $G$;
\item for each $P \in \mcP$, a description of the action of $N_G(P)$ on $P$;
\item given $(P,Q) \in \mcP \times \mcP$, an element of $\Hom_G(P,Q)$ (if one exists).\end{enumerate}
As a result of Sylow's theorem, all the representatives can be chosen to be subgroups of a single $p$-Sylow subgroup $S$ of $G$, and for (iii), it suffices to consider the case $Q=S$.  We can therefore tackle the problem by the following approach:
\vspace{-12pt}
\begin{enumerate}[(a)] \itemsep0pt
\item Find a $p$-Sylow subgroup $S$ of $G$, and obtain a set of representatives $S_i$ for the conjugacy classes of subgroups of $S$, together with the sets $I_i = \Hom_S(S_i,S)$.
\item Determine which $S_i$ are `fused', that is conjugate, in $G$, and given any pair $(i,j)$ such that $S_i$ is conjugate to $S_j$ in $G$, choose an isomorphism $\phi_{ij}$ from $S_i$ to $S_j$ induced by $G$.
\item Choose one representative $P_i$ for each conjugacy class of $p$-subgroups of $G$ from among the representatives $S_j$ that are contained in it, chosen so that $N_S(P_i)$ is a $p$-Sylow subgroup of $N_G(P_i)$ (this is always possible, by Sylow's theorem).
\item For each $P_i$, find the group $A_i$ of automorphisms of $P_i$ induced by $N_G(P_i)$.\end{enumerate}

Every homomorphism from $P_i$ to $S$ induced by $G$ is now obtained as an element of $A_i$, followed by an isomorphism $\phi_{ij}$, followed by an element of $I_j$.  Furthermore, this decomposition is unique.  The subgroups $S_i$ and the sets of homomorphisms $I_i$ can be regarded as purely internal to $S$, with no influence from the rest of $G$, whereas the $A_i$ and $\phi_{ij}$ encode information about the action of $G$ as a whole on its $p$-subgroups.

We refer to automorphisms on subgroups of $S$ induced by $G$ as \emph{$p$-local automorphisms of $G$ on $S$}.  More generally, given any automorphism $\theta$ on a subgroup $P$ of $G$, and any subgroup $Q$ of $P$, there is a restriction of $\theta$ to an isomorphism $\theta_Q$ from $Q$ to another subgroup $Q^\theta$ of $P$.  We refer to such an isomorphism as a \emph{$p$-local map of $G$ on $S$} if it is formed by restricting a $p$-local automorphism of $G$ on $S$.  The importance of $p$-local maps is shown by Alperin's Fusion Theorem:

\begin{thm}[Alperin \cite{Alp}]\label{alpfus}Let $G$ be a finite group, let $S \in \Sylp(G)$, let $P$ and $Q$ be subgroups of $S$, and let $\phi \in \Hom_G(P,Q)$.  Then $\phi$ can be written as a composition $\psi_1 \dots \psi_n$ such that each $\psi_i$ is a $p$-local map of $G$ on $S$.\end{thm}

(In fact, Alperin proves a stronger result, but the version above will suffice for this discussion.) 

Thus, to obtain the maps $\phi_{ij}$, and to understand the way in which $G$ interacts with its $p$-subgroups, it generally suffices to understand the $p$-local automorphism groups $A_i$ of $G$.

Now consider a profinite group $G$.  This time, we wish to know about homomorphisms between pro-$p$ subgroups of $G$.  Since Sylow's theorem applies, we can apply much the same approach as before, starting with a $p$-Sylow subgroup $S$ of $G$.  This time, define a $p$-local automorphism of $G$ on $S$ to be an automorphism induced by $G$ on an \emph{open} subgroup $P$ of $S$ such that $N_S(P)$ is a $p$-Sylow subgroup of $N_G(P)$, and $p$-local maps as restrictions of these to isomorphisms between closed subgroups.  Even with such a restriction, we can `approximate' the fusion by compositions of $p$-local maps, in the sense of the following theorem, which is a direct application of Theorem \ref{alpfus} to the finite images of a profinite group.

\begin{thm}Let $\{G_i \mid i \in I \}$ be an inverse system of finite groups, with inverse limit $G$, and set $N_i$ to be the kernel of the projection map from $G$ to $G_i$.  Let $S \in \Sylp(G)$, let $P$ and $Q$ be subgroups of $S$, and let $\phi \in \Hom_G(P,Q)$.  Then there is a set $\{\phi_i:P \rightarrow Q_i \mid i \in I \}$ of homomorphisms from $P$ to subgroups $Q_i$ of $S$, such that, for all $i \in I$:
\vspace{-12pt}
\begin{enumerate}[(i)] \itemsep0pt
\item $Q_i$ is an open subgroup of $S$ satisfying $Q_i N_i = QN_i$;
\item $\phi_i$ is an isomorphism from $P$ to $Q_i$;
\item the isomorphism induced by $\phi_i$ from $PN_i/N_i$ to $QN_i/N_i$ is the same as the map induced by $\phi$ from $PN_i/N_i$ to $QN_i/N_i$;
\item $\phi_i$ is the composition of a finite sequence of $p$-local maps of $G$ on $S$.\end{enumerate}\end{thm}

We regard the $Q_i$ above as successive approximations to $Q$ that converge to $Q$, and the $\phi_i$ as successive approximations to $\phi$ that converge to $\phi$.  If $Q_i$ and $\phi_i$ are specified for all $i$ in $I$, this is enough to determine both $Q$ and $\phi$ uniquely.  We can therefore reformulate our original problem as follows:

\paragraph{Problem} Given a profinite group $G$, with $p$-Sylow subgroup $S$, find the automorphisms induced on open subgroups of $S$ by conjugation in $G$. 

As in the finite case, there is a further reduction of the problem.  Given a normal subgroup $K$ of $G$ that is pro-$p'$, then $K$ plays no part in the $p$-local automorphism groups:

\begin{lem}Let $G$ be a profinite group, with $p$-Sylow subgroup $S$, and $P$ a subgroup of $S$.  Let $\phi: G \rightarrow H$ be a surjective homomorphism, with kernel $K$, such that $K$ is a pro-$p'$-subgroup of $G$.  Let $Q = P^\phi$, let $A$ be the group of automorphisms of $P$ induced by $N_G(P)$, and let $B$ be the group of automorphisms of $Q$ induced by $N_H(Q)$.  Then $\phi$ restricts to an isomorphism $\psi$ from $P$ to $Q$, and the map $\xi: A \rightarrow B$ defined by $\alpha^\xi = \psi^{-1}\alpha \psi$ is an isomorphism.\end{lem}

\begin{proof}Since $P \cap K = 1$, it follows that $\phi$ restricts to an isomorphism from $P$ to $Q$.  Hence also $\phi$ induces a homomorphism from $A$ to $B$, the injectivity of which is immediate from the fact that the induced map from $P$ to $Q$ is injective.  So it suffices to prove that $\xi$ is surjective.

Let $\beta$ be an automorphism of $Q$ induced by conjugation in $N_H(Q)$, in other words $y^\beta = y^h$ for some $h \in N_H(Q)$.  Then $h$ has a preimage $g$ in $G$ that normalises $PK$.  But $K$ is a pro-$p'$ group, so $P$ is a $p$-Sylow subgroup of $PK$, and $P^g$ is another $p$-Sylow subgroup of $PK$.  Now $K$ is a complement to $P^g$ in $PK$, so by Sylow's theorem, there is an element $k$ of $K$ such that $P^{gk} = P$, in other words $gk \in N_G(P)$, so $gk$ induces an element $\alpha$ of $A$.  Now $(gk)^\phi = g^\phi = h$, so $\psi^{-1}\alpha \psi = \beta$.\end{proof}

\section{$p'$-embeddings in profinite groups}

The previous section motivates the following definition:

\begin{defn}Let $S$ be a pro-$p$ group, and $G$ a profinite group.  Say $G$ is a \emph{$p'$-embedding} of $S$ if $S$ is isomorphic to a $p$-Sylow subgroup of $G$, and $O_{p'}(G)=1$.  The $p'$-embeddings of $S$ form a class, which we denote $\Emb(S)$.  Write $\EmbLF(S)$ for the class of $p'$-embeddings $G$ of $S$ for which $E(G)=1$, and call such $p'$-embeddings \emph{layer-free}.\end{defn}

We wish to describe (in some sense) the class $\Emb(S)$, in order to give an account of the possible $p$-fusion of a profinite group with Sylow subgroup isomorphic to $S$.  Given $G \in \Emb(S)$, we will usually assume $S \in \Sylp(G)$.  We will usually specialise to the case where $S$ is finitely generated as a topological group: it is in this situation where the analogy with fusion in finite groups is strongest, and results from finite groups can be employed more easily than in a more general context.  In particular, we have already seen that if $S$ is finitely generated and $G \in \Emb(S)$, then $G$ is virtually pro-$p$.  However, the following basic question remains:

\begin{que}\label{embfin}For which finitely generated pro-$p$ groups $S$ are $\Emb(S)$ and $\EmbLF(S)$ finite?\end{que}

\begin{defn}Let $U$ be a subgroup of the pro-$p$ group $S$.  Say $U$ is \emph{layerable} in $S$ if there is some $G \in \Emb(S)$ such that $S \cap E(G)=U$.  Say $U$ is \emph{eligible} in $S$ if there is some $G \in \Emb(S)$ such that $O_p(G)=U$, and say $U$ is \emph{LF-eligible} if there is some $G \in \EmbLF(S)$ such that $O_p(G)=U$.\end{defn}

\begin{lem}\label{eliglem}Let $S$ be a finitely generated pro-$p$ group.
\vspace{-12pt}
\begin{enumerate}[(i)] \itemsep0pt
\item We have $\Emb(S) \subset \FR$, so if $G \in \Emb(S)$ then $C_G(F^*(G))=Z(O_p(G))$.
\item If $U$ is an eligible subgroup of $S$, then $U \unlhd_o S$.
\item If $U$ is a layerable subgroup of $S$, then $U \unlhd S$ and $|U|$ is finite, but $U$ is not contained in $\Phi(S)$ unless $U=1$.\end{enumerate}\end{lem}

\begin{proof}(i) By Corollary \ref{tatefin}, any $G \in \Emb(S)$ is (pro-$p$)-by-finite.  Since $\FR$ is closed under extensions, this ensures $G \in \FR$, so $C_G(F^*(G)) = Z(F(G))$ by Theorem \ref{frcent}.  But $Z(F(G))=Z(O_p(G))$, since $O_{p'}(G)=1$.

(ii) Let $G \in \Emb(S)$, such that $U = O_p(G)$.  Then $S$ has finite index in $G$, so $\Core_G(S) \unlhd_o G$; hence $\Core_G(S) \unlhd_o S$.  But $\Core_G(S) = O_p(G)$ by Sylow's theorem.
 
(iii) Let $G \in \Emb(S)$, such that $S \cap E(G) = U$.  Then $U \unlhd S$ since $E(G) \unlhd G$.  Since $G$ is virtually pro-$p$, there is some finite set of components $\mcR$ such that $R = \langle \mcR \rangle$ has the same $p'$-order as $E(G)$.  By Theorem \ref{opilayer}, $RZ(E(G))/Z(E(G))$ is a direct factor of $E(G)/Z(E(G))$, so there is a complementary direct factor $W$ of $E(G)/Z(E(G))$ that is a pro-$p$ group.  Since $E(G)$ is perfect, this ensures $W=1$, so $E(G)/Z(E(G))=RZ(E(G))/Z(E(G))$.  Hence $E(G)$ is finite by Theorem \ref{schurmult}.  If $U \leq \Phi(S)$, then $E(G)$ is $p'$-normal by Corollary \ref{tatecor}; but then $H \leq O_{p'}(G) = 1$, so $E(G)$ is a pro-$p$ group.  This forces $E(G)=1$ and hence $U=1$.\end{proof}

If a class $\mcC$ of $p'$-embeddings of a fixed pro-$p$ group $S$ is finite, then clearly there must be a bound on $|G:S|$ for any $G$ in $\mcC$.  In fact, the converse is true as well.  First, we will need some results from the cohomology theory of finite groups.

\begin{thm}\label{htwothm}Let $G$ be a finite group, and let $M$ be an abelian finite group on which $G$ acts.  Given an extension
\[\mcE = \{\xymatrix{1 \ar[r] & M \ar[r]^\alpha & E \ar[r]^\pi & G \ar[r] & 1}\}\]
of $M$ by $G$, obtain $t_\mcE$ as follows: 

Let $\tau$ be any function from $G$ to $E$ such that $\pi \tau = \id_G$.  Let $f: G \times G \rightarrow M$ be the function determined by $\tau(x)\tau(y) = \tau(xy) \alpha(f(x,y))$.  Let $t_\mcE$ be the equivalence class of $f$ modulo $2$-coboundaries.

Then:
\vspace{-12pt}
\begin{enumerate}[(i)]  \itemsep0pt
\item $f$ is a $2$-cocycle, any choice of $\tau$ gives the same $t_\mcE$, and $t_\mcE$ depends only on the equivalence class of the extension $\mcE$;
\item the map $\mcE \mapsto t_\mcE$ defines a bijection from the set of equivalence classes of extensions of $M$ by $G$ to $\rH^2(G,M)$;
\item $\mcE$ splits if and only if $t_\mcE = 0$.\end{enumerate}\end{thm}

\begin{proof}See \cite{Wil}, Lemmas 6.2.1. and 6.2.2.  (In fact, \cite{Wil} gives a proof for profinite groups in the context of profinite cohomology.)\end{proof}

\begin{thm}\label{coprimecohom}Let $M$ be a finite abelian group, and let $G$ be a finite group acting on $M$.  Suppose $H$ is a subgroup of $G$ for which $|G:H|$ is coprime to $|M|$.  Then for $n>0$, the restriction map $\rH^n(G,M) \rightarrow \rH^n(H,M)$ is injective.\end{thm}

\begin{proof}
 See \cite{Eve}, Proposition 4.2.5.
\end{proof}

We are now ready to prove a theorem about the role of Sylow subgroups in the extension theory of profinite groups.

\begin{thm}\label{extnthm}Let $P$ be a finitely generated pro-$p$ group, and let $K$ be a finite group.  Suppose the extensions
\[\xymatrix{1 \ar[r] & P \ar[r] & G \ar[r] & K \ar[r] & 1 }\]
and
\[\xymatrix{1 \ar[r] & P \ar[r] & G^* \ar[r] & K \ar[r] & 1 }\]
admit a common restriction
\[\xymatrix{1 \ar[r] & P \ar[r] & S \ar[r] & T \ar[r] & 1 }\]

where $T$ is a $p$-Sylow subgroup of $K$, and the action of $K$ on $P/\Phi(P)$ is the same in both extensions.

Then the extensions are equivalent, and hence $G \cong G^*$.\end{thm}

\begin{proof}We may regard $P$ as an open subgroup of $S$, and $S$ as a $p$-Sylow subgroup of both $G$ and $G^*$.  Define subgroups $P_i$ of $P$ by $P_1 = P$, and thereafter $P_{i+1} = [P_i,P]P^p_i$.  Then $P_i$ is an open characteristic subgroup of $P$ for all $i$.  Set $G_i = G/P_i$, set $G^*_i = G^*/P_i$, and set $M_i = P_i/P_{i+1}$.  Then for $i \geq 1$, we have extensions $\mcE_i$ and $\mcE^*_i$ of finite groups given by 
\[\mcE_i = \{\xymatrix{1 \ar[r] & M_i \ar[r] & G_{i+1} \ar[r] & G_i \ar[r] & 1 }\}\]
\[\mcE^*_i = \{\xymatrix{1 \ar[r] & M_i \ar[r] & G^*_{i+1} \ar[r] & G^*_i \ar[r] & 1 }\}\]
and by an inverse limit argument, it suffices to prove that these extensions are equivalent for all $i$.  By induction, we may assume that we have an isomorphism $\theta$ between $G_i$ and $G^*_i$; furthermore, the actions of $G_i$ and $G^*_i$ on $P_i/P_{i+1}$ are determined by the action of $K$ on $P_i/P_{i+1}$, which is in turn determined by the action of $K$ on $P/\Phi(P)$, by Theorem \ref{cinvthm}.  Hence $\theta$ induces an isomorphism from $M_i$ as a $G_i$-module to $M_i$ as a $G^*_i$-module.  Now by Theorem \ref{htwothm}, the extensions $\mcE_i$ and $\mcE^*_i$ are both associated in a natural way to elements $t$ and $t^*$ say of $\rH^2(G_i,M_i)$, and the extensions are equivalent if and only if $t=t^*$.  However, both extensions have the common restriction
\[\xymatrix{1 \ar[r] & M_i \ar[r] & S_{i+1} \ar[r] & S_i \ar[r] & 1 },\]
where $S_i = S/P_i$.  This corresponds to the condition that $t^\rho = (t^*)^{\rho}$, where
$\xymatrix{\rH^2(G_i,M_i) \ar[r]^\rho & \rH^2(S_i,M_i)}$ is the natural restriction map.  But $S_i$ is a $p$-Sylow subgroup of $G_i$ and $M_i$ is a $p$-group, so by Theorem \ref{coprimecohom}, $\rho$ is injective.  Hence $t=t^*$ and so $\mcE_i$ and $\mcE^*_i$ are equivalent.\end{proof}

\begin{cor}\label{embsizecor}Let $S$ be a $d$-generated pro-$p$ group.  Let $\Emb(S;n)$ denote the class of those $G \in \Emb(S)$ for which $|G/O_p(G)| \leq n$.  Then $\Emb(S;n)$ is finite, with $|\Emb(S;n)|$ bounded by a function of $(d,n,p)$.\end{cor}

\begin{proof}Fix $P \unlhd S$.  Let $\Emb(S;P;n)$ be the class of those $G \in \Emb(S)$ for which $O_p(G) = P$, and for which $|G/P| \leq n$.  Clearly we only need to consider those $P$ for which $|S/P| \leq n$; since $S$ is finitely generated, the number of possibilities for $P$ is bounded by a function of $(d,n,p)$, so it suffices to consider $|\Emb(S;P;n)|$ for a single $P$.  By the theorem, for each possible isomorphism type $K$ of $G/P$ and each possible action of $K$ on $P/\Phi(P)$, there is at most one extension of $P$ by $K$ that restricts to the natural extension of $P$ by $S/P$; clearly all of $\Emb(S;P;n)$ arises in this way.  Thus $|\Emb(S;P;n)|$ is at most the number of actions of groups of order at most $n$ on $P/\Phi(P)$, which is given by a function of $(d,n,p)$.\end{proof}

In some cases, restricting to layer-free $p'$-embeddings will simplify the analysis, but layer-free $p'$-embeddings also give a good description of $p'$-embeddings in general.  By Corollary \ref{fplayer}, given a finitely generated pro-$p$ group $S$ and $G \in \Emb(S)$, there will be some finite $L$ for which $G/L \in \EmbLF(SL/L)$.

The structure of a layer-free $p'$-embedding of a pro-$p$ group $S$ is constrained by the automorphism groups of the eligible subgroups of $S$.  The proposition below summarises various equivalent conditions for finite LF-eligibility.

\begin{prop}\label{finelg}Let $S$ be a finitely generated pro-$p$ group.  The following are equivalent:
\vspace{-12pt}
\begin{enumerate}[(i)] \itemsep0pt
\item $\EmbLF(S)$ is finite;
\item there is a bound on $|G/O_p(G)|$ for all $G \in \EmbLF(S)$;
\item there is a bound on $|G/O_p(G)|_p$ for all $G \in \EmbLF(S)$;
\item there is a bound on $d(P)$ for all LF-eligible subgroups $P$ of $S$;
\item there is a bound on $c(P)$ for all LF-eligible subgroups $P$ of $S$.\end{enumerate}\end{prop}

\begin{proof}All $p'$-embeddings are virtually pro-$p$, so (i) implies (ii); conversely (ii) implies (i) by Corollary \ref{embsizecor}.  Clearly (ii) implies (iii), and (iii) implies (iv) by the Schreier index formula.  If (iv) holds, then $|G/O_p(G)| \leq |\GL(d(P),p)|$ by Lemma \ref{deltalem}, implying (i).

We have $c(P) \leq d(P)$ for any pro-$p$ group $P$, so (iv) implies (v).  Now assume (v), with a bound of $c$ on $c(P)$, and consider the $p$-group $K = S/O_p(G)$.  Then for all $G \in \EmbLF(S)$, it follows that $K$ is nilpotent of class at most $(c-1)$ and exponent bounded by a function of $c$; also, $K$ is generated by at most $d(S)$ elements.  These three conditions give a bound on $|K|=|G/O_p(G)|_p$ in terms of $c$ and $S$, and hence a bound on $d(P)$ by the Schreier index formula, giving (iv).\end{proof}

\begin{rem}Condition (iv) is automatic if $S$ is a pro-$p$ group of finite rank.\end{rem}

\section{The local ordering of $p'$-embeddings}

Let $G$ be a $p'$-embedding of a finitely generated pro-$p$ group $S$.  It is clear that given any subgroup $H$ of $G$ containing $S$, then $K = H/O_{p'}(H)$ is also a $p'$-embedding of $S$.  If we regard $S$ as a subgroup of $K$ in the obvious way, then $O_p(K)$ contains $O_p(G)$.  Of particular interest is the possibility that this containment could be proper, giving the potential to build up $G$ from $p'$-embeddings in which the $p$-core has smaller index.

\begin{defn}Say $H$ is a \emph{strong $p$-local subgroup} of the $p'$-embedding $G$ of $S$ if $H$ is the normaliser of an open normal subgroup $P$ of $S$.  Note that any given $p'$-embedding $G$ has only finitely many strong $p$-local subgroups, since all of them lie between $G$ and $S$, and $S$ has finite index in $G$.

Given a finitely generated pro-$p$ group $S$, define a relation $\lep$ on $\Emb(S)$ to be the smallest transitive relation on $\Emb(S)$ such that whenever $G \in \Emb(S)$ and $H$ is a strong $p$-local subgroup of $G$, then $H/O_{p'}(H) \lep G$.  This induces the \emph{local ordering} on the isomorphism types in $\Emb(S)$.\end{defn}

We define $\lep$ in this way to ensure transitivity.  However, there is also a useful characterisation of the isomorphism types occurring below a given $G$ in terms of strong $p$-local subgroups of $G$ itself:

\begin{lem}Let $S$ be a finitely generated pro-$p$ group.  Let $G \in \Emb(S)$, and let $H = N_G(P_1) \cap \dots \cap N_G(P_k)$, where each $P_i$ is an open normal subgroup of $S$.  Then $H/O_{p'}(H) \lep G$.  Moreover, every isomorphism type $K$ of profinite group such that $K \lep G$ arises in this way.\end{lem}

\begin{proof}If $L$ is a strong $p$-local subgroup of $G$, then there is a natural embedding of $S$ into $L/O_{p'}(L)$.  As such, we could obtain some $K \lep G$ by taking $K_0 = G$ and $K_{i+1} = N_{K_i}(P_{i+1})/O_{p'}(N_{K_i}(P_{i+1}))$, and then setting $K = K_k$; moreover, any $K \lep G$ can be obtained in such a way by making suitable choices for the $P_i$.  We claim that in fact such a $K$ will be isomorphic to $H/O_{p'}(H)$.

Set $H_i$ to be the intersection of the $N_G(P_j)$ for $j \leq i$; by induction, we may assume $H_{k-1}/O_{p'}(H_{k-1}) \cong K_{k-1}$ and identify these two groups.  Under this identification, $N_{K_{k-1}}(P_k)$ lifts to the normaliser $M$ of $R = P_k O_{p'}(H_{k-1})$ in $H_{k-1}$.  Now $M$ contains $N_{H_{k-1}}(P_k)$, which is precisely $H$.  Since $P_k$ is a $p$-Sylow subgroup of $R$, and $R$ is normal in $M$, Sylow's theorem ensures that $M = RN_M(P_k)$, so
\[ M = O_{p'}(H_{k-1})N_M(P_k) = O_{p'}(H_{k-1})H \]
so $H$ has an image isomorphic to $N_{K_{k-1}}(P_k)$, and $H/O_{p'}(H) \cong K$ as required.\end{proof}

\begin{cor}Let $S$ be a finitely generated pro-$p$ group, and let $G \in \Emb(S)$.  Then there are only finitely many isomorphism types $H \in \Emb(S)$ such that $H \lep G$.\end{cor}

The $p'$-embeddings under a layer-free $p'$-embedding are in fact subgroups of it:

\begin{lem}Let $S$ be a finitely generated pro-$p$ group.  Let $G \in \EmbLF(S)$, and let $H$ be an intersection of strong $p$-local subgroups of $G$.  Then $H \in \EmbLF(S)$.\end{lem}

\begin{proof}Clearly $S \in \Sylp(H)$.  By Lemma \ref{eliglem}, $G/Z(O_p(G))$ acts faithfully on $O_p(G)$, and so $O_{p'}(H)$ acts faithfully on $O_p(G)$; since $O_p(G)$ and $O_{p'}(H)$ are normal subgroups of $H$ with trivial intersection, this ensures $O_{p'}(H)=1$.  By Theorem \ref{opilayer}, $E(H)$ centralises $O_p(G) \leq F(H)$, so $E(H) \leq Z(O_p(G))$; this ensures that $H$ has no components, so $E(H)=1$.\end{proof}

The next few results consider the consequences of Tate's theorem for the structure of $p$-local subgroups.

\begin{prop}\label{locphiprop}Let $G$ be a $p'$-embedding of the finitely generated pro-$p$ group $S$.  Let $R$ be a normal subgroup of $S$, such that $[S,O_p(G)] \leq R \leq \Phi(S)$.  Let $M = N_G(R)$, let $T = O_{p'}(M)$, and let $U = M/RT$.  Then either $O_p(M/T) > O_p(G)T/T$ or $S/O_p(G)$ acts faithfully on $E_p(U)$ (or both).\end{prop}

\begin{proof}Certainly $O_p(M/T) \geq O_p(G)T/T$, so we may assume $O_p(M/T) = O_p(G)T/T$, which means $O_p(U) = O_p(G)/RT$.  Then $O_p(U)$ is central in $ST/RT$, since $[S,O_p(G)] \leq R$.  But by Corollary \ref{tatecor}, $O_{p'}(U) = 1$, so $F^*(U)=O_p(U)E_p(U)$.  Since $F^*(U)$ contains its own centraliser in $U$, we thus have $C_{ST/RT}(E_p(U)) = C_{ST/RT}(F^*(U)) = Z(F^*(U)) = O_p(U)$, giving a faithful action of $S/O_p(G)$ on $E_p(U)$ as required.\end{proof}

\begin{cor}\label{opphi}Let $S$ be a non-trivial finitely generated pro-$p$ group.
\vspace{-12pt}
\begin{enumerate}[(i)] \itemsep0pt
\item Let $G \in \Emb(S)$.  Suppose that there exists some $R \unlhd G$ such that $[S,O_p(G)] \leq R \leq \Phi(S)$.  Then $S/O_p(G)$ acts faithfully on $E_p(G/R)$; in particular, if $E_p(G/R)=1$ then $G$ is $p$-normal.  If $R=O_p(G)$, then $G/R$ acts faithfully on $E_p(G/R)$, and $E_p(G/R)$ is a direct product of non-abelian simple groups.
\item Let $G$ be any $p'$-embedding of $S$.  Then $E^*_p(G) \cap S \not\le \Phi(S)$.\end{enumerate}\end{cor}

\begin{proof}(i) We have $O_{p'}(N_G(R))=1$ and $O_p(N_G(R)) = O_p(G)$; hence $S/O_p(G)$ acts faithfully on $E_p(G/R)$ by Proposition \ref{locphiprop}.  If $R=O_p(G)$, this ensures that the kernel of the action of $G/R$ on $E_p(G/R)$ is a pro-$p'$ group; but $O_p(G/R)O_{p'}(G/R)=1$ by Corollary \ref{tatecor}, so $E_p(G/R) = F^*(G/R)$ and $Z(E_p(G/R))=1$.  Hence $E_p(G/R)$ is a direct product of non-abelian simple groups, on which $G/R$ acts faithfully.

(ii) Let $N = E^*_p(G)$.  We may assume that $O_p(G) \leq \Phi(S)$, since $O_p(G) \leq N$; write $K=G/O_p(G)$.  Then $E_p(G/O_p(G)) \not=1$ by part (i), ensuring that $E_p(G/O_p(G))$ is not a pro-$p$ group, and hence $N$ is also not a pro-$p$ group; since $O_{p'}(N) \leq O_{p'}(G) = 1$, this ensures $N$ is not $p'$-normal.  Hence $N \cap S \not\le \Phi(S)$ by Corollary \ref{tatecor}.\end{proof}

The possibility of a layer appearing in certain sections of $G$ complicates the analysis; however, stronger conclusions can be drawn if $G$ is $p$-separable.

\begin{thm}\label{pseplocal}Let $G$ be a $p$-separable $p'$-embedding of the finitely generated pro-$p$ group $S$.  Let $R$ be a normal subgroup of $S$, such that $[S,O_p(G)] \leq R \leq \Phi(S) \cap O_p(G)$, and let $M = N_G(R)$.  Then either $S=O_p(G)$ or $O_p(M) > O_p(G)$.  Furthermore, either $d(S/O_p(M)) < d(S/O_p(G))$ or $|M/O_p(M)|$ is bounded by a function of $p$ and $d(O_p(G)/R)$ (or both).\end{thm}

\begin{proof}Suppose $O_p(M) = O_p(G)$.  Then $S/O_p(G)$ acts faithfully on $E_p(M/R)$ by Corollary \ref{opphi}.  But $M$ is $p$-separable, so $E_p(M/R)=1$, and hence $S = O_p(G)$.  We may now assume $d(S/O_p(M)) = d(S/O_p(G))$, since otherwise $O_p(M)$ must strictly contain $O_p(G)$ and $d(S/O_p(M)) < d(S/O_p(G))$.  Let $M_1 = M/R$.  Then $O_{p'}(M_1)=1$ by Corollary \ref{tatecor}, so $F^*(M_1)=O_p(M_1)$, since $M_1$ is $p$-separable; furthermore $O_p(M_1) = O_p(M)/R$.

Let $M_2 = M/O_p(G)$, and let $H$ be the lift of $O_{p'}(M_2)$ to $M$.  Then $H/R$ centralises $O_p(M)/O_p(G)$ since $O_p(M_2) \cap O_p(M)/O_p(G) = 1$, while $O_p(M)/R$ contains its own centraliser in $M_1$.  By coprime action, it follows that the kernel of the action of $H/R$ on $O_p(G)/R$ is a pro-$p$ group, namely $O_p(G)/R$ itself, since $O_p(G)/R$ is abelian by the choice of $R$.  Hence $O_p(G)/R$ admits a faithful action of $O_{p'}(M_2)$, and so $O_{p'}(M_2) \lesssim \GL(n,p)$, where $n = d(O_p(G)/R)$.

Let $M_3 = M/O_p(M)$.  Since $d(S/O_p(M)) = d(S/O_p(G))$, we have $O_{p'}(M_3) \cong O_{p'}(M_2)$ by Corollary \ref{tatecor}.  Furthermore, $O_p(M_3) =1$, so $F^*(M_3) \leq O_{p'}(M_3)$.  Hence $F^*(M_3) \lesssim \GL(n,p)$.  Since $F^*(M_3)$ contains its own centraliser in $M_3$, this ensures that $|M_3|$ is bounded by a function of $n$ and $p$ as required.\end{proof}

\begin{cor}\label{lepcor}Let $G$ be a $p$-separable $p'$-embedding of the pro-$p$ group $S$, with $d(S)=d$ finite.
\vspace{-12pt}
\begin{enumerate}[(i)] \itemsep0pt
\item Let $R = O_p(G) \cap \Phi(S)$, let $M = N_G(R)$, and suppose $O_p(G) < S$.  Then $O_p(M) > O_p(G)$.  Furthermore, either $d(S/O_p(M)) < d(S/O_p(G))$ or $|M/O_p(M)|$ is bounded by a function of $p$ and $d(S)$ (or both).
\item Given $P \unlhd_o S$, let $m_P$ be the number of conjugates of $P \cap \Phi(S)$ under the action of $\Aut(P)$.  Suppose that there is some $n$ such that $(m_P)_{p'} \leq n$ for all $P \unlhd_o S$.  Then $|G|_{p'}$ is bounded by $n^d f(d,p)$ for some function $f$ of $d$ and $p$.
\item Let $L = [S,O_p(G)]$ and suppose $L^G \leq \Phi(S)$.  Then $G$ is $p$-normal.\end{enumerate}\end{cor}

\begin{proof}(i) Note that $d(O_p(G)/R) \leq d(S)$.  The conclusion is now a special case of Theorem \ref{pseplocal}.

(ii) Let $G_0 = G$, and thereafter set $G_{i+1} = N_{G_i}(O_p(G_i) \cap \Phi(S))$, repeating until we reach either a $p$-normal $G_i$, or a $G_i$ such that $d(S/O_p(G_i)) = d(S/O_p(G_{i-1}) > 0$.  One of these must happen before we reach $G_d$, so there is a last term $G_j$ say with $j \leq d$.  Now $|G_i:G_{i+1}|$ is at most $n$ for all $i$; furthermore $|G_j|_{p'}$ is bounded by a function of $d$ and $p$ by Theorem \ref{pseplocal}.  Hence $|G|_{p'} \leq n^j |G_j|_{p'} \leq n^d f(d,p)$.

(iii) Suppose $L^G \leq \Phi(S)$.  Then $L^G$ satisfies the conditions of Theorem \ref{pseplocal}, and $O_p(N_G(L^G)) = O_p(G)$ since $L^G$ is normal, so $O_p(G)=S$.\end{proof}

Now let $G$ be a $p'$-embedding with a layer.  It would be useful to obtain a layer-free $p'$-embedding $H$ satisfying $H \lep G$, such that $H$ retains as much as possible of the structure of $G$, so that we can use properties of layer-free embeddings to control the structure of $G$.

\begin{prop}Let $G$ be a $p'$-embedding of the finitely generated pro-$p$ group $S$.  Then there is a subgroup $H$ of $G$ containing both $S$ and $C_G(E(G))$, such that $H/O_{p'}(H) \lep G$ and $E(H/O_{p'}(H)) = 1$.\end{prop}

\begin{proof}Form a descending sequence of subgroups of $G$ inductively as follows.  Start with $G_0 = G$.  Let $E_i$ be the lift of $E(G_i/O_{p'}(G_i))$ to $G_i$, and let $K_i = (S \cap E_i)(O_p(G) \cap G_i)$.  Now set $G_{i+1} = N_{G_i}(K_i)$.  By induction, it is clear that each $G_i$ contains $S$, so in fact $K_i = (S \cap E_i)O_p(G)$.

Let $O$ be the lift of $O_p(G_i/O_{p'}(G_i))$ to $G_i$.  Then $[O_p(G),E_i] \leq [O,E_i] \leq O_{p'}(G_i)$; this means in particular that $E_i$ centralises $O_p(G)$.  By contrast, $C_G(E(G))/Z(O_p(G))$ acts faithfully on $O_p(G)$, since $G/Z(O_p(G))$ acts faithfully on $F^*(G)$.  We now have two normal subgroups $E_i Z(O_p(G))/Z(O_p(G))$ and $C_G(E(G))/Z(O_p(G))$ of $G_i/Z(O_p(G))$, of which one acts faithfully on $O_p(G)$ and the other centralises $O_p(G)$; it follows that these normal subgroups have trivial intersection and hence commute.  In particular, 
\[ [C_G(E(G)),K_i] \leq [C_G(E(G)),E_i]O_p(G) \leq O_p(G) \leq K_i, \]
so $C_G(E(G)) \leq N_G(K_i)$.  Since this holds for all $i$, we have $C_G(E(G)) \leq G_i$ for all $i$.
 
Since $S$ has finite index in $G$, the sequence $G_i$ of subgroups will eventually terminate, that is, $K_i \unlhd G_i$ for some $i$.  Set $H = G_i$ and $M = H/O_{p'}(H)$; note that $H$ contains both $S$ and $C_G(E(G))$.  Suppose $E(M) > 1$, and let $Q/O_{p'}(H) \in \Comp(M)$, with $Z/O_{p'}(H) = Z(Q/O_{p'}(H))$.  Then $S \cap Q$ is not contained in $Z$, which means that $[S \cap Q,Q]O_{p'}(H) \geq Q$, since $Q/O_{p'}(H)$ is quasisimple; as $K_i$ contains $S \cap Q$, this ensures $[K_i,Q]O_{p'}(H) \geq Q$, whereas $Q$ is not contained in $K_iO_{p'}(H)$.  Hence $Q$ does not normalise $K_iO_{p'}(H)$, a contradiction.  Thus $E(M)=1$.  Finally, $H$ is an intersection of strong $p$-local subgroups of $G$ by its construction, and so $M \lep G$.\end{proof}

\begin{rem}Given any profinite group $G$ with finite layer $E(G)$, the index of $C_G(E(G))$ must divide $|\Aut(E(G))|$, which is itself finite.  The subgroup $H$ obtained in the proof of the above proposition is uniquely determined as a subgroup of $G$ by the choice of Sylow subgroup $S$; in particular, its isomorphism type is uniquely determined.\end{rem}

\section{$p'$-embeddings of $\CT_p$-groups}

Recall the concept of control of $p$-transfer, as described in Section \ref{ptransprelim}, and the equivalent definitions arising from Theorem \ref{tate}.

\begin{defn}Define the class $\CT_p$ to consist of those finitely generated pro-$p$ groups $S$ such that $N_G(S)$ controls $p$-transfer in $G$ for any profinite group $G$ that has $S$ as a $p$-Sylow subgroup.\end{defn}

In this section, we will consider the consequences that control of $p$-transfer has for the structure of $p'$-embeddings.  As motivation for why this property might be worth investigating, consider Theorem \ref{yosh} below.

\begin{defn}A finitely generated pro-$p$ group $S$ is \emph{weakly regular} if it has no quotient isomorphic to $C_p \wr C_p$.\end{defn}

\begin{thm}[Yoshida \cite{Yosh} (finite version); Gilotti, Ribes, Serena \cite{GRS} (profinite version)]\label{yosh}Every weakly regular pro-$p$ group is a $\CT_p$-group.\end{thm}

\begin{rem}Note that $S$ is weakly regular if and only if $S/\Phi(\Phi(S))$ is weakly regular, since $\Phi(\Phi(C_p \wr C_p))=1$.  It is shown in \cite{GRS} that a powerful pro-$p$ group is necessarily weakly regular.\end{rem}
 
Our first goal is the following theorem.

\begin{thm}\label{abecentthm}Let $S \in \CT_p$, and let $G \in \EmbLF(S)$.  Let $H=S[G,S]$, and let $P=O_p(G)$.  Then:
\vspace{-12pt}
\begin{enumerate}[(i)] \itemsep0pt
\item any abelian $p'$-subgroup of $G/P$ that is normalised by $H/P$ is centralised by $H/P$;
\item $F(H/P)$ has nilpotency class at most $2$.\end{enumerate}\end{thm}

We begin the proof with a lemma.

\begin{lem}\label{selftranslem}Let $S$ be a pro-$p$ group, let $G \in \EmbLF(S)$, and let $S \leq K \leq G$.  Suppose $S$ controls $p$-transfer in $K$.  Then $S = K$.\end{lem}

\begin{proof}By Corollary \ref{ptranscomp}, $O^p(K)=O_{p'}(K)$ is a complement to $S$ in $K$.  But $O_{p'}(K)$ acts trivially on $O_p(G)$, whereas $O_p(G)$ contains its own centraliser in $G$; thus $O_{p'}(K) = 1$, so $S=K$.\end{proof}

\begin{proof}[Proof of Theorem](i) It suffices to consider abelian $q$-subgroups of $G/P$, where $q \in p'$.  Let $K \leq G$ such that $K'O^{q}(K) \leq P$ and $[K,H] \leq KP$; it is clear that this accounts for all abelian $q$-subgroups of $G/P$ that are normalised by $H/P$.  Then $N_{K/P}(S/P) = C_{K/P}(S/P)$, and $[K/P,S/P] \cap C_{K/P}(S/P) =1$ by part (iii) of Theorem \ref{cinvthm}; it follows that $N_{[K,S]}(S) \leq P$.  Hence $N_M(S) = S$, where $M = S[K,S]$.  Since $S \in \CT_p$, this ensures $S$ controls $p$-transfer in $M$, so $M=S$ by Lemma \ref{selftranslem}.  Thus $[K,S] \leq K \cap S \leq P$.  The same argument shows that $K/P$ commutes with every $p$-Sylow subgroup of $G/P$.  But $H/P$ is generated by these $p$-Sylow subgroups by construction, so $K/P$ is centralised by $H/P$.

(ii) Write $T = F(H/P)$.  Since $H/P$ is finite, $T$ is nilpotent.  Let $c$ be the nilpotency class of $T$, and assume $c > 2$.  Then $\gamma_{c-1}(T)$ is abelian, since $[\gamma_{c-1}(T),\gamma_{c-1}(T)] \leq \gamma_{2c-2}(T)$, and $2c-2 = c + (c-2) > c$; thus $\gamma_{c-1}(T)$ is central.  But then $\gamma_c(T)=1$, contradicting the definition of $c$.\end{proof}

\begin{cor}\label{fnilpcor}Let $S \in \CT_p$, and let $G$ be a prosoluble $p'$-embedding of $S$.  Let $H=S[G,S]$, and let $P=O_p(H)$.  Then either $G$ is $p$-normal, or $F(H/P)$ has nilpotency class exactly $2$.\end{cor}

\begin{proof}By Theorem \ref{abecentthm}, $F(H/P)$ has nilpotency class at most $2$, so we may assume $F(H/P)$ has nilpotency class less than $2$.  This means $F(H/P)$ is abelian, and so by the theorem $F(H/P)=Z(H/P)$.  Now $H/P$ is a finite soluble group, so $F(H/P) \geq C_{H/P}(F(H/P)) = H/P$, so $H/P$ is abelian, which means $S$ is normal in $H$.  By Sylow's theorem, $S$ is the unique $p$-Sylow subgroup of $H$. Since $H$ is generated by its $p$-Sylow subgroups, it follows that $H=S$, so $S \unlhd G$.\end{proof}

We now give an application of Glauberman's $ZJ$-theorem to this context.  Before stating the $ZJ$-theorem, we need some definitions.

\begin{defn}Let $S$ be a finite $p$-group.  The \emph{Thompson subgroup} $J(S)$ of $S$ is the group generated by all abelian subgroups of $S$ of greatest possible order.\end{defn}

\begin{defn}Let $p$ be a prime.  Define $\Qd(p)$ to be the group of $3 \times 3$ matrices of the form

\[ \left( \begin{array}{cc}
A & 0 \\
u & 1 \end{array} \right)\]

where $A \in \SL(2,p)$, $0$ denotes a zero column vector and $u$ is any row vector of length $2$ over $\bF_p$.\end{defn}

\begin{thm}[Glauberman \cite{Gla}]\label{zjthm}Let $p$ be an odd prime, let $G$ be a $\Qd(p)$-free finite group, and let $S$ be a $p$-Sylow subgroup of $G$.  Suppose that $C_G(O_p(G)) \leq O_p(G)$.  Then $Z(J(S))$ is a characteristic subgroup of $G$.\end{thm}

Note that if $p > 3$, then $\SL(2,p)$ involves a non-abelian finite simple group of order divisible by $p$, so all $p$-separable groups are $\Qd(p)$-free.  In addition, all pro-$2'$ groups are $\Qd(p)$-free for every $p$, since $\SL(2,p)$ has even order for every $p$.

Given a prosoluble $p'$-embedding $G$ of a $\CT_p$-group, we can now apply the $ZJ$-theorem to give a further restriction on the structure of $G$.

\begin{prop}\label{zjctp}Let $S \in \CT_p$, and let $G$ be a prosoluble $p'$-embedding of $S$.  Let $Q$ be a $q$-Sylow subgroup of $G$, where $q$ is coprime to $2p$, such that $Q$ is permutable with $S$.  If $q=3$, suppose also that $G$ is $\Qd(3)$-free.   Let $H=S[S,Q]J(Q)$.  Then there is a $q$-Sylow subgroup $R$ of $H$ such that $J(Q) = J(R)$ and $Z(J(R)) = Z(R)$.\end{prop}

\begin{proof}We may assume $G=SQ$.  Let $P=O_p(G)$, and let $K = Z(J(Q))P/P$.  Then $G/P$ is $\Qd(p)$-free; moreover $F(G/P) = O_q(G/P)$ contains its own centraliser in $G/P$.  It follows by Theorem \ref{zjthm} applied to $G/P$ that $K$ is normal in $G/P$; clearly, $K$ is also abelian.  Hence by part (i) of Theorem \ref{abecentthm}, $K$ is centralised by $S[G,S]/P$, in other words $[Z(J(Q)),S[G,S]] \leq O_p(G)$.  This ensures $[Z(J(Q)),H] \leq O_p(G)$.  By Sylow's theorem, there is some $q$-Sylow subgroup $R$ of $H$ such that $J(Q) \leq R$, so that $J(R) = J(Q)$ and hence $Z(J(R)) = Z(J(Q))$.  Hence $[Z(J(R)),R] \leq R \cap O_p(G) = 1$, so $Z(J(R)) \leq Z(R)$.  But every abelian subgroup of $R$ of largest order contains $Z(R)$, so $Z(J(R)) = Z(R)$.\end{proof}

The last theorem of this section concerns the primes dividing the order of $G$, where $G$ is a prosoluble $p'$-embedding of a $\CT_p$-group.

\begin{thm}Let $S \in \CT_p$ such that $c(S)=c$ and $d(S)=d$, and let $G \in \Emb(S)$.  Suppose that $G$ is prosoluble, and that $G$ is not $p$-normal.  Then there is a prime $q$, such that all of the following conditions are satisfied:
\vspace{-12pt}
\begin{enumerate}[(i)] \itemsep0pt
\item $p \not= q$ and $|G|_q > 1$;
\item $\ord(p,q) \leq m$, where $m = \min \{c,(d-1)\}$;
\item $p \cdot \ord(q,p)$ is even.\end{enumerate}\end{thm}

For the proof, we need another lemma.

\begin{lem}\label{sympauto}Let $q$ be a prime, and let $Q$ be a $q$-group of nilpotency class $2$.  Let $P$ be a $p$-group of automorphisms of $Q$, where $p \not= q$, such that $P$ centralises $Z(Q)$.  Suppose also that $M = Q/Z(Q)$ is irreducible as a $P$-module.  Let $N$ be a maximal subgroup of $Q'$, and identify $Q'/N$ with $\bF_q$.  Then the homomorphism $(-,-)_N$ from $M \times M$ to $Q'/N$ defined by $(xZ(Q),yZ(Q))_N = [x,y]N$ is a non-degenerate, skew-symmetric, alternating bilinear form for $M$ as a vector space over $\bF_q$, and this form is preserved by $P$.  Hence $P$ acts on $M$ as a subgroup of $\Sp(M)$, the symplectic group on $M$ associated to the given form.  In particular, $p \cdot \ord(q,p)$ is even.\end{lem}

\begin{proof}The equation $(xZ(Q),yZ(Q))_1 = [x,y]$ specifies a function $(-,-)_1$ from $M \times M$ to $Q'$.  This is a homomorphism since $M$ is abelian, and hence it is surjective by the definition of $Q'$; hence $(-,-)_N$ is a non-trivial quadratic form.  The form is preserved by $P$ since $P$ centralises $Z(Q)$, which contains $Q'$, and $M$ is irreducible as a $P$-module, so $(-,-)_N$ is non-degenerate on $M$.  Finally, $(-,-)_N$ is also skew-symmetric and alternating by the identities $[x,y] \equiv [y,x]^{-1}$ and $[x,x] \equiv 1$.

We conclude that $P$ acts on $M$ as a subgroup of $\Sp(M)$.  Hence $\Sp(M)$ has a non-trivial irreducible $p$-subgroup.  This implies that one of $p$ and $\ord(q,p)$ is even, by Proposition \ref{classicwr}.\end{proof}

\begin{proof}[Proof of Theorem]Since $G$ is prosoluble, for some $q \not= p$ there must be a $\{p,q\}$-Hall subgroup $H$ such that $H > N_H(S) \geq S$; hence $q$ also divides the order of $S[S,H]$, and $S[S,H]$ is also a $p'$-embedding of $S$ that is not $p$-normal.  Hence we may assume $G$ is a pro-$\{p,q\}$ group, and that $G$ is generated by its $p$-Sylow subgroups.  It now remains to show that $q$ satisfies conditions (ii) and (iii).

By Lemma \ref{selftranslem}, $S$ does not control $p$-transfer in $G$.  Since $S \in \CT_p$, this means $S$ does not control $p$-transfer in $N_G(S)$.  By Corollary \ref{tatecor}, $N_G(S)/S$ must act non-trivially on $S/\Phi(S)$; say the kernel of this action is $M/S$.  In particular $|N_G(S)/M|_q > 1$.  Since $G$ is prosoluble and not $p$-normal, we have $S > O_p(G)\Phi(S) > \Phi(S)$ by Corollary \ref{opphi}, and hence the action of $N_G(S)/S$ is reducible; hence $O^{(m,p)}(N_G(S)/M)=1$ by Lemma \ref{deltalem}.  Condition (ii) now follows by the fact that $|\GL(m,p)|_q > 1$ if and only if $\ord(p,q)$ is at most $m$.

Let $T = F(G/O_p(G))$; then $Z(T)$ is central in $G/O_p(G)$ by Theorem \ref{abecentthm}, and $T$ is nilpotent of class $2$ by Corollary \ref{fnilpcor}.  Let $P = S/O_p(G)$, and consider $T/Z(T)$ as a $P$-module; let $Q/Z(T)$ be a minimal submodule.  Then applying part (ii) of Theorem \ref{abecentthm} to the subgroups of $Q$, we see $Z(Q)=Z(T)$.  We are now in the situation of Lemma \ref{sympauto}, and so $p \cdot \ord(q,p)$ is even.\end{proof}

\begin{eg}Depending on $m$ and $p$, the set $\pi = \pi(m,p)$ of primes $q$ satisfying conditions (ii) and (iii) of the theorem may be surprisingly small.  Suppose $p=3$, and $m \leq 10$.  Then $\pi \subseteq \{2,5,11,41\}$.  So if $S$ is a weakly regular pro-$3$ group generated by at most $11$ elements, and $G$ is a prosoluble $3'$-embedding of $S$, then either $S \unlhd G$, or $G$ involves at least one of the primes in $\{2,5,11,41\}$.  Similarly, if $p=7$ and $m \leq 7$, then $\pi \subseteq \{2,3,5,19\}$.\end{eg}

\section{Normal subgroup conditions and just infinite pro-$p$ groups}

The main aim of this section is to prove the following theorem:

\begin{thm}\label{obphithm}Let $S$ be an infinite finitely generated pro-$p$ group, and let $\mcK$ be the set of open normal subgroups of $S$ that are not contained in $\Phi(S)$.  Suppose $\mcK$ is finite, and that $|S:S^{(n)}|$ is finite for all $n$.  Then $\Emb(S)$ is finite.\end{thm}

Note in particular that $\mcK$ as defined above is finite whenever $S$ is a just infinite pro-$p$ group, by Theorem \ref{genob}.

\begin{lem}\label{obphilem}Let $S$ be a finitely generated pro-$p$ group.  Let $\mcK$ be the set of open normal subgroups of $S$ that are not contained in $\Phi(S)$.  The following are equivalent:\vspace{-12pt}
\begin{enumerate}[(i)] \itemsep0pt
\item $\mcK$ is finite;
\item $|S:\Ob_S(\Phi(S))|$ is finite;
\item $\Phi(S)$ contains every normal subgroup of $S$ of infinite index.\end{enumerate}\end{lem}

\begin{proof}Assume (i).  Then $\Ob_S(\Phi(S))$ is the intersection of finitely many open subgroups of $S$, so is itself open in $S$.

Assume (ii), and let $P$ be a normal subgroup of $S$ not contained in $\Phi(S)$.  Then every open normal subgroup containing $P$ contains $\Ob_S(\Phi(S))$, so $P$ itself contains $\Ob_S(\Phi(S))$; in particular, $P$ is of finite index.  Hence (iii) holds.
  
Suppose $\mcK$ is infinite.  Then $K/\Phi(K)$ is finite for every $K \in \mcK$, so $\mcK$ contains an infinite descending chain $K_1 > K_2 > \dots$ by Lemma \ref{phichain}.  By Lemma \ref{obchain}, the intersection of the $K_i$ is a normal subgroup $L$ say, which is not contained in $\Phi(S)$; but $L$ has infinite index, contradicting (iii).  Hence (iii) implies (i).\end{proof}

\begin{defn}Let $G$ be a $p'$-embedding of the pro-$p$ group $S$.  Say $G$ is a \emph{Frattini} $p'$-embedding if $O_p(G) \le \Phi(S)$.  Otherwise, say $G$ is a \emph{standard} $p'$-embedding.  All $p$-separable $p'$-embeddings are standard, by part (iii) of Corollary \ref{lepcor}.\end{defn}

\begin{lem}\label{phiobemblem}Let $S$ be an infinite pro-$p$ group.  Let $G \in \Emb(S)$.  Suppose that $|S:P| \leq p^t$ for every normal subgroup $P$ of $S$ that is not contained in $\Phi(S)$.

Then $E(G) = 1$, and $|G:E^*_p(G)|_p \leq p^t$.\end{lem}

\begin{proof}By Lemma \ref{eliglem}, $S \cap E(G)$ is a finite normal subgroup of $S$; hence $S \cap E(G) \leq \Phi(S)$, by Lemma \ref{obphilem}.  But this implies $E(G)=1$ by Lemma \ref{eliglem}.  By Corollary \ref{opphi}, $(E^*_p(G) \cap S) \not\le \Phi(S)$.  Hence $|S:E^*_p(G) \cap S| \leq p^t$, and so $|G:E^*_p(G)|_p \leq p^t$.\end{proof}

\begin{proof}[Proof of Theorem \ref{obphithm}]Let $d = d(S)$, let $G \in \Emb(S)$, let $P=O_p(G)$, and let $E=E^*_p(G)/P$; note $|G:E^*_p(G)|_p \leq p^t$ by Lemma \ref{phiobemblem}.  Let $p^t$ be the maximum of $|S:N|$ as $N$ ranges over $\mcK$; note that $t \geq (d-1)$.  By Theorem \ref{extnthm}, it suffices to bound $|G:P|$ in terms of properties of $S$.

If $G$ is a standard $p'$-embedding, then $|S:P| \leq p^t$, and so $d(P) \leq (d-1)p^t + 1$ by the Schreier index formula.  Now $G/P \lesssim \GL(d(P),p)$ since $G \in \EmbLF(S)$, and so $|G:P|$ is bounded by a function of $p$ and $t$.  From now on, we may assume that $G$ is a Frattini $p'$-embedding.  We proceed by a series of claims.

\emph{(i) We have $d_p(E) \leq p^t t + 1$, and hence both $|\Comp(E)|$ and $d_p(Q)$ for $Q \in \Comp(E)$ are at most $p^t(t+1)$.}

Using the Schreier index formula, we obtain the following inequalities:
\[ d_p(E) - 1 \leq p^t(d_p(G/P) - 1) \leq p^t (d(S) - 1) \leq p^t t.\]
In turn, it is clear that both $|\Comp(E)|$ and $d_p(Q)$ are bounded by $d_p(E)$.

\emph{(ii) Let $T$ be a $p$-Sylow subgroup of $E$ contained in $S/P$.  Then the derived length $l$ of $T$ is bounded by a function of $p$ and $t$.}

Since $E$ is a central product of components, it suffices to prove this claim in the case of $E$ quasisimple.  In this case, it follows from Corollary \ref{dpdeg} that $\deg(E)$ is bounded by a function of $d_p(E)$, and hence by a function of $p$ and $t$ by claim (i).  If $E/Z(E)$ is of Lie type, then the claim now follows by Zassenhaus's theorem.  Otherwise, Corollary \ref{dpdeg} ensures that $|E|$ is bounded by a function of $d_p(E)$, which in turn gives a bound on the derived length of $T$.

\emph{(iii) There is a bound on $|S:P|$ in terms of properties of $S$.}

Let $R = S/P$.  Then $|R:T|=|G/P:E|_p$, and by Lemma \ref{phiobemblem} we have $|G/P:E|_p \leq p^t$, so certainly $R^{(t)} \leq T$.  But then $R^{(l+t)} \leq T^{(l)} = 1$, so $S/P$ is soluble of derived length at most $l+t$.  This means that $P$ contains the open subgroup $S^{(l+t)}$ of $S$, so $|S:P|$ is bounded by properties of $S$.

\emph{(iv) There is a bound on $|G:P|$ in terms of properties of $S$.}

We have a bound on $|S:P|$, giving a bound on $d(P)$ in terms of properties of $S$.  But $G$ is layer-free by Lemma \ref{phiobemblem}, so $G/P \lesssim \GL(d(P),p)$.\end{proof}

\begin{cor}Let $S$ be a just infinite pro-$p$ group.  Then $\Emb(S)$ is finite.\end{cor}

\begin{proof}If $S$ is insoluble, the result follows immediately from the theorem.  If $S$ is soluble, then the last non-trivial term in its derived series has finite index, so $S$ is virtually abelian.  In this case $S$ has finite rank and $\Fin(S)=1$ so $\Emb(S) = \EmbLF(S)$; hence $\Emb(S)$ is finite by Proposition \ref{finelg}.\end{proof}

\section{$p'$-embeddings of abelian and $2$-generator pro-$p$ groups}

Let $S$ be a finitely generated pro-$p$ group and let $G \in \Emb(S)$.  We consider first the action of $Z(S)$ on $F^*(G)$.

\begin{prop}Let $G$ be a $p'$-embedding of the finitely generated pro-$p$ group $S$.  Then:
\vspace{-12pt}
\begin{enumerate}[(i)] \itemsep0pt
\item $Z(S)O_p(G)/O_p(G)$ acts faithfully on $E(G)$, but trivially on $\Comp(G)$;
\item $d(Z(S)F^*(G)/F^*(G)) \leq 4|\Comp(G)|$;
\item $G$ has a finite normal subgroup $N$ such that $O_p(G/N)$ contains the centre of a $p$-Sylow subgroup of $G/N$.\end{enumerate}\end{prop}

\begin{proof}(i) By Theorem \ref{opilayer}, $O_p(G)$ centralises $E(G)$, so $Z(S)O_p(G)/O_p(G)$ acts on $E(G)$.  Suppose $s \in Z(S)$ centralises $E(G)$.  Then $s$ centralises $F^*(G)=O_p(G)E(G)$, as $O_p(G) \leq S$.  Hence $s \in Z(F^*(G)) \leq O_p(G)$.  Let $Q \in \Comp(G)$.  Then $Z(S) \leq N_S(Q \cap S)$, and so $Z(S) \leq N_S(Q)$ by Lemma \ref{pscomplem}.

(ii) By part (i), $Z(S)F^*(G)/F^*(G) \lesssim \Out(Q_1) \times \dots \times \Out(Q_n)$, where $\Comp(G) = \{Q_1, \dots, Q_n \}$.  The conclusion follows by Proposition \ref{qsout}.
 
(iii) By induction on the $p'$-order of $G$, it suffices to find a finite normal subgroup $N$ such that $N$ is not a $p$-group and $O_{p'}(G/N)=1$, or to find that $O_p(G)$ already contains $Z(S)$.
 
Suppose $E(G) \not=1$.  Then set $H=G/E(G)$, and choose $K$ such that $KE(G)/E(G)=O_{p'}(H)$.  Then $N = E(G)K$ is finite and not a $p$-group, and $O_{p'}(G/N)=1$.  So we may assume $E(G)=1$.  This means $Z(S) \leq O_p(G)$.\end{proof}

Say a profinite group $G$ is cyclic if $d(G) \leq 1$.  We consider first the $p'$-embeddings of cyclic pro-$p$ groups, and then the $p'$-embeddings of pro-$p$ groups $S$ such that $d(S) \leq 2$.

\begin{prop}Let $S$ be a cyclic pro-$p$ group, and let $G \in \Emb(S)$.  Then one of the following holds:
\vspace{-12pt}
\begin{enumerate}[(i)] \itemsep0pt
\item $S \unlhd G$ and $G/S$ is cyclic of order dividing $p-1$;
\item $G$ has a single component $Q$, such that $S \leq Q$ and $G/Z(Q)$ is almost simple.\end{enumerate}\end{prop}

\begin{proof}Let $P = O_p(G)$.  If $S=P$, then case (i) occurs.  Otherwise, $P \leq \Phi(S)$, so $G/P$ acts faithfully on $E_p(G/P)$ by Corollary \ref{opphi}.  Let $R/P \in \Comp_p(G/P)$; then $R/P \in \fsg$.  Now $R$ is a central extension of $P$ by $R/P$, since $\Aut(P)$ is soluble, so $Q=O^\sol(G)$ is quasisimple.  Since $Q \unlhd G$ but $Q$ is not $p'$-normal, then $S \cap Q \not\le \Phi(S)$ by Corollary \ref{tatecor}, so $S \leq Q$.  Clearly $Q = E_p(G) = E(G)$, and $G/Z(Q)$ is almost simple, since $G/P = G/Z(Q)$ acts faithfully on $Q/Z(Q)$.\end{proof}

Recall Proposition \ref{pspace} and the definition of $f_p$ given afterwards.

\begin{lem}\label{dfplem}Let $S$ be a pro-$p$ group with $d(S)\leq 2$.  Let $l=d(S)-f_p(S)$.  Then $l \in \{0,1,2\}$.

If $l=0$ then $S$ is finite.

If $l=1$, then $S$ is an extension of a finite group by an infinite cyclic group.

If $l=2$ then $S$ has no non-trivial layerable subgroups.\end{lem}

\begin{proof}By Proposition \ref{pspace}, there is a finite normal subgroup $K$ of $S$, such that $K\Phi(S)/\Phi(S)$ has dimension $f_p(S)$, and hence $d(S/K) = l$.

If $d(S/K)=0$, then $S=K$, so $S$ is finite.  Conversely, if $S$ is finite then $d(S/K)=0$.

If $d(S/K)=1$, then $S/K$ is infinite cyclic, since otherwise we would have $S$ finite.

If $d(S/K)=2$, then $\Fin(S) \subseteq \Phi(S)$, so $S$ has no non-trivial layerable subgroups by Lemma \ref{eliglem}.\end{proof}

Normal subgroups not contained in the Frattini subgroup of a $2$-generator pro-$p$ group have consequences for its finite images.

\begin{prop}Let $S$ be a pro-$p$ group with $d(S)=2$, and suppose $P$ is a normal subgroup of $S$ not contained in $\Phi(S)$.  Then $S$ has an image isomorphic to the semidirect product $A \rtimes T$ where $A$ is elementary abelian, $d(A) \geq d(P) - 2$ and $A$ is generated by the conjugates of a single element under the action of $T$, and $T$ is cyclic.\end{prop}

\begin{proof}The conclusion clearly holds if $S=P$, so we may assume $P < S$.  Let $x$ be an element of $P \setminus \Phi(S)$.  Let $K = \langle x \rangle^S$.  Then $S/K$ is cyclic since $K$ is not contained in $\Phi(S)$; hence $S=KR$, where $R$ is a cyclic subgroup of $S$.  Now let $L = \Phi(K)(K\cap R)$, and consider the image $S/L$ of $S$.  This decomposes as a semidirect product $A \rtimes T$, where $A = K/L$ and $T = RL/L$; here $T$ is cyclic, and $A$ is elementary abelian and generated by the conjugates of $xL$ under the action of $T$.  Since $K \cap R$ is cyclic, $d(A) \geq d(K) - 1$.  Finally, $K \leq P$ and $P/K$ is cyclic, so $P$ can be generated by $K$ together with at most one element outside of $K$; hence $d(K) \geq d(P) - 1$, which means $d(A) \geq d(P) - 2$.\end{proof}

Now define an invariant $\wrd(S)$ to be the supremum of $\mathrm{log}_p|A|$, as $A$ ranges over all elementary abelian groups such that $S$ has an image of the form $A \rtimes T$ as specified by the proposition.  In general this may be infinite, for instance if $S$ is the free pro-$p$ group on $2$ generators.

On the other hand, any $A \rtimes T$ as above satisfies $\Phi((A \rtimes T)')=1$, so it must be an image of $S/\Phi(S')$.  This means $\wrd(S)=\wrd(S/\Phi(S'))$.  Thus for $\wrd(S)$ to be finite, it suffices for $S/S'$ to be finite.

\begin{cor}\label{wdcor}Let $G$ be a $p'$-embedding of the $2$-generated pro-$p$ group $S$, and suppose $\wrd(S)$ is finite.  Let $H$ be a normal subgroup of $G$ such that $(H \cap S) \not\le \Phi(S)$.  Then $d_p(H) \leq \wrd(S) + 2$.\end{cor}

We now obtain a list of possible structures for $p'$-embeddings of a $2$-generator pro-$p$ group.

\begin{thm}Let $S$ be a pro-$p$ group such that $d(S) = 2$, and let $G \in \Emb(S)$.  Write $P= O_p(G)$ and $H=G/O_p(G)$.  Let $V=(E^*_p(G) \cap S)\Phi(S)/\Phi(S)$.
  
If $G$ is a standard $p'$-embedding, then exactly one of the following holds:
\vspace{-12pt}
\begin{enumerate}[(i)] \itemsep0pt
\item $S = P$ and $H \lesssim \Delta(S)$;
\item $p$ is odd, $f_p(S) \geq 1$, $E(G)$ is quasisimple, with $|E(G) \cap S|$ of order bounded by properties of $S$, and such that $(E(G) \cap S)Z(Q)/Z(Q)$ is cyclic, $F^*(G) = SE(G)$, and $H \lesssim \Delta(P) \times \Aut(E(G))$;
\item $E(G)=1$, $d(S/P)=1$ and $H \leq \Delta(P) \leq \GL(k,p)$ for some $k$.\end{enumerate}
If instead $G$ is a Frattini $p'$-embedding, then $H \lesssim \Aut(E(H))$, and exactly one of the following holds:
\vspace{-12pt}
\begin{enumerate}[(i)] \itemsep0pt
\setcounter{enumi}{3}
\item $E(H)$ is a non-abelian finite simple group with $d_p(E(H))=2$, and $S \leq E^*_p(G)$;
\item $p$ is odd, $[S,S] \leq P$, $E(H)$ is the direct product of two non-abelian finite simple groups (possibly isomorphic), each with cyclic $p$-Sylow subgroups, and $S \leq E^*_p(G)$;
\item $E(H)$ is the direct product of $p^l$ copies of a single non-abelian finite simple subgroup $Q$ of $H$ for some integer $l$, with $E(H)$ being the $S$-invariant closure of $Q$, and $|V|=p$.\end{enumerate}
 
Let $n$ be the smallest integer such that, whenever $R \in \Comp(H)$ and $R/Z(R)$ is of Lie type, the defining field of $R$ has order at most $n$.  In case (i), $|H|$ divides $(p-1)(p^2-1)$.  In cases (ii), (iv) and (v), $|H|$ is bounded by a function of $p$ and $n$.  If in addition $\wrd(S)$ is finite, we can replace $k$ by $\wrd(S)+2$, and $l$ satisfies $d_p(Q)p^l \leq \wrd(S)+2$.  Hence $|H|$ is now bounded by a function of $\wrd(S)$ and $p$ in case (iii), and in case (vi) it is bounded by a function of $S$ and $n$.\end{thm}

\begin{proof}Suppose $G$ is a standard $p'$-embedding.  Then $d(S/P) < d(S)$, that is $d(S/P) \leq 1$.  If $E(G)=1$, then clearly (i) or (iii) holds according to the value of $d(S/P)$, so we may assume $E(G) > 1$.  This also ensures $d(S/P)=1$ and $f_p(S) \geq 1$.  Lemma \ref{dfplem} now ensures that $\Fin(S)$ is finite, and so the order of $E(G) \cap S$ is bounded by a property of $S$.  Furthermore, $(E(G) \cap S)Z(E(G))/Z(E(G))$ is cyclic, as it is isomorphic to $(E(G) \cap S)P/P$; moreover, $E(G)P/P$ contains a component of $H$, so $(E(G) \cap S)P/P$ is not contained in $\Phi(S/P)$.  This ensures that $E(G)$ consists of a single component $Q$, that $S \leq (Q \cap S)P$, and that $(Q \cap S)Z(Q)/Z(Q)$ is cyclic; this also ensures $p$ is odd, as no non-abelian finite simple group has a cyclic $2$-Sylow subgroup.  Case (ii) now follows.

Suppose now that $G$ is a Frattini $p'$-embedding.  This ensures that $H \lesssim \Aut(E(H))$ by Corollary \ref{opphi}, and the order of $V$ is either $p$ or $p^2$.

Suppose $|V|=p$; then by Corollary \ref{comporb}, there is a single $S/P$-conjugacy class of components of $H$, giving case (vi), so we may assume $|V|=p^2$, which means that $S$ is a subgroup of $E^*_p(G)$.  If $E(H)$ is simple, then $d_p(E(H)) = 2$, and we are in case (iv).  Otherwise, $E(H)$ is decomposable; by Corollary \ref{comporb}, there are at most two $S/P$-conjugacy classes of component of $H$, and hence there are exactly two components of $H$, since $S/P \leq E(H)$, so that $S/P$ normalises every component.  Since $d(S/P) \leq 2$, and $p$ divides the order of each component, each component must have a cyclic $p$-Sylow subgroup; this ensures in turn that $S/P$ is abelian, and that $p$ is odd.  This is case (v).

Now consider bounds on the order of $H$.  In case (i), $H$ is isomorphic to a $p'$-subgroup of $\GL(2,p)$, so has order dividing $(p-1)(p^2-1)$.  In case (ii), let $R = E(G)$; in case (iv), let $R = E(H)$; in case (v), let $R$ be either of the components of $H$.  Then $d_p(R/Z(R))$ is at most $2$, and so $\deg(R)$ is bounded by a function of $p$ by Corollary \ref{dpdeg}, which means $|R|$ is bounded by a function of $p$ and $n$.  In case (ii), $|S:P|$ is at most $|R|$, so $d(P) \leq 2|R|$, giving a bound on $|\Delta(P) \times \Aut(R)|$, and hence on $|H|$, as a function of $p$ and $n$.  In cases (iv) and (v), we obtain a bound on $|E(H)|$ as a function of $p$ and $n$, and hence on $|H|$, since $H \lesssim \Aut(E(H))$.

Now suppose $\wrd(S)$ is finite.  Then Corollary \ref{wdcor} ensures $d_p(P) \leq \wrd(S) + 2$ in case (iii), and $d_p(E(H)) \leq \wrd(S) + 2$ in case (vi).  This gives the required bounds on $k$ and $l$.  This immediately gives a bound on $|H|$ as a function of $\wrd(S)$ and $p$ in case (iii).  In case (vi), there is a bound on both the number of components of $H$ and on $\deg(R)$ for each component $R$ of $H$, and so $|E(H)|$ is bounded by a function of $(\wrd(S),p,n)$, which in turn gives a bound on $|H|$.\end{proof}

\begin{rem}If the $p'$-embedding $G$ is $p$-separable, then only cases (i) and (iii) are possible.\end{rem}

\begin{cor}Let $G$ be a $2'$-embedding of the $2$-generator pro-$2$ group $S$.  Write $P= O_2(G)$, $H=G/O_2(G)$.  Then $G/E^*_2(G)$ is soluble, and in cases (i) and (iii) of the theorem, $G$ is prosoluble.\end{cor}

\begin{proof}Cases (ii) and (v) of the theorem do not apply.  In cases (i) and (iv), $G/E^*_2(G)$ has odd order, and in cases (iii) and (vi) it has a cyclic $2$-Sylow subgroup; thus in all cases $G/E^*_2(G)$ is soluble.  In case (iii), $E^*_2(G)/O_p(G)$ also has a cyclic $2$-Sylow subgroup, so is trivial, and in case (i), $G/O_p(G)$ is isomorphic to a subgroup of the soluble group $\GL(2,2)$.  Hence in cases (i) and (iii), $G/O_p(G)$ is soluble, so that $G$ itself is prosoluble.\end{proof}

\chapter*{Index of Notation}

\addcontentsline{toc}{chapter}{\protect\numberline{}Index of Notation}

$\cdot$ extension (Section 1.1)

$\lesssim$ isomorphic to a subgroup

$\leq_f$ subgroup of finite index

$\lep$ see Definition 6.3.1

$\unlhd^2$ subnormal of defect at most $2$ (Section 1.1)

$\sgp(G)$ the set of all subgroups of $G$

$\nsgp(G)$ the set of all normal subgroups of $G$

$\fsgp(G)$ the set of all subgroups of $G$ of finite index

$\nsgp_\Phi$, $\nsgp^*_\Phi$ see Definition 2.1.3

$\rtimes$ semidirect product

$\wr$ wreath product

$\ri$ see Section 3.1

$[1]$ the class of trivial groups

$\ra$ see Section 3.1

$\Aut(G)$ the group of continuous automorphisms of $G$

$c$ (subscript) closed subset (Section 1.1)

$c(G)$ see Definition 2.5.1

$c^\le(G)$ the supremum of $c(H)$ as $H$ ranges over the open subgroups of $G$

$c_\pi(G)$ the supremum of $c(S)$ for $S$ a $p$-Sylow subgroup of $G$, as $p$ ranges over $\pi$

$C_n$ the cyclic group of order $n$

$C_G(X)$ the centraliser or pointwise stabiliser of $X$ under the action of $G$ (conjugation action unless otherwise indicated)

$\Comm(G)$ the commensurator of $G$ (Definition 2.4.1)

$\Comp(G)$ the set of components of $G$ (Definition 1.4.11)

$\Comp_\pi(G)$ the set of components $Q$ of $G$ such that $p$ divides $|Q|$ for every $p$ in $\pi$ (Definition 1.4.11) 

$\Core_K(H)$ the intersection of the $K$-conjugates of $H$

$\CT_p$ see Definition 6.4.1

$d(G)$ the smallest cardinality of a (topological) generating set for $G$

$d_p(G)$ the smallest cardinality of a (topological) generating set for a $p$-Sylow subgroup of $G$

$d_\pi(G)$ the supremum of $d_p(G)$ as $p$ ranges over the set of primes $\pi$

$\db(n)$ a bound on the derived length of a soluble linear group of degree $n$ (Theorem 1.5.2)

$\deg(G)$ see Definition 1.4.9

$\Delta(G)$ the group of automorphisms of $G/\Phi(G)$ induced by $\Aut(G)$ (Definition 2.5.1)

$E(G)$ the subgroup generated by the components of $G$ (Definition 1.4.11)

$E_\pi(G)$, $E^*_\pi(G)$ see Definition 1.4.11

$\Emb(S)$ the class of $p'$-embeddings of $S$ (Definition 6.2.1)

$\EmbLF(S)$ the class of layer-free $p'$-embeddings of $S$ (Definition 6.2.1)

$\eb(n)$ a bound on the exponent of a class of groups arising from Mal'cev's theorem (Theorem 1.5.5, Corollary 1.5.6)

$F(G)$ the pro-Fitting subgroup of $G$ (Section 1.3)

$F^*(G)$ the generalised pro-Fitting subgroup of $G$ (Definition 1.4.11)

$f_p(G)$ see Proposition 2.2.4 and Corollary 2.2.5

$\bF_q$ the field of $q$ elements

$\FD$ the class of profinite groups $G$ for which $F^*(G)=1$ (Definition 4.1.1)

$\fin$ the class of finite groups

$\Fin(G)$ the union of all finite normal subgroups of $G$

$\FR$ the class of Fitting-regular groups (Definition 4.1.1)

$\hat{G}$ the profinite completion of $G$

$G'$ the derived subgroup of $G$

$G^n$ the group generated by $n$-th powers of elements of $G$

$G^{(n)}$ the $n$-th term of the derived series of $G$

$|G|$ the cardinality of the underlying set of $G$

$|G:H|$ the profinite index of $H$ in $G$ (Definition 1.3.2)

$[G,H]$ the group generated by commutators $[g,h]$ where $g \in G$ and $h \in H$

$G^H$ the group generated by the $H$-conjugates of $G$

$\GL(n,p^e)$ the general linear group of dimension $n$ over the field of $p^e$ elements

$\gamma_n(G)$ the $n$-th term of the lower central series of $G$

$\rh$, $\rhp$ see Section 3.1

$\mathrm{H}^n(G,M)$ the $n$-th cohomology group of $G$ acting on $M$

$\Ilhd_n(G)$ the intersection of all normal subgroups of $G$ of index at most $n$

$J(S)$ the Thompson subgroup of $S$ (Definition 6.4.7)

$j_N(G)$ see Definition 3.7.6

$\KComm(G)$ see Definition 2.4.7

$\LComm(G)$ see Definition 2.4.7

$m_G(H)$ the minimum cardinality of a set of $G$-conjugates of $H$ that generates $H^G$

$\rN$, $\rNa$, $\rNh$, $\rNi$ see Section 3.1

$n_\pi$ the $\pi$-part of the supernatural number $n$ (Definition 1.3.1)

$N_G(H)$ the normaliser of $H$ in $G$

$o$ (subscript) open subset (Section 1.1)

$O^{(n,\pi)}(G)$, $O^{(n,\pi)^*}(G)$ see Definition 1.5.1

$O^p(G)$ the smallest normal subgroup of $G$ such that $G/O^p(G)$ is pro-$p$

$O_\pi(G)$ the $\pi$-core of $G$ (Definition 1.3.8)

$O^\mcX(G)$ the $\mcX$-residual of $G$, where $\mcX$ is a class of groups (Section 1.1)

$O_\mcX(G)$ the $\mcX$-radical of $G$, where $\mcX$ is a class of groups (Section 1.1)

$\ob_G(n)$, $\ob^*_G(n)$: see Definition 3.6.1

$\Ob_G(H)$, $\Ob^*_G(H)$: see Definition 3.5.1

$\ord(n,p)$ the multiplicative order of $n$ as an element of the field of $p$ elements

$\Out(G)$ the group of continuous automorphisms of $G$, modulo inner automorphisms

$p'$ the set of all prime numbers other than $p$

$\bP$ the set of prime numbers

$\bP_n$ the set of prime numbers that are at most $n$

$\bP'_n$ the set of prime numbers greater than $n$

$\nil$ the class of pronilpotent groups

$\sol$ the class of prosoluble groups

$\Phi(G)$ the intersection of all maximal subgroups of $G$ (Definition 1.3.12)

$\Phi^\lhd(G)$ the intersection of all maximal normal subgroups of $G$ (Definition 2.1.1)

$\Phi^{\lhd n}(G)$ see Definition 3.7.2

$\Phi^f(G)$ the intersection of all finite index subgroups of $G$ (Definition 2.1.1)

$\pi'$ the set of prime numbers not contained in the set of primes $\pi$

$\prod$ Cartesian product

$\bQ_p$ the field of $p$-adic numbers

$\Qd(p)$ see Definition 6.4.8

$r(G)$ the supremum of $d(H)$ as $H$ ranges over all subgroups of $G$

$\ir(\theta)$ the index ratio of $\theta$ (Definition 2.4.4)

$s_b(n)$ the sum of the digits of the base $b$ expansion of $n$

$\fsg$ the class of non-abelian finite simple groups

$\SL(n,p^e)$ the special linear group of dimension $n$ over the field of $p^e$ elements

$\Sp(2n,p^e)$ the symplectic group of dimension $2n$ over the field of $p^e$ elements

$\St_G(n)$ the $n$-th level stabiliser (Definition 3.1.2)

$\Sylp(G)$ the set of $p$-Sylow subgroups of $G$

$T_{[n]}$ the subtree induced by vertices of norm at most $n$ (Definition 3.1.2)

$T_v$ the subtree with root $v$ induced by the vertex $v$ and its descendants (Definition 3.1.3)

$VZ(G)$ the union of all finite conjugacy classes of $G$

$\wrd(S)$ see after Proposition 6.6.4

$\overline{X}$ the topological closure of $X$

$\rX$, $\rXa$, $\rXh$, $\rXi$ see Section 3.1

$\bZ_p$ (the additive group of) the $p$-adic integers

$Z(G)$ the centre of $G$

\end{document}